%% file: ms.tex
\documentclass[11pt]{amsart}
\usepackage[utf8]{inputenc}
\usepackage[T1]{fontenc}
\usepackage{epigraph}
\usepackage{amsfonts}
\usepackage{amsmath}
\usepackage{amssymb}
\usepackage{amsthm}
\usepackage{amscd}
\usepackage{bbm}
\usepackage[matrix,arrow,curve]{xy}
\usepackage{hyperref}
\usepackage{tikz-cd}
\usepackage{mathrsfs}
\usepackage{multirow}
\usepackage{stmaryrd}
\usepackage{mathtools}
\usepackage{etoolbox}
\newcommand{\bbk}{\mathbbm{k}}
\newcommand{\bbHom}{\operatorname{\mathbbm{H}om}}
\newcommand{\bbExt}{\operatorname{\mathbbm{E}xt}}
\newcommand{\Ker}{\operatorname{Ker}}
\newcommand{\Coker}{\operatorname{Coker}}
\renewcommand{\Im}{\operatorname{Im}}
\newcommand{\Id}{\operatorname{Id}}
\newcommand{\id}{\operatorname{id}}
\newcommand{\gr}{\operatorname{gr}}
\newcommand{\Gr}{\operatorname{Grass}}
\newcommand{\oub}{\operatorname{fgt}}
\newcommand{\ind}{\operatorname{ind}}
\newcommand{\supp}{\operatorname{supp}}
\newcommand{\rk}{\operatorname{rk}}
\newcommand{\Hom}{\operatorname{Hom}}

\newcommand{\Ext}{\operatorname{Ext}}
\newcommand{\End}{\operatorname{End}}
\newcommand{\Spec}{\operatorname{Spec}}
\newcommand{\Frac}{\operatorname{Frac}}
\newcommand{\Sch}{\mathrm{Sch}}
\newcommand{\Nil}{\underline{\cal Nil}}
\newcommand{\bbN}{\mathbb{N}}
\newcommand{\bbZ}{\mathbb{Z}}

\newcommand{\bbFq}{\mathbb{F}_q}
\newcommand{\bbQ}{\mathbb{Q}}

\newcommand{\bbP}{\mathbb{P}}
\newcommand{\bbC}{\mathbb{C}}
\newcommand{\bbT}{\mathbb{T}}
\newcommand{\ra}{\rightarrow}
\newcommand{\la}{\leftarrow}
\newcommand{\xra}[1]{\xrightarrow{#1}}
\newcommand{\xla}[1]{\xleftarrow{#1}}

\renewcommand{\phi}{\varphi}
\newcommand{\eps}{\varepsilon}
\newcommand{\Ocal}{\ensuremath{\mathcal{O}}}

\renewcommand{\bf}[1]{\mathbf{#1}}
\newcommand{\bb}[1]{\mathbb{#1}}
\newcommand{\cal}[1]{\mathcal{#1}}
\newcommand{\fk}[1]{\ensuremath{\mathfrak{#1}}}
\newcommand{\ona}[1]{\operatorname{#1}}

\theoremstyle{plain}
\newtheorem{thm}{Theorem} [section]
\newtheorem*{thm*}{Theorem}
\newtheorem{lm}[thm]{Lemma}
\newtheorem{prop}[thm]{Proposition}
\newtheorem{corr}[thm]{Corollary}

\newtheorem{conj}[thm]{Conjecture}
\newtheorem*{gupr}{Guiding principle}

\theoremstyle{definition}
\newtheorem{defi}[thm]{Definition}
\newtheorem*{defi*}{Definition}

\theoremstyle{remark}
\newtheorem{exe}[thm]{Example}
\newtheorem{rmq}[thm]{Remark}
\newtheorem*{nota}{Notation}

\AtEndEnvironment{ncorr}{\qed}

\title[CoHAs for Higgs torsion sheaves and moduli of triples]{Cohomological Hall algebras for Higgs torsion sheaves, moduli of triples and sheaves on surfaces}
\author{Alexandre Minets}
\address{Institute of Science and Technology Austria, Klosterneuburg, Austria.}
\email{alexandre.minets@ist.ac.at}

\begin{document}

\begin{abstract}
For any free oriented Borel-Moore homology theory $A$, we construct an associative product on the $A$-theory of the stack of Higgs torsion sheaves over a projective curve $C$. We show that the resulting algebra $A\bf{Ha}_C^0$ admits a natural shuffle presentation, and prove it is faithful when $A$ is replaced with usual Borel-Moore homology groups. We also introduce moduli spaces of stable triples, heavily inspired by Nakajima quiver varieties, whose $A$-theory admits an $A\bf{Ha}_C^0$-action. These triples can be interpreted as certain sheaves on $\bbP(T^*C)$. In particular, we obtain an action of $A\bf{Ha}_C^0$ on the cohomology of Hilbert schemes of points on $T^*C$.
\end{abstract}
\maketitle
\tableofcontents
\input{intro.tex}
\input{coh.tex}
\input{product.tex}
\input{shuffle.tex}
\input{geometry.tex}
\input{modules.tex}
\input{quiver.tex}
\input{neg.tex}
\appendix
\input{cobord.tex}
\bibliography{bib}{}
\bibliographystyle{plain}

\end{document}

%% file: intro.tex
\section{Introduction}
Let $\cal C$ be a hereditary abelian category over finite field $\bbFq$, such that all $\Hom$- and $\Ext$-spaces have finite dimension. We have two important examples of such categories:
\begin{itemize}
	\item for a finite quiver $Q$, the category of finite dimensional representations $\ona{Rep}Q=\mathrm{Rep}_{\bbFq}Q$;
	\item for a smooth projective curve $C$ over $\bbFq$, the category of coherent sheaves $\ona{Coh}C$.
\end{itemize}

Given a category $\cal C$ satisfying the conditions above, one can associate to it the \emph{Hall algebra} $\cal H(\cal C)$, as defined in~\cite{S}.
Broadly speaking, its basis is given by isomorphism classes of objects in $\cal C$, and the product is given by the sum of all non-isomorphic extensions.
In the case $\cal C=\ona{Rep}Q$, where $Q$ is a quiver of Dynkin type, a famous theorem by Ringel~\cite{Ri} describes the Hall algebra $\cal H(\ona{Rep}Q)$ as the positive half of the quantum group $U_\nu(\fk g_Q)$, specialized at $\nu=q^{1/2}$.
Moreover, one can upgrade $\cal H(\cal C)$ to a (twisted, topological) bialgebra, such that the Drinfeld double $D(\cal H(\ona{Rep}Q))$ is isomorphic to the quantum group itself.

By contrast, the Hall algebra $\cal H(\ona{Coh}C)$ seems to be far less understood.
For instance, an explicit description of (the spherical part of) $\cal H(\ona{Coh}C)$ by generators and relations is known only when $C$ is rational~\cite{Ka2} or elliptic~\cite{BS}.
Our principal motivation is to get a better understanding of this algebra.
One way to do it is to study its representation theory.
Unfortunately, since we do not possess an explicit combinatorial description of $\cal H(\ona{Coh}C)$ in terms of generators and relations (see, however,~\cite[Section 4.11]{S} for partial results), we have to construct its representations indirectly.

We use an approach close in spirit to the well-known construction of Nakajima~\cite{Nak94}, which realizes irreducible representations of the universal enveloping algebra $U(\fk g)$ of a simple Lie algebra $\fk g$ as homology groups of certain varieties.
Let us summarize a variant of this construction, following the point of view from~\cite{YZ}.
Namely, for a finite type quiver $Q=(I,E)$ and a projective $\bbC Q$-module $P$ with top of graded dimension $\bf w\in \bbZ_+^I$, one considers the algebraic stack $T^*\ona{Rep}_{\bf v}^{\la P}Q$, where $\ona{Rep}_{\bf v}^{\la P}Q$ parametrizes pairs $(V,\phi)$ with $V\in \ona{Rep}Q$, $\underline{\dim}V=\bf v$, and $\phi\in \Hom_{\bbC Q}(P,V)$.
The $\bbC$-points of this stack can be identified with representations of a quiver $\overline{Q}^\heartsuit$, satisfying certain conditions~\cite[Section 5]{Gi}.
For every dimension vector $\bf v\in \bbZ_+^I$, one introduces a stability condition on these representations, such that subrepresentations of stable representations are stable, and the moduli stack of stable representations forms a smooth variety $\cal M(\bf v,\bf w)$.
Inside these varieties, one has Lagrangian subvarieties $\cal L(\bf v,\bf w):=\pi_{\bf v}^{-1}(0)$, where $\pi_{\bf v}:\cal M(\bf v,\bf w)\ra \Spec\Gamma(\Ocal_{\cal M(\bf v,\bf w)})$ is the affinization map.
Finally, one considers a correspondence $Z_{\bf v}\subset (T^*\ona{Rep}_{\bf v}Q\times M(\bf v_1,\bf w))\times\cal M(\bf v+\bf v_1,\bf w)$, which parametrizes triples $(V,V_1,V_2)$ with $V_2/V_1\simeq V$.
Denoting the projections on the first and second factor by $\Phi_{\bf v}$ and $\Psi_{\bf v}$ correspondingly, we have the following operators in Borel-Moore homology:
\begin{align*}
e_{i,\bf v} & =(\Psi_{\epsilon_i})_*(\Phi_{\epsilon_i})^*([T^*\ona{Rep}_{\epsilon_i}Q]\boxtimes - ):H(\cal M(\bf v,\bf w))\ra \cal M(\bf v+\epsilon_i,\bf w),\\
f_{i,\bf v} & =\left\langle(\Phi_{\epsilon_i})_*(\Psi_{\epsilon_i})^*(-),[T^*\ona{Rep}_{\epsilon_i}Q]\right\rangle:H(\cal M(\bf v,\bf w))\ra \cal M(\bf v-\epsilon_i,\bf w),
\end{align*}
where $\epsilon_i$ is the dimension vector of the simple representation at vertex $i\in I$.
Then $e_i:=\sum_{\bf v}e_{i,\bf v}$, $f_i:=\sum_{\bf v}f_{i,\bf v}$ give rise to an action of $U(\fk g_Q)$ on $M_{\bf w}=\bigoplus_{\bf v}H(\cal M(\bf v,\bf w))$, and moreover its restriction to $\bigoplus_{\bf v}H(\cal L(\bf v,\bf w))$ is the irreducible highest module of weight $\bf w$.

In fact, this action can be extended to a much bigger algebra, so-called Yangian.
This can be achieved by realizing it inside the \emph{cohomological Hall algebra}~\cite{SV2,YZ}, isomorphic to $\bigoplus_{\bf v}H(T^*\ona{Rep}_{\bf v} Q)$ as a vector space (see~\cite{MO} for another perspective on Yangians).
The latter algebra then acts on $M_{\bf w}$ by correspondences similar to the ones described above.
The purpose of this paper is to begin investigation of analogous algebras and their representations in the context of curves.

In order to apply the same set of ideas to our situation, we have to introduce several modifications to our context.
First, we have to consider $T^*\mathop{\underline{\cal Coh}} C$ instead of $\mathop{\underline{\cal Coh}} C$; note that the former stack is isomorphic to the stack of Higgs sheaves $\mathop{\underline{\cal Higgs}} C$.
Secondly, we will study a homological version of Hall algebra.
It will be modeled on the vector space $A(\mathop{\underline{\cal Higgs}} C)$, where $A$ is either Borel-Moore homology or an arbitrary free oriented Borel-Moore homology theory (see~\cite[Chapter 5]{LM} for the definition of the latter).

Optimistically, our program is as follows:
\begin{enumerate}
	\item construct a (bi-)algebra structure $A\bf{Ha}_C$ on $A(\mathop{\underline{\cal Higgs}} C)$;
	\item define a suitable stability condition on $T^*\underline{\cal Coh}^{\la \cal F}C$, where $\underline{\cal Coh}^{\la \cal F}C$ is the stack of pairs $(\cal E,\alpha)$ with $\cal E\in\ona{Coh}C$, $\alpha\in\Hom(\cal F,\cal E)$;
	\item construct an action of the Drinfeld double $D(A\bf{Ha}_C)$ on the $A$-theory $A(\cal M)$ of the moduli of stable objects.
\end{enumerate}

In the present article, we treat a very particular case of the plan above.
Namely, we restrict our attention to the category of \emph{torsion} sheaves on $C$.
Then, we have the following result:
\begin{thm}
There exists an associative product on $\bigoplus_d A(\mathop{\underline{\cal Higgs}}_{d}^{0} C)$, which makes it into an algebra $A\bf{Ha}_{0,C}$ (Theorem~\ref{ass}).
\end{thm}

The proof uses the techniques found in~\cite{SV1,YZ}.
Because of our restrictions on the rank of sheaves, all stacks we consider can be explicitly realized as global quotients, and thus we can forget their stacky nature and work with equivariant $A$-theory of their atlases instead.
In positive rank the stack $\underline{\cal Coh}_{r,d}$ is only \emph{locally} a quotient stack, so that one has to check that separate constructions in each patch can be glued together.
This was done in~\cite{SS}.

Note that we do \emph{not} construct a coproduct on $A\bf{Ha}_{C}^{0}$.
However, if we denote by $A\bf{Ha}^{0,T}_{C}$ the version of $A\bf{Ha}_{C}^{0}$ equivariant with respect to the scaling action of $\bb G_m$ on the cotangent fibers, one can define a certain algebra $A\bf{Sh}_C$ with explicit formulas for multiplication and construct a map $\rho:A\bf{Ha}^{0,T}_{C}\ra A\bf{Sh}_C$.
Roughly speaking, $A\bf{Sh}_C$ looks like the space of formal series with coefficients in $A(C)$, and the product is given by twisted symmetrization (see Definition~\ref{shuf}).
We expect $\rho$ to be injective (Conjecture~\ref{conjAll}). This prediction is supported by the following theorem:
\begin{thm}\label{thmalgintro}
If $A=H$ are the usual Borel-Moore homology groups, the map $\rho:H\bf{Ha}^{0,T}_{0,C}\ra H\bf{Sh}_C$ becomes injective after tensoring by $\Frac(A_T(pt))$ (Corollary~\ref{shufflefaith}).
\end{thm}
If the conjecture is true, this map can be used to find relations in $A\bf{Ha}^{0,T}_{C}$ via direct computations, and also to transport a natural coproduct from $A\bf{Sh}_C$.

Next, let us pick a locally free sheaf $\cal F$ as framing.
\begin{defi*}
A \textit{stable Higgs triple} of rank $0$, degree $d$ and frame $\cal F$ is a triple $(\cal E,\alpha,\theta)$ with $\cal E\in\ona{Coh}_{0,d}C$, $\alpha\in\Hom(\cal F,\cal E)$, $\theta\in\bbExt^1(\cal E,(\cal F\xra{\alpha}\cal E)\otimes \omega)$, such that the image of $\alpha$ generates $\cal E$ under $\theta$ (Definition~\ref{trstable}).
\end{defi*}

\begin{thm}
Let $C$ be a smooth projective curve, and $d$, $n$ positive integers.
\begin{enumerate}
\item The moduli of stable Higgs triples of degree $d$ and frame $\cal F$ is represented by a smooth quasi-projective variety $\mathscr B(d,\cal F)$ (Theorem~\ref{moduli});
\item Let $\cal F=\bbk^n\otimes\Ocal$. Then for any $n$, the space $A\mathscr M_n=\bigoplus_d A(\mathscr B(d,\bbk^n\otimes\Ocal))$ is equipped with a structure of an $A\bf{Ha}^0_{C}$-module (Corollary~\ref{reps}).
\end{enumerate}
\end{thm}

The second part of this theorem is proved by the same methods as Theorem~\ref{thmalgintro}.
We strongly expect that the same result holds for any locally free $\cal F$.
As for the first part, it is done by realizing stable Higgs triples as sheaves on a compactification of $T^*C$. 
Namely, we have the following theorem:
\begin{thm}
	The variety $\mathscr B(d,\cal F)$ is isomorphic to the moduli space of $f$-semisimple torsion-free sheaves on $\bbP_C(\omega\oplus\Ocal)$, equipped with framing at infinity and satisfying certain numerical conditions.
	In particular, $\mathscr B(d,\Ocal)$ is isomorphic to the Hilbert scheme of points $\ona{Hilb}_d T^*C$ (Section~\ref{negu}).
\end{thm}
This isomorphism can be understood as a relative version of classical derived equivalence between the category of sheaves on $\bbP^1$ and of representations of the Kronecker quiver~\cite{Bei}.
We refer the reader to Section~\ref{negu} for definitions and precise statement.

Unfortunately, it is not entirely clear how to extend a $A\bf{Ha}^{0,T}_{C}$-module structure on $A\mathscr M^T_n$ to a Yetter-Drinfeld module~\cite{RT} with respect to some coproduct on $A\bf{Ha}^{0,T}_{C}$.
Still, the isomorphism $\mathscr B(d,\Ocal)\simeq\ona{Hilb}_{d}(T^*C)$ suggests that $A\mathscr M^T_n$ should admit an action of the Drinfeld double of $A\bf{Ha}^{0,T}_{C}$, similar to~\cite[Chapter 8]{Nak}.

In higher rank, we expect the moduli of stable Higgs triples to retain a close relation to the moduli of sheaves on $\bbP_C(\omega\oplus\Ocal)$ framed at infinity.
This is evidenced by the fact that similar objects appear in the works of Negu{\c t}~\cite{Ne2,Ne3}, where for any smooth projective surface $S$ he defines an action of a certain $\cal W$-algebra on the $K$-theory of moduli of stable sheaves on $S$.
We expect that for $S=T^*C$, these algebras get embedded into a suitable completion of $K\bf{Ha}^0_{C}$.
In general, since Higgs sheaves on $C$ can be thought of as coherent sheaves with proper support on $T^*C$ via BNR-correspondence~\cite{BNR}, one can imagine a much more general picture:

\begin{gupr}
Let $S$ be a smooth projective surface together with a smooth divisor $D\subset S$.
Denote by $\underline{\cal Coh}(S,D)$ the stack of $\Ocal_S$-modules with support disjoint from $D$, and by $\underline{\cal Coh}_0(S,D)$ its substack of $\Ocal_S$-modules of finite length.
Then $A\bf{Ha}_S=A(\underline{\cal Coh}(S,D))$ should admit a Hall-like structure of an associative algebra, such that $A\bf{Ha}^0_{S}=A(\underline{\cal Coh}_0(S,D))$ is a subalgebra containing $A$-theoretic $\cal W$-algebra.
Furthermore, the Drinfeld double $D(A\bf{Ha}_S)$ should act on the $A$-theory of stable sheaves on $S$ framed at $D$, for a certain stability condition.
\end{gupr}

For this principle to hold true, one will certainly need additional technical assumptions, such as transversality of the divisor defining stability condition with $D$.
However, this discussion reaches far beyond the scope of this article.

Since the first draft of the present paper has appeared, some additional progress has been made in generalizing its results.
As mentioned above, the definition of algebra $A\bf{Ha}^0_{C}$ was extended to positive rank Higgs sheaves in~\cite{SS}.
In $K$-theory, the rank 0 algebra $K\bf{Ha}^0_{S}$ was defined for any smooth surface $S$ in~\cite{Zh}.
In homology, the full algebra $H\bf{Ha}_S$ was defined in~\cite{KV}.
Moreover, it was shown $H\bf{Ha}_S$ acts on the Borel-Moore homology groups of rank 1 semi-stable sheaves on $S$.

Let us finish the introduction with a brief outline of the structure of the paper.
In Section~\ref{cohquot} we choose explicit presentations of $\underline{\cal Coh}_{0,d}$ and $\underline{\cal Higgs}_{0,d}$ as global quotient stacks, given by certain $\cal Quot$-schemes. We also recollect basic facts about these schemes.
In Section~\ref{prod} we recall a construction introduced in works of Schiffmann and Vasserot, which permits us to define an associative product on $\bigoplus_d A(\underline{\cal Higgs}_{0,d})$.
In Section~\ref{globsh} we introduce \textit{global shuffle algebras} $A\bf{Sh}_g$, prove that these algebras satisfy some quadratic relations, and obtain a shuffle presentation $\rho$ of $A\bf{Ha}_C^{0,T}$ for a certain choice of $g$. The map $\rho$ is obtained by localizing our product diagrams to the fixed point sets of a certain torus $\bbT$. In passing, we also propose a geometric interpretation of the difference between two types of shuffle product, appearing in literature in similar context (Corollary~\ref{prodnorm}).
In Section~\ref{injsh}, we prove that for $A=H$ the shuffle presentation $\rho$ is faithful. The proof uses the scaling torus action and weight filtration in a crucial way, so that it cannot be easily translated to other homology theories. However, we conjecture that $\rho$ is faithful for general $A$.
In Section~\ref{modtri} we introduce the moduli stack of Higgs triples, construct an action of $A\bf{Ha}_C^{0}$ on $\bigoplus_d A(\mathscr B(d,n))$, and discuss how it can be related to the classical action of Heisenberg algebra on cohomology groups of Hilbert schemes of points on $T^*C$~\cite{Nak}.
In Section~\ref{quishe}, we collect some technical facts about quiver sheaves for later use.
In Section~\ref{negu}, we provide an alternative description of $\mathscr B(d,n)$ as a moduli of sheaves on a compactification of $T^*C$. We also briefly describe the relation between our work and the $W$-algebras of Negu\c{t}.
Finally, in Appendix~\ref{OBM} we recall the notion of oriented Borel-Moore homology functor, following the monograph by Levine and Morel~\cite{LM}, and gather the statements necessary for our proofs. In particular, we adapt the localization theorem of Borel-Atiyah-Segal to this framework.

\subsection*{Acknowledgements:}

This paper constitutes a part of author's Ph.D. thesis, written under direction of Olivier Schiffmann.
The author would like to thank him for his perpetual support and constant encouragement.
I would also like to thank Quoc Ho, Sergei Mozgovoy, Andrei Negu{\c t}, Francesco Sala and Gufang Zhao for their help and illuminating discussions, and the anonymous referee for their valuable suggestions.

\subsection*{Conventions:}
We denote by $\Sch/\bbk$ the category of $\bbk$-schemes of finite type over $\bbk$; $pt$ stands for the terminal object $\Spec\bbk\in \Sch/\bbk$.
For any $X\in \Sch/\bbk$, the category of coherent $\Ocal_X$-sheaves is denoted by $\ona{Coh}X$.
We will usually denote coherent sheaves by calligraphic letters, and implicitly identify locally free sheaves with corresponding vector bundles.
For any $\cal E,\cal F\in \ona{Coh}X$, we write $\Ext^i(\cal E,\cal F)$ for Ext-functors, $\Hom:=\Ext^0$, and $\cal Ext^i(\cal E,\cal F)$ for Ext-sheaves, $\cal Hom:=\cal Ext^0$.
More generally, for any two complexes of sheaves $\cal E^\bullet,\cal F^\bullet$ we denote by $\bbHom(\cal E^\bullet,\cal F^\bullet)$ the space of morphisms in the derived category $\cal D^b(\ona{Coh}X)$, and $\bbExt^i(\cal E^\bullet,\cal F^\bullet):=\bbHom(\cal E^\bullet,\cal F^\bullet[i])$.
Finally, we will liberally use the language of stacks; see~\cite{LMB} or~\cite{Ols} for background.

%% file: coh.tex
\section{Coherent sheaves and $\cal Quot$-schemes}\label{cohquot}

Let $\bbk$ be an algebraically closed base field of characteristic $0$. Let $C$ be a smooth projective curve defined over $\bbk$, and $\Ocal$ its structure sheaf. Then one can define the following algebraic stacks (over $\Sch/\bbk$ in étale topology):
\begin{itemize}
\item $\underline{\cal Coh}_{0,d}$, the stack of torsion sheaves on $C$ of degree $d$~\cite[Théorème 4.6.2.1]{LMB};
\item for any $\cal F\in \ona{Coh} C$, the stack $\underline{\cal Coh}_{0,d}^{\leftarrow \cal F}$ of pairs $(\cal E\in \ona{Coh}_{0,d}C,\alpha\in\Hom(\cal F,\cal E))$~\cite[Section 4.1]{GHS};
\item the cotangent stack $\underline{\cal Higgs}_{0,d}:=T^*\underline{\cal Coh}_{0,d}$. It is defined as the relative Spec of the symmetric algebra of the tangent sheaf; see \cite[Chapitre 14, 17]{LMB} for the relevant definitions.
\end{itemize}

\begin{rmq}
Note that the tangent sheaf of a stack is \textit{not} the same as the tangent complex, but is rather its zeroth cohomology.
\end{rmq}

All of the stacks above can be realized as global quotient stacks. Below we will make an explicit choice of such presentation for computational purposes.

\begin{defi}
Let $\underline{\cal Quot}_{0,d}$ be the following functor:
\begin{align*}
\underline{\cal Quot}_{0,d}:\Sch/\bbk & \ra Set^{op},\\
T & \mapsto \left\{\phi:\bbk^d\otimes \Ocal_{T\times C}\twoheadrightarrow \cal E_T \,\middle\vert\, \begin{array}{c}
\cal E_T\in \ona{Coh}(T\times C),\phi\text{ flat over }T,\\
\text{for any }t\in T, \rk\cal E_t=0\text{ and }\deg\cal E_t=d.\\
\end{array}\right\},\\
(T'\xra{f} T) & \mapsto (\phi\mapsto f^*\phi).
\end{align*}
Moreover, let us consider its open subfunctor $\underline{\cal Quot}^\circ_{0,d}\subset \underline{\cal Quot}_{0,d}$, consisting of quotients
\[
\phi:\bbk^d\otimes \Ocal_{T\times C}\twoheadrightarrow \cal E_T,
\]
such that the map $H^0(\phi_t):\bbk^d\ra H^0(\cal E_t)$ is an isomorphism for all $t\in T$.
\end{defi}

Let $G_d:=GL_d(\bbk)$. Note that $G_d$ acts on $\underline{\cal Quot}^\circ_{0,d}$ via linear transformations of $\bbk^d$.

\begin{prop}\label{atlas}
Let $d>0$ be an integer.
\begin{enumerate}
\item $\underline{\cal Quot}_{0,d}$ and $\underline{\cal Quot}^\circ_{0,d}$ are representable by smooth schemes $\cal Quot_{0,d}$ and $\cal Quot^\circ_{0,d}$ respectively, and $\cal Quot_{0,d}$ is a projective variety of dimension $d^2$;
\item we have an isomorphism $[\cal Quot_{0,d}^\circ/G_d]\simeq\underline{\cal Coh}_{0,d}$;
\item let $\cal F\simeq V\otimes \Ocal$ for some finite dimensional vector space $V$. Then we have an isomorphism $[(\Hom(V,\bbk^d)\times\cal Quot_{0,d}^\circ)/G_d]\simeq\underline{\cal Coh}^{\leftarrow \cal F}_{0,d}$.
\end{enumerate}
\end{prop}
\begin{proof}
For (1), see~\cite{LeP}. The claim (2) follows from the observation that any torsion sheaf of degree $d$ is generated by its global sections, and every isomorphism of torsion sheaves is completely determined by its action on global sections. Finally, for (3) let us consider the natural map $\underline{\cal Coh}_{0,d}^{\leftarrow V\otimes\Ocal}\ra \underline{\cal Coh}_{0,d}$. This is a vector bundle, which is trivialized in the atlas given by $\cal Quot^\circ_{0,d}$:
$$
\underline{\cal Coh}_{0,d}^{\leftarrow V\otimes\Ocal}\times_{\underline{\cal Coh}_{0,d}} \cal Quot^\circ_{0,d}\simeq \Hom(V,\bbk^d)\times\cal Quot_{0,d}^\circ.
$$
The statement of (2) then implies the desired isomorphism.
\end{proof}

\begin{rmq}
Note that (3) fails for other sheaves $\cal F\not\simeq V\otimes\Ocal$. In general, we would have to pick a certain closed subvariety out of $\Hom(\cal F,\bbk^d\otimes \Ocal)\times\cal Quot_{0,d}^\circ$; see~\cite{HL} for more details.
\end{rmq}

Recall that for any algebraic group $G$ and any smooth $G$-variety $X$ the cotangent bundle $T^*X$ is naturally equipped with a Hamiltonian $G$-action. Let $\mu: T^*X\ra \fk g^*$ be the corresponding moment map, where $\fk g$ is the Lie algebra of $G$, and put $T^*_G X:=\mu^{-1}(0)$. Note that the infinitesimal $G$-action provides a morphism $\fk g\otimes\Ocal_X\xra{\mu^*} \cal T_X$, where $\cal T_X$ is the tangent sheaf of $X$.

\begin{lm}\label{ham}
Let $X$ be a smooth variety equipped with an action of $G$. Then we have a natural isomorphism of stacks $T^*[X/G]\simeq[T^*_G X/G]$.
\end{lm}
\begin{proof}
It follows from the definition of the moment map that the composition $$T^*X=\Spec\ona{Sym}(\cal T_X)\xra{\mu^*} \Spec\ona{Sym}(\fk g\otimes\Ocal_X)=\fk g^*\times X\xra{pr_1}\fk g^*$$ coincides with $\mu$. Therefore $\Spec\ona{Sym}(\Coker\mu^*)\simeq \mu^{-1}(0)=T^*_G X$, and we obtain the desired isomorphism after descending to $[X/G]$.
\end{proof}
The lemma above implies that $$\underline{\cal Higgs}_{0,d}=T^*\underline{\cal Coh}_{0,d}\simeq [T^*_G \cal Quot^\circ_{0,d}/G].$$

\begin{exe}\label{q01}
Let $d=1$. As a set, $\cal Quot_{0,1}(\bbk)=\{\Ocal\xra{\phi} \cal E\mid \cal E\in \ona{Coh}_{0,1}, \phi\neq 0\}$. Note that since the $G_1=\bb G_m$-action is trivial here and $\Hom(\Ocal,\cal E)=\bbk$ by Riemann-Roch, we actually have $\cal Quot_{0,1}\simeq\cal Quot^\circ_{0,1}\simeq C$, and $\underline{\cal Coh}_{0,1}\simeq C\times \mathrm B\bb G_m$. The corresponding universal family is given by $\Ocal_\Delta\in \ona{Coh}(C\times C)$, where $\Ocal_\Delta$ is the structure sheaf of the diagonal $\Delta\subset C\times C$. Moreover, $T^*_G \cal Quot_{0,1}\simeq T^* \cal Quot_{0,1}\simeq T^* C$, and thus $\underline{\cal Higgs}_{0,d}\simeq T^*C\times \mathrm{B}\bb G_m$.
\end{exe}

Let us also define the following filtered version of $\cal Quot^\circ_{0,d}$.

\begin{defi}
For any $d_\bullet=\{0=d_0\leq d_1\leq\ldots\leq d_k=d\}$, fix a filtration $\bbk^{d_1}\subset\ldots\subset\bbk^{d}$. Denote by $\widetilde{\cal Quot}_{0,d_\bullet}$ the subset of $\cal Quot_{0,d}$ consisting of quotients $\bbk^d\otimes \Ocal\xra{\phi} \cal E$ such that the map $H^0(\phi)|_{\bbk^{d_i}\otimes \Ocal}:\bbk^{d_i}\ra H^0(\Im\phi|_{\bbk^{d_i}\otimes \Ocal})$ is an isomorphism for any $i$.
\end{defi}

We introduce the following notations for later use:
\begin{gather*}
G_{d_\bullet}=\prod\limits_{i=1}^k G_{d_i-d_{i-1}},\quad \cal Quot_{0,d_\bullet}=\prod\limits_{i=1}^k \cal Quot_{0,d_i-d_{i-1}},\quad \cal Quot^\circ_{0,d_\bullet}=\prod\limits_{i=1}^k \cal Quot^\circ_{0,d_i-d_{i-1}}.
\end{gather*}

We also fix isomorphisms $\bbk^{d_i}/\bbk^{d_{i-1}}\simeq \bbk^{d_i-d_{i-1}}$ for each $i\in [1,k]$, so that $G_{d_i-d_{i-1}}$ gets identified with invertible maps in $\Hom(\bbk^{d_i-d_{i-1}},\bbk^{d_i}/\bbk^{d_{i-1}})$, and quotients of $(\bbk^{d_i}/\bbk^{d_{i-1}})\otimes \Ocal$ are parametrized by the variety $\cal Quot_{0,d_i-d_{i-1}}$.

\begin{prop}\label{quot}
$\widetilde{\cal Quot}_{0,d_\bullet}$ is a smooth closed subvariety in $\cal Quot^\circ_{0,d}$.
\end{prop}

\begin{proof}
In order to prove that $\widetilde{\cal Quot}_{0,d_\bullet}$ is a closed subvariety of $\cal Quot^\circ_{0,d}$, let us recall the construction of $\cal Quot$-schemes in~\cite[Chapter 4]{LeP}. Namely, fix $n\gg 0$ and an ample line bundle $\Ocal(1)$ on $C$. Denote by $H=H^0(\Ocal(n))$ the space of global sections of $\Ocal(n)$, $h=\dim H$. Let $\Gr^d(\bbk^d\otimes H)$ be the Grassmanian of subspaces of codimension $d$ in $\bbk^d\otimes H$, and consider the following map:
\begin{align*}
\cal Quot_{0,d}&\ra \Gr^d(\bbk^d\otimes H),\\
\left(0\ra\cal K\ra \bbk^d\otimes\Ocal\xra\phi \cal E\ra 0\right)&\mapsto \left(H^0(\cal K\otimes \Ocal(n))\subset \bbk^d\otimes H\right).
\end{align*}
It is a closed embedding for $n$ big enough. Now, for each quotient $\bbk^d\otimes \Ocal\xra{\phi} \cal E$ and for each $i\in [1,k]$ we have the restricted short exact sequence 
\begin{align}\label{seqres}
0\ra\cal K_i\ra \bbk^{d_i}\otimes \Ocal\xra{\phi_i}\cal E_i\ra 0
\end{align}
with $\cal K_i:=\cal K\cap (\bbk^{d_i}\otimes \Ocal)$, $\phi_i=\phi|_{\bbk^{d_i}\otimes\Ocal}$, and $\cal E_i=\Im\phi_i$. Since $H^0(\phi)$ is an isomorphism, $H^0(\phi_i)$ is injective for all $i$, and thus $h^0(\cal E_i)\geq d_i$ for all $i$; moreover, $\phi$ belongs to $\widetilde{\cal Quot}_{0,d_\bullet}$ precisely when all the previous inequalities turn into equalities. Tensoring~(\ref{seqres}) by $\Ocal(n)$ and taking global sections, we get an exact sequence
\begin{align*}
0\ra H^0(\cal K_i(n))\ra \bbk^{d_i}\otimes H\ra H^0(\cal E_i(n))\ra 0
\end{align*} 
for $n$ big enough and all $\cal K_i$. Since $\cal E_i$ is torsion sheaf, there exists an isomorphism $\cal E_i(n)\simeq \cal E_i$, and the exact sequence above implies that
$$
\dim H^0(\cal K_i(n))=\dim H^0(\cal K(n))\cap (\bbk^{d_i}\otimes H)\leq d_i(h-1),
$$
where the equality holds if and only if $\phi$ belongs to $\widetilde{\cal Quot}_{0,d_\bullet}$. Therefore 
$$
\widetilde{\cal Quot}_{0,d_\bullet}=\cal Quot^\circ_{0,d}\cap\left\{V\subset \bbk^d\otimes H\,\middle\vert\, \begin{array}{c}
\dim V=d(h-1),\\
\dim V\cap (\bbk^{d_i}\otimes H)\geq d_i(h-1)
\end{array}\right\}\subset \Gr^d(\bbk^d\otimes H).
$$
The second set is closed in $\Gr(\bbk^d\otimes H)$, and thus $\widetilde{\cal Quot}_{0,d_\bullet}$ is closed in $\cal Quot^\circ_{0,d}$ as well.

In order to prove that $\widetilde{\cal Quot}_{0,d_\bullet}$ is smooth, consider the following diagram:
$$
\begin{tikzcd}
	 & {\widetilde{\cal Quot}_{0,d_\bullet}\times G_{d_{\bullet}}}\arrow[dl,"p"]\arrow[dr,"q"'] & \\
	{\cal Quot_{0,d_\bullet}\times G_{d_\bullet}} & & {\widetilde{\cal Quot}_{0,d_\bullet}}
\end{tikzcd}
$$
where we identify $G_{d_i-d_{i-1}}$ with invertible maps in $\Hom(\bbk^{d_i-d_{i-1}},\bbk^{d_i}/\bbk^{d_{i-1}})$, the map $p$ sends $(\phi, (g_i)_{i=1}^k)$ to $\left((g_i^*(\phi|_{\bbk^{d_i}/\bbk^{d_{i-1}}}))_{i=1}^k,(g_i)_{i=1}^k\right)$, and $q$ is the projection on the first coordinate. Note that we have the following map between short exact sequences for any point in $\widetilde{\cal Quot}_{0,d_\bullet}$:
$$
\begin{tikzcd}
0\arrow[r] & \bbk^{d_{i-1}}\arrow[r]\arrow[d,"\simeq"] & \bbk^{d_i}\arrow[r]\arrow[d,"\simeq"] & \bbk^{d_i}/\bbk^{d_{i-1}}\arrow[r]\arrow[d,"H^0(\phi|_{\bbk^{d_i}/\bbk^{d_{i-1}}})"] & 0\\
0\arrow[r] & H^0(\cal E_{i-1})\arrow[r] & H^0(\cal E_i)\arrow[r] & H^0(\cal E_i/\cal E_{i-1})\arrow[r] & 0
\end{tikzcd}
$$
It follows that the map $\phi|_{\bbk^{d_i}/\bbk^{d_{i-1}}}:(\bbk^{d_i}/\bbk^{d_{i-1}})\otimes\Ocal\ra \cal E_i/\cal E_{i-1}$ induces an isomorphism on global sections, and thus the image of $p$ belongs to $\cal Quot^\circ_{0,d_\bullet}\times G_{d_\bullet}$. Moreover, the diagram above also implies that $p$ is an affine fibration over $\cal Quot^\circ_{0,d_\bullet}$ with fiber
$$\bigoplus\limits_{i=1}^k\Hom(\bbk^{d_i-d_{i-1}}\otimes\Ocal, \cal E_{i-1})=\bigoplus\limits_{i=1}^k\bbk^{d_i-d_{i-1}}\otimes H^0(\cal E_{i-1})$$
of dimension $(d_i-d_{i-1})d_{i-1}$. Since $\cal Quot^\circ_{0,d_\bullet}$ is smooth and $q$ is a trivial $G_{d_{\bullet}}$-torsor, this observation implies the smoothness of $\widetilde{\cal Quot}_{0,d_\bullet}$.
\end{proof}
\begin{rmq}
In the proof above we chose an $n$ big enough so that all $\cal K$'s and $\cal K_i$'s cease to have higher cohomology groups and become generated by global sections after tensoring by $\Ocal(n)$.
It is possible because all our sheaves are parametrized by a finite union of $\cal Quot$-schemes, and thus form bounded families (see Lemma 4.4.4 in~\cite{LeP}).
\end{rmq}

\begin{nota}
Throughout the paper, for any quotient $\bbk^d\otimes\Ocal\xra{\phi}\cal E$ in $\cal Quot_{0,d}$ we will denote $\Ker \phi$ by $\cal K$, and the inclusion $\cal K\hookrightarrow \bbk^d\otimes\Ocal$ by $\iota$. We will also decorate $\cal K$, $\cal E$, $\phi$ and $\iota$ with appropriate indices and markings.
\end{nota}

Next, we recall the description of tangent spaces of $\cal Quot$-schemes.
\begin{prop}\label{tan}
Let $\bbk^d\otimes\Ocal\xra{\phi} \cal E$ be a point in $\cal Quot_{0,d}$, and let $\cal K=\Ker \phi$. Then the tangent space $T_\phi \cal Quot_{0,d}$ at $\phi$ is naturally isomorphic to $\Hom(\cal K,\cal E)$. Moreover, if $\phi\in \widetilde{\cal Quot}_{0,d_\bullet}$ we have
$$
T_\phi \widetilde{\cal Quot}_{0,d_\bullet}=\{v\in \Hom (\cal K,\cal E): v(\cal K_i)\subset\cal E_i\quad \forall i\}.
$$
\end{prop}
\begin{proof}
The proof of the first claim can be found in~\cite[Chapter 8]{LeP}. The second claim can be proved in a similar fashion, keeping track of the condition of admitting a sub-quotient throughout the proof of the first claim. 
\end{proof}

Because of this proposition, we will usually regard $T^*\cal Quot^\circ_{0,d}$ as a variety, whose $\bbk$-points are identified with pairs
\begin{gather*}
T^*\cal Quot^\circ_{0,d}=\left\{(\phi,\beta)\mid \phi\in\cal Quot^\circ_{0,d},\text{ }\beta\in \Hom(\cal K,\cal E)^*\right\}.
\end{gather*}

Additionally, let us define for later purposes the \textit{nilpotent} part $(T_\phi \widetilde{\cal Quot}_{0,d_\bullet})^{nilp}$ of $T_\phi \widetilde{\cal Quot}_{0,d_\bullet}$:
\begin{equation}\label{TQuotnilp}
(T_\phi \widetilde{\cal Quot}_{0,d_\bullet})^{nilp}:=\{v\in \Hom (\cal K,\cal E): v(\cal K_i)\subset\cal E_{i-1}\quad \forall i\}.
\end{equation}

Note that
$$
\Hom(\bbk^d\otimes\Ocal, \cal E)\simeq\Hom(\bbk^d,H^0(\cal E))\simeq \Hom(\bbk^d,\bbk^d)=:\fk g_d,
$$
with the second isomorphism being induced by $H^0(\phi):\bbk^d\xra{\sim}H^0(\cal E)$. In these terms the moment map for the $G_d$-action on $T^*\cal Quot^\circ_{0,d}$ can be written as follows:
\begin{align*}
\mu: T^*\cal Quot^\circ_{0,d} &\ra\mathfrak g_d^*,\\
(\phi,\beta)&\mapsto \iota^*\beta\in \Hom(\bbk^d\otimes\Ocal, \cal E)^*\simeq \mathfrak g_d^*.
\end{align*}

Since $\cal E$ is a torsion sheaf, we have $\Ext^1(\bbk^d\otimes\Ocal, \cal E)=0$, and thus over each $\phi\in\cal Quot^\circ_{0,d}$ the restriction $\mu_\phi$ of the map $\mu$ to $T^*_\phi\cal Quot^\circ_{0,d}$ can be embedded in a long exact sequence:
$$
0\ra \Ext^1(\cal E,\cal E)^*\ra \Hom(\cal K,\cal E)^*\xra{\mu_\phi} \Hom(\bbk^d\otimes\Ocal, \cal E)^*\ra\Hom(\cal E,\cal E)^*\ra 0.
$$
This implies that $\mu_\phi^{-1}(0)\simeq\Ext^1(\cal E,\cal E)^*$, and we get an identification on the level of $\bbk$-points 
\begin{equation}\label{ExtInQuot}
T^*_{G_d}\cal Quot^\circ_{0,d}=\left\{(\phi,\beta): \phi\in\cal Quot^\circ_{0,d},\text{ }\beta\in \Ext^1(\cal E,\cal E)^*\right\}\subset T^*\cal Quot^\circ_{0,d}.
\end{equation}

\begin{exe}\label{quotA1}
Let $C=\bb A^1$. Even though this curve is not projective, we can fix an isomorphism $\bb A^1= \bbP^1\setminus \{\infty\}$, and define
$$
\cal Quot_{0,d}(\bb A^1):=\{\bbk^d\otimes\Ocal_{\bbP^1}\xra{\alpha}\cal E: \supp \cal E\subset\bb A^1\}\subset \cal Quot_{0,d}(\bb P^1).
$$
Then the open subvariety $$\cal Quot_{0,d}^\circ(\bb A^1)=\cal Quot_{0,d}(\bb A^1)\cap \cal Quot_{0,d}^\circ(\bb P^1)$$ parametrizes quotients $\bbk^d[t]\xra{\alpha} V$, where $V$ is a torsion $\bbk[t]$-module of length $d$, and $\alpha$ induces an isomorphism of $\bbk$-vector spaces $\bbk^dt^0\simeq V$. Equivalently, $\cal Quot_{0,d}^\circ(\bb A^1)$ parametrizes linear operators on $\bbk^d$, that is
\begin{align*}
\cal Quot_{0,d}^\circ(\bb A^1)\simeq \mathfrak g_d,
\end{align*}
where the $G_d$-action on the left gets identified with the adjoint action on the right. Thus $T^*\cal Quot_{0,d}^\circ(\bb A^1)\simeq (\mathfrak g_d)^2$, and $T_{G_d}^*\cal Quot_{0,d}^\circ(\bb A^1)$ is isomorphic to the commuting variety $\mathscr C(\mathfrak g_d)=\{(u,v)\in (\mathfrak g_d)^2:[u,v]=0\}$.
\end{exe}
In light of the example above, we will refer to $T^*_{G_d}\cal Quot_{0,d}^\circ$ as the \textit{commuting variety} of $C$, and denote it by $\mathscr C_d=\mathscr C_d(C)$. We will also write $\mathscr C_{d_\bullet}=\mathscr C_{d_\bullet}(C):=T^*_{G_{d_\bullet}}\cal Quot^\circ_{0,d_\bullet}$.

\begin{exe}\label{quotloc}
Let us fix a geometric point $x\in C(\bbk)$, and consider the \textit{punctual} $\cal Quot$-scheme
$$
\cal Quot_{0,d}^\circ(x):= \{\bbk^d\otimes\Ocal\xra{\alpha}\cal E: \supp \cal E=x\} \subset \cal Quot^\circ_{0,d}.
$$
Such quotient is completely determined by its localization at $x$. More explicitly, since $C$ is smooth, the completion $\hat{\Ocal}_x$ is (non-canonically) isomorphic to $\bbk \llbracket t\rrbracket$. The stalk of $\alpha$ at $x$ is thus of the form $\bbk^d \llbracket t\rrbracket\xra{\alpha_1}\cal E$, where $\cal E$ is a $\bbk \llbracket t\rrbracket$-module, and $\alpha_1$ induces an isomorphism of $\bbk$-vector spaces $\bbk^dt^0\simeq \cal E$. Such quotient is in its turn uniquely determined by a nilpotent operator $T$ on $\bbk^d$, the correspondence given by
$$
T\in\End(\bbk^d)\mapsto 
(0\ra \bbk^d\llbracket t \rrbracket\xra{\iota} \bbk^d\llbracket t \rrbracket\ra \cal E\ra 0),\quad \iota(vt^i)=(T.v)t^i-vt^{i+1}.
$$
Thus we see that $\cal Quot_{0,d}^\circ(x)$ is isomorphic to the nilpotent cone $\mathscr N_d\subset\fk g_d$ together with the adjoint action of $G_d$. Moreover, under this identification the cotangent space $\Hom(\cal K,\cal E)^*$ in $\cal Quot^\circ_{0,d}$ of a point $\alpha$ gets identified with $\fk g_d$, and the restriction of the moment map $\mu:T^*\cal Quot^\circ_{0,d}\ra \fk g_d$ to $\mathscr N_d\times \fk g_d$ is the commutator. In particular,
$$
\mathscr C_d\cap T^*\cal Quot^\circ_{0,d}|_{\cal Quot^\circ_{0,d}(x)}\simeq \mathscr C_d^{\mathrm n,\bullet}:=\{(u,v)\in (\mathfrak g_d)^2:[u,v]=0, u\text{ nilpotent}\}.
$$

\end{exe}

%% file: product.tex
\section{The product}\label{prod}
Let us once and for all fix a free oriented Borel-Moore homology theory (OBM) $A$; for the definition and basic facts about this notion, see Appendix~\ref{OBM}.
As explained there, we abuse the notation somewhat and consider the usual Borel-Moore homology $H_*$ as if it were a free OBM.
We also equip the cotangent bundle $T^*C$ with an action of $\bb G_m$, given by dilations along the fibers; let us denote this torus by $T$.

We begin by recalling a general construction from~\cite{SV}.
Let $G$ be an algebraic group with fixed Levi and parabolic subgroups $H\subset P$.
Assume we are given smooth quasi-projective varieties $X'$, $Y$, $V$, equipped with actions of $G$, $H$, $P$ respectively, and $H$-equivariant morphisms
$$
\begin{tikzcd}
& V\arrow[dl,"p"']\arrow[dr,"q"] & \\
Y & & X'
\end{tikzcd}
$$
such that $p$ is an affine fibration and $q$ is a closed embedding.
Set $W=G\times_P V$, $X=G\times_P Y$, where the $P$-action on $Y$ is induced by the natural projection $P\ra H$, and consider the following maps of $G$-varieties:
\begin{center}
\begin{tabular}{rc}
$
\begin{tikzcd}
& W\arrow[dl,"f"']\arrow[dr,"g"] & \\
X & & X'
\end{tikzcd}
$&
\begin{tabular}{l}
\\$f:(g,v)\ona{mod}P\mapsto (g, p.v)\ona{mod}P,$\\
$g:(g,v)\ona{mod}P\mapsto g.q(v).$\\
\end{tabular}
\end{tabular}
\end{center}

The map $(f,g): W\ra X\times X'$ is a closed embedding, so from now on we will identify the smooth variety $W$ with its image in $X\times X'$. Let $Z=T^*_W(X\times X')$ be the conormal bundle.
Projections on factors define two maps:
$$
\begin{tikzcd}
& Z\arrow[dl,"\Phi"']\arrow[dr,"\Psi"] & \\
T^*X & & T^*X'
\end{tikzcd}
$$
We denote $Z_G=Z\cap (T^*_G X\times T^*_G X')$. Then $\Phi^{-1}(T^*_G X)=\Psi^{-1}(T^*_G X')=Z_G$~\cite[Lemma 7.3(b)]{SV1}, and we have the following induced diagram:
\begin{equation}\label{diagZ}
\begin{tikzcd}
 & Z_G\arrow[d,hook]\arrow[dl,"\Phi_G"']\arrow[dr,"\Psi_G"] & \\
T^*_G X\arrow[d,hook]& Z\arrow[dl,"\Phi"']\arrow[dr,"\Psi"] & T^*_G X'\arrow[d,hook]\\
T^*X & & T^*X'
\end{tikzcd}
\end{equation}
Now, $\Psi$ and $\Psi_G$ are projective, $\Phi$ is an lci map, so that we get the following morphisms in $A$-groups:
\begin{align*}
(\Psi_G)_* &: A^G(Z_G)\ra A^G(T^*_G X'),\\
\Phi^! &: A^G(T^*_G X)\ra A^G(Z_G).
\end{align*}
By composing these two maps and using the induction isomorphism $A^H(T^*_H Y)\xra{\sim}A^G(T^*_G X)$ (see Proposition~\ref{indA}), we obtain a map 
\begin{align}\label{pushpull}
\Upsilon=(\Psi_G)_*\circ \Phi^!\circ \ind_H^G: A^H(T^*_H Y)\ra A^G(T^*_G X').
\end{align}

Let us apply this general construction to a particular case of $\cal Quot$-schemes of rank $0$.
Namely, let $d_\bullet=\{0=d_0\leq d_1\leq\ldots\leq d_k=d\}$, denote $G=G_d$, $H=G_{d_\bullet}$, and let $P=P_{d_\bullet}$ be the parabolic group preserving the flag $\bbk^{d_1}\subset\ldots\subset\bbk^d$.
We use Gothic letters $\fk g$, $\fk p$, $\fk h$ for corresponding Lie algebras, and $\fk p_-$ for the parabolic algebra opposite to $\fk p$.
Next, put
\begin{align*}
Y=\cal Quot^\circ_{0,d_\bullet}, \qquad V=\widetilde{\cal Quot}_{0,d_\bullet},\qquad X'=\cal Quot^\circ_{0,d}.
\end{align*}
By Proposition~\ref{quot} we have a closed embedding $g:V\hookrightarrow X'$ and an affine fibration $f=\gr:V\twoheadrightarrow Y$.

The following lemma will help us to identify all the terms in diagram~(\ref{diagZ}).
\begin{lm}\label{calc}
Let $G,P,H,X',X,V,Y$ be as above.
\begin{enumerate}
\item There exist natural isomorphisms of $G$-varieties
\begin{gather*}
T^* X'=T^*\cal Quot^\circ_{0,d},\qquad Z=G\times_P\left\{(\phi,\beta)\in T^* \cal Quot^\circ_{0,d}|_{\widetilde{\cal Quot}_{0,d_\bullet}}: \beta|_{(T_\phi \widetilde{\cal Quot}_{0,d_\bullet})^{nilp}}=0\right\},\\
T^* X=G\times_P(\mathfrak p_-^*\times_{\mathfrak h^*}T^* \cal Quot^\circ_{0,d_\bullet})=G\times_P\{(x,(\phi_\bullet,\beta_\bullet))\in \mathfrak p_-^*\times T^* \cal Quot^\circ_{0,d_\bullet} : x|_{\mathfrak h}=\mu(\phi_\bullet,\beta_\bullet)\},
\end{gather*}
where $(T_\phi \widetilde{\cal Quot}_{0,d_\bullet})^{nilp}$ is defined as in~(\ref{TQuotnilp}).
For each $(\phi, \beta)\in T^*\widetilde{\cal Quot}_{0,d_\bullet}$ we have
\begin{align*}
\Phi((g,\phi,\beta)\ona{mod}P)&=(g,\mu(\phi,\beta),\gr(\phi,\beta))\ona{mod}P,\\
\Psi((g,\phi,\beta)\ona{mod}P)&=g.(\phi,\beta).
\end{align*}
\item There are isomorphisms of $G$-varieties
\begin{gather*}
T^*_G X'=\mathscr C_d, \qquad T^*_G X=G\times_P \mathscr C_{d_\bullet}.
\end{gather*}
\end{enumerate}
\end{lm}
\begin{proof}
(1) The first isomorphism is obvious, so we start with $T^* X$:
\begin{align*}
T^* X&=T^*(G\times_P \cal Quot^\circ_{0,d_\bullet})=T^*_P(G\times \cal Quot^\circ_{0,d_\bullet})/P\\
&=\{(g,x,(\phi_\bullet,\beta_\bullet))\in G\times \mathfrak g^*\times T^* \cal Quot^\circ_{0,d_\bullet} : g.x|_{\mathfrak p}-\mu(\phi_\bullet,\beta_\bullet)=0\in\mathfrak p^*\}/P\\
&=G\times_P\{(x,(\phi_\bullet,\beta_\bullet))\in \mathfrak p_-^*\times T^* \cal Quot^\circ_{0,d_\bullet} : x|_{\mathfrak h}=\mu(\phi_\bullet,\beta_\bullet)\}.
\end{align*}
Let us also note that the moment map $\mu: T^*X\ra \fk g^*$ is given by
\begin{align*}
\mu\left( (g,x,\phi_\bullet,\beta_\bullet)\text{ mod }P \right)=g.x.
\end{align*}
By the same reasoning,
\begin{align*}
T^* &(X\times X')=T^*(G\times_P (\cal Quot^\circ_{0,d_\bullet}\times \cal Quot^\circ_{0,d}))=T^*_P(G\times \cal Quot^\circ_{0,d_\bullet}\times \cal Quot^\circ_{0,d})/P\\
&=G\times_P\{(x,(\phi_\bullet,\beta_\bullet),(\phi, \beta))\in \mathfrak g^*\times T^* \cal Quot^\circ_{0,d_\bullet} \times T^* \cal Quot^\circ_{0,d} : x|_{\mathfrak p}=\mu(\phi_\bullet,\beta_\bullet)+\mu(\phi,\beta)|_{\mathfrak p}\}.
\end{align*}
Next, let us compute $Z=T^*_W(X\times X')$. We have $W=G\times_P \widetilde{\cal Quot}_{0,d_\bullet}$, and therefore
\begin{align*}
T^* W &= T^*_P(G\times \widetilde{\cal Quot}_{0,d_\bullet})/P\\
&= \{(g,x,\phi,\beta)\in G\times \mathfrak g^*\times T^*\widetilde{\cal Quot}_{0,d_\bullet} : g.x|_{\mathfrak p}-\mu(\phi,\beta)=0\}/P\\
&= G\times_P\{(x,\phi,\beta)\in \mathfrak g^*\times T^*\widetilde{\cal Quot}_{0,d_\bullet} : x|_{\mathfrak p}=\mu(\phi,\beta)\},
\end{align*}
so that
\begin{align*}
T^*(X\times X')|_W = G\times_P\left\{
\begin{array}{c}
(x,\phi,\beta_\bullet,\beta)\in \mathfrak g^*\times \widetilde{\cal Quot}_{0,d_\bullet} \times T^*_{\gr \phi} \cal Quot^\circ_{0,d_\bullet} \times T^*_\phi \cal Quot^\circ_{0,d}\\
\text{such that }x|_{\mathfrak p}=\mu(\phi_\bullet,\beta_\bullet)+\mu(\phi,\beta)|_{\mathfrak p}
\end{array}
\right\}.
\end{align*}
But the conormal bundle $T^*_W(X\times X')$ can be expressed as the kernel of the following map of vector bundles:
\begin{align*}
T^*(X\times X')|_W &\ra T^* W,\\
(g,x,\phi,\beta_\bullet, \beta)\ona{mod} P &\mapsto (g,x,\phi,\beta_\bullet-\beta|_{T_\phi \widetilde{\cal Quot}_{0,d_\bullet}})\ona{mod} P.
\end{align*}
Therefore, we finally obtain
\begin{align*}
T^*_W(X\times X') &=G\times_P \{(\phi,\beta_\bullet,\beta)\in \widetilde{\cal Quot}_{0,d_\bullet} \times T^*_{\gr \phi} \cal Quot^\circ_{0,d_\bullet} \times T^*_\phi \cal Quot^\circ_{0,d}: \beta|_{T_\phi \widetilde{\cal Quot}_{0,d_\bullet}}=\beta_\bullet\}\\
&\simeq G\times_P\left\{(\phi,\beta)\in \widetilde{\cal Quot}_{0,d_\bullet} \times T^*_\phi \cal Quot^\circ_{0,d}: \beta|_{(T_\phi \widetilde{\cal Quot}_{0,d_\bullet})^{nilp}}=0\right\}.
\end{align*}

Note that the desired formula for $\Phi$ follows from the first equality, and the formula for $\Psi$ is evident.
The claim (2) follows from the explicit descriptions of moment maps $T^*X'\ra\frak g^*$, $T^*X\ra\frak g^*$.
\end{proof}

The general construction thus produces a map
$$
\Upsilon:A^H(\mathscr C_{d_\bullet})\ra A^G(\mathscr C_d).
$$
For instance, if $k=2$, then $H=G_{d'}\times G_{d''}$, $d=d'+d''$, and we get a bilinear map
$$
\Upsilon:A\bf{Ha}_C^0[d']\otimes A\bf{Ha}_C^0[d'']\ra A\bf{Ha}_C^0[d], \qquad A\bf{Ha}_C^0[d]:=A(\underline{\cal Higgs}_{0,d})=A^{G_d}(\mathscr C_d).
$$
We denote $A\bf{Ha}_C^0=\bigoplus_{d\geq 0} A\bf{Ha}_C^0[d]$, where $A\bf{Ha}_C^0[0]:=A(pt)$.

\begin{thm}\label{ass}
$(A\bf{Ha}_C^0,\Upsilon)$ is an associative algebra.
\end{thm}

\begin{proof}
We begin by introducing some notations.
Let $d_1\leq d_2\leq d$, $d_\bullet=(d_1,d_2,d)$, $d'_\bullet=(d_1,d)$, $d''_\bullet=(d_2-d_1,d-d_1)$, $G=G_d$, $P=P_{d_\bullet}$, $P'=P_{d'_\bullet}$.
Define the following varieties:
\begin{align*}
X_1&=G\times_P \cal Quot^\circ_{0,d_\bullet}; & W_1&=G\times_{P'} \widetilde{\cal Quot}_{0,d'_\bullet};\\
X_2&=G\times_{P'}\cal Quot^\circ_{0,d'_\bullet}; & W_2&=G\times_{P} \widetilde{\cal Quot}_{0,d_\bullet};\\
X_3&=\cal Quot^\circ_{0,d}; & W_3&= G\times_P (\cal Quot^\circ_{0,d_1}\times \widetilde{\cal Quot}_{0,d''_\bullet}).
\end{align*}
These varieties are $\cal Quot$- and $\widetilde{\cal Quot}$-bundles over certain partial flag varieties, so we may identify their $\bbk$-points as pairs (flag, quotient).
Adopting ``mod 3''-notation for indices, we have obvious inclusions $W_i\hookrightarrow X_{i-1}\times X_{i+1}$.
\begin{lm}\label{transvers}
Using the notations above,
\begin{enumerate}
\item we have an isomorphism $W_2\ra W_1\times_{X_2} W_3$;
\item the intersection $(W_1\times X_1)\cap (X_3\times W_3)$ in $X_3\times X_2\times X_1$ is transversal.
\end{enumerate}
\end{lm}
\begin{proof}
First of all, we introduce a small abuse of notation.
Namely, for any morphism of sheaves $\cal E\xra{f}\cal F$ and for any subsheaf $\cal E'\hookrightarrow \cal E$, the codomain of $H^0(f)|_{\cal E'}$ is assumed to be $H^0(f(\cal E'))$.
With that in mind, we have
\begin{itemize}
\item[$-$] $X_1=\left\{(D_\bullet, (\phi_i)_{i=1}^3)\left| \begin{matrix}
\dim D_\bullet=d_\bullet, \phi_i:D_i/D_{i-1}\otimes \Ocal\ra \cal E_i,\\
H^0(\phi_i)\text{ is an iso}, i\in\{1,2,3\}
\end{matrix}\right.  \right\}$;
\item[$-$] $X_2=\{(D'_\bullet, (\phi_i)_{i=1}^2)\mid \dim D'_\bullet=d'_\bullet, \phi_i:D_i/D_{i-1}\otimes \Ocal\ra \cal E_i,H^0(\phi_i)\text{ is an iso}, i\in\{1,2\}\}$;
\item[$-$] $X_3=\{\phi\mid \bbk^d\otimes \Ocal\ra \cal E,H^0(\phi)\text{ is an iso}\}$;
\item[$-$] $W_1=\{(D'_\bullet, \phi)\mid \dim D'_\bullet=d'_\bullet, \phi:\bbk^d\otimes\Ocal\ra \cal E, H^0(\phi)|_{D_i\otimes\Ocal}\text{ is an iso},i\in\{1,2\}\}$;
\item[$-$] $W_2=\{(D_\bullet, \phi)\mid \dim D_\bullet=d_\bullet, \phi:\bbk^d\otimes\Ocal\ra \cal E, H^0(\phi)|_{D_i\otimes\Ocal}\text{ is an iso},i\in\{1,2,3\}\}$;
\item[$-$] $W_3=\left\{(D_\bullet, \phi_1,\phi_2)\left| \begin{matrix}
\dim D_\bullet=d_\bullet, \phi_1:D_1\otimes\Ocal\ra \cal E_1,\phi_2:\bbk^d/D_1\otimes \Ocal\ra \cal E_2,\\
H^0(\phi_i),i\in\{1,2\}, H^0(\phi_2)|_{D_2/D_1\otimes\Ocal}\text{ are iso}
\end{matrix}\right.  \right\}$.
\end{itemize}
Next, consider a commutative diagram
\begin{center}
\begin{tabular}{rc}
$
\begin{tikzcd}
X_1 & W_3\arrow[r,"p_{32}"]\arrow[l,"p_{31}"'] & X_2\\
& W_2\arrow[ul,"p_{21}"]\arrow[dr,"p_{23}"']\arrow[u,"\alpha_{23}"']\arrow[r,"\alpha_{21}"] & W_1\arrow[u,"p_{12}"']\arrow[d,"p_{13}"] \\
& & X_3
\end{tikzcd}
$&
\begin{tabular}{l}
$p_{12}:(D'_\bullet,\phi)\mapsto (D'_\bullet,\gr'\phi)$;\\
$p_{13}:(D'_\bullet,\phi)\mapsto \phi$;\\
$p_{21}:(D_\bullet,\phi)\mapsto (D_\bullet,\gr\phi)$;\\
$p_{23}:(D_\bullet,\phi)\mapsto \phi$;\\
$p_{31}:(D_\bullet,\phi_1,\phi_2)\mapsto (D_\bullet,(\phi_1,\gr''\phi_2))$;\\
$p_{32}:(D_\bullet,\phi_1,\phi_2)\mapsto (\oub(D_\bullet),(\phi_1,\phi_2))$;\\
$\alpha_{21}:(D_\bullet,\phi)\mapsto (\oub(D_\bullet),\phi)$;\\
$\alpha_{23}:(D_\bullet,\phi)\mapsto (D_\bullet,\gr'\phi)$,
\end{tabular}
\end{tabular}
\end{center}
where $\gr=\gr_{d_\bullet}:\widetilde{\cal Quot}_{0,d_\bullet}\ra \cal Quot^\circ_{0,d_\bullet}$ is the natural affine fibration, and
\begin{gather*}
\oub:(D_1\subset D_2\subset E)\mapsto (D_1\subset E),\\
\gr'=\gr_{d'_\bullet}, \gr''=\gr_{d''_\bullet}.
\end{gather*}

\paragraph{\textbf{(1)}}
We have the following equalities on the level of $\bbk$-points:
\begin{align*}
W_1\times_{X_2}W_3&=\left\{\left((D'_\bullet,\phi),(D_\bullet,\phi_1,\phi_2)\right)\mid\oub D_\bullet=D'_\bullet,(\phi_1,\phi_2)=\gr'\phi\right\}\\
&=\left\{\left((\oub D_\bullet,\phi),(D_\bullet,\gr'\phi)\right)\right\}\subset W_1\times W_3.
\end{align*}
The natural map
\begin{align*}
p=(p_{23},p_{21}):W_2 & \ra W_1\times_{X_2}W_3,\\
(D_\bullet,\phi) & \mapsto \left((\oub D_\bullet,\phi),(D_\bullet,\gr'\phi)\right)
\end{align*}
can be thus seen to be a bijection.
The fiber product $W_1\times_{X_2} W_3$ is normal by~\cite[Proposition 6.14.1]{EGAIV}, $W_2$ is connected, therefore Zariski's main theorem implies that $p$ is an isomorphism.

\paragraph{\textbf{(2)}}
To prove that our intersection is transversal, we need to show that for any $x\in (W_1\times X_1)\cap (X_3\times W_3)\simeq W_2$ there is an equality
$$
T_x(W_1\times X_1)\cap T_x(X_3\times W_3)= T_x W_2.
$$
By Proposition~\ref{tan} we have the following isomorphisms:
\begin{align}
T_\phi\cal Quot_{0,d}&=\Hom(\cal K,\cal E);\label{tand}\\
T_\phi\cal Quot_{0,d_\bullet}&=\Hom(\cal K_\bullet,\cal E_\bullet):=\bigoplus_i \Hom(\cal K_i/\cal K_{i-1},\cal E_i/\cal E_{i-1});\label{tandbul}\\
T_\phi\widetilde{\cal Quot}_{0,d_\bullet}&=\widetilde{\Hom_{d_\bullet}}(\cal K,\cal E):=\{\beta\in \Hom(\cal K,\cal E):\beta(\cal K_i)\subset \cal E_i\text{ for all $i$}\}.\label{tandtil}
\end{align}
Let us fix a flag $D_\bullet$ of dimension $d_\bullet$, and a quotient $\phi\in\cal Quot_{0,d}$ such that $(D_\bullet, \phi)\in W_2$.
Then the equalities above allow us to compute all the tangent spaces in question:
\begin{center}
\begin{tabular}{rc}
$
\begin{tikzcd}
T_{(D_\bullet,\gr\phi)}X_1 &[-15pt] T_{(D_\bullet,\gr'\phi)}W_3\arrow[r]\arrow[l] &[-15pt] T_{(D'_\bullet,\gr'\phi)}X_2\\
& T_{(D_\bullet,\phi)}W_2\arrow[ul]\arrow[dr]\arrow[u]\arrow[r] & T_{(D'_\bullet,\phi)}W_1\arrow[u]\arrow[d] \\
& & T_\phi X_3
\end{tikzcd}
$&
\begin{tabular}{l}
\\
$T_{(D'_\bullet,\phi)}W_1=\fk g/\fk p'\oplus \widetilde{\Hom_{d'_\bullet}}(\cal K,\cal E)$;\\
$T_{(D_\bullet,\phi)}W_2=\fk g/\fk p\oplus \widetilde{\Hom_{d_\bullet}}(\cal K,\cal E)$;\\
$T_{(D_\bullet,\gr'\phi)}W_3=\fk g/\fk p\oplus\Hom(\cal K_1,\cal E_1)$\\
$\qquad\qquad\qquad\qquad\oplus \widetilde{\Hom_{d_\bullet}}(\cal K/\cal K_1,\cal E/\cal E_1)$,
\end{tabular}
\end{tabular}
\end{center}
where $\fk p',\fk p\subset \fk g$ are parabolic subalgebras associated to flags $D'_\bullet$, $D_\bullet$ respectively, and $\cal K_1=\Ker\phi|_{D_1\otimes \Ocal_X}$.
A straightforward computation shows that at a point $x=(D_\bullet, \phi)\in W_2$ we have
$$
T_x(W_1\times X_1)\cap T_x(X_3\times W_3)=\{(\beta,(\xi^{red},\gr'\beta),(\xi,\gr\beta))\mid \beta\in\widetilde{\Hom_{d_\bullet}}(\cal K,\cal E),\xi\in \fk g/\fk p\},
$$
where $\xi^{red}$ denotes the image of $\xi$ under the quotient map $\fk g/\fk p\ra \fk g/\fk p'$. The space above is isomorphic to $T_xW_2$ by means of the map
\begin{align*}
T_xW_2&\ra T_x(W_1\times X_1)\cap T_x(X_3\times W_3),\\
(\xi,\beta)&\mapsto (\beta,(\xi^{red},\gr'\beta),(\xi,\gr\beta)).
\end{align*}
This proves the transversality.
\end{proof}

Let us now put $Z_i=T^*_{W_i}(X_{i-1}\times X_{i+1})$.
The lemma above, combined with Theorem 2.7.26 in~\cite{CG}, tells us that the projection $T^*(X_3\times X_2\times X_1)\ra T^*(X_3\times X_1)$ gives rise to an isomorphism
$$
Z_1\times_{T^* X_2} Z_3\xra{\sim} Z_2.
$$ 
Therefore, we obtain the following diagrams with cartesian squares:
$$
\begin{tikzcd}
T^* X_1 & Z_3\arrow[r,"\Psi_{3}"]\arrow[l,"\Phi_{3}"'] & T^* X_2\\
 & Z_2\arrow[ul,"\Phi_{2}"]\arrow[dr,"\Psi_{2}"']\arrow[u,"\alpha"']\arrow[r,"\beta"] & Z_1\arrow[u,"\Phi_{1}"']\arrow[d,"\Psi_{1}"] \\
 & & T^* X_3
\end{tikzcd}\qquad\qquad
\begin{tikzcd}
T^*_G X_1 & Z_{3G}\arrow[r,"\Psi_{3G}"]\arrow[l,"\Phi_{3G}"'] & T^*_G X_2\\
 & Z_{2G}\arrow[ul,"\Phi_{2G}"]\arrow[dr,"\Psi_{2G}"']\arrow[u,"\alpha_G"']\arrow[r,"\beta_G"] & Z_{1G}\arrow[u,"\Phi_{1G}"']\arrow[d,"\Psi_{1G}"] \\
 & & T^*_G X_3
\end{tikzcd}
$$
where $Z_{iG}:=Z_i\cap (T^*_G X_{i-1}\times T^*_G X_{i+1})$. 
In view of Lemma~\ref{Gysin}, the second diagram implies that
\[
\Upsilon\circ(\id\otimes \Upsilon)=((\Psi_{1G})_*\circ \Phi_{1G}^!)\circ((\Psi_{3G})_*\circ \Phi_{3G}^!)=(\Psi_{1G})_*\circ (\beta_{G})_*\circ\alpha_G^!\circ \Phi_{3G}^!=(\Psi_{2G})_*\circ \Phi_{2G}^!.
\]
In a similar way we may prove that $\Upsilon\circ(\Upsilon\otimes \id)=(\Psi_{2G})_*\circ \Phi_{2G}^!$, so that the associativity of multiplication in $A\bf{Ha}_C^0$ follows.
\end{proof}

Note that all varieties in the definition of $A\bf{Ha}_C^0$ admit a $T=\bb G_m$-action (by dilation along the fibers of cotangent and conormal bundles) and all maps we consider are $T$-equivariant. Therefore, the construction above also defines an associative product on
$$
A\bf{Ha}_C^{0,T}=\bigoplus_d A\bf{Ha}_C^{0,T}[d]=\bigoplus_d A^T(\underline{\cal Higgs}_{0,d}).
$$

\begin{exe}\label{SValgs}
	Let $C=\bb A^1$, and equip it with a natural action of weight $1$ of another torus $T'=\bb G_m$. In this case $K\bf{Ha}_C^{0,T\times T'}$ and $H\bf{Ha}_C^{0,T\times T'}$ are precisely the $K$-theoretic and cohomological Hall algebras studied in~\cite{SV1} and~\cite{SV2} respectively.
\end{exe}

%% file: shuffle.tex
\section{Global shuffle algebra}\label{globsh}
In this section we focus our attention on the algebra $A\bf{Ha}_C^{0,T}$.
In order to study its product, we will utilize the localization theorem~\ref{loc}.
Let $d_\bullet=\{0=d_0\leq d_1\leq\ldots\leq d_k=d\}$, and denote $G=G_d$, $P=P_{d_\bullet}$, $H=G_{d_\bullet}$ as before.
We will also denote by $T_d\subset H$ the maximal torus, which consists of operators diagonal with respect to the standard basis $v_1,\ldots,v_d$ of $\bbk^d$.
The Weyl group of $G$ is then naturally isomorphic to $\mathfrak S_d$, and the Weyl group of $H$ is isomorphic to $\mathfrak S_{d_\bullet}=\prod_{i}\mathfrak S_{d_i-d_{i-1}}$.

Denote $\mathbb T=T_d\times T$, and write $t_1,\ldots,t_d$ for the basis of character lattice of $T_d$ corresponding to the standard basis of $\bbk^d$.
In the same way, let $t$ be the character of $T$ of weight 1.
One can think about characters of $\bb T$ as equivariant line bundles over a point.
In this fashion, we identify $A^{\mathbb T}(pt)_{loc}:=\Frac(A^{\mathbb T}(pt))$ with $A(pt)(\!( e(t_1),\ldots, e(t_d),e(t))\!)$ (see Example~\ref{Aofpoint}).
\begin{lm}\label{Tst}
We have $(T^*\cal Quot^\circ_{0,d})^{\bb T}=\cal Quot_{0,1}^d$, where the embedding $\cal Quot_{0,1}^d\hookrightarrow\cal Quot^\circ_{0,d}$ is defined by the basis $\{v_1,\ldots,v_d\}$ associated to $T_d$.
\end{lm}
\begin{proof}
First of all, $(T^*\cal Quot^\circ_{0,d})^{T_d\times T}=(\cal Quot^\circ_{0,d})^{T_d}$.
Next, a point $(\bbk^d\otimes\Ocal\xra{\phi}\cal E)\in \cal Quot_{0,d}$ is $T_d$-stable if and only if for any $g\in T_d$ we have $\cal K_\phi:=\Ker\phi=\Ker(\phi\circ(g\otimes \id))=g.\cal K_\phi$.
We choose an integer $N$ and an ample line bundle $\cal L$ such that for any $\phi$ the sheaf $\cal L^{N}\otimes \cal K_\phi$ is generated by its global sections.
Then $\cal K_\phi$ is uniquely determined by the subspace $V:=H^0(\cal L^{N}\otimes \cal K_\phi)\subset \bbk^d\otimes H^0(\cal L^{N})$ of codimension $d$, that is by a point in certain Grassmanian.
Let us denote $H=H^0(\cal L^{N})$.
Then $\cal K_\phi$ is $T$-stable if and only if $V$ is $T$-stable.
But we know that the only subspaces stable under the torus actions are direct sums of subspaces of weight spaces.
Therefore $V=\bigoplus_{i=1}^n V\cap (\bbk v_i\otimes H)$, and thus $\cal K_\phi$ is $T$-stable iff $\cal K_\phi=\bigoplus_{i=1}^d \cal K_\phi\cap(\bbk v_i\otimes\Ocal)$.
Finally, this is equivalent to saying that $\cal E=\bigoplus_{i=1}^d \phi(\bbk v_i\otimes \Ocal)$, that is $\phi\in\cal Quot_{0,1}^d$.
\end{proof}

Recall that $\cal Quot_{0,1}\simeq C$ (see Example~\ref{q01}), and let $p_{ij}:\cal Quot_{0,1}\times \cal Quot_{0,1}\times C\ra C\times C$ denote the projection along the unnamed factor.
\begin{lm}\label{Homsheaves}
Let $\cal K, \cal E\in \ona{Coh}(\cal Quot_{0,1}\times C)\simeq \ona{Coh}(C\times C)$ be the universal families of kernels and images of quotients $\Ocal\ra \cal E$ respectively.
Then
\begin{gather*}
p_{12*}\cal Hom(p_{13}^* \cal K,p_{23}^* \cal E)\simeq \Ocal_{C\times C}(\Delta),\\
p_{12*}\cal Hom(p_{13}^* \Ocal,p_{23}^* \cal E)\simeq \Ocal_{C\times C},
\end{gather*}
where $\Delta\subset C\times C$ is the diagonal.
\end{lm}

\begin{proof}
It is easy to see that $\cal E\simeq \Ocal_\Delta$, $\cal K\simeq \Ocal(-\Delta)$.
Since $\cal K$ is locally free and $\cal E$ is a torsion sheaf over any point $\phi\in\cal Quot_{0,1}$, the higher $\Ext$-sheaves $\cal Ext^i(p_{13}^* \cal K,p_{23}^* \cal F)$ vanish for all $i>0$.
We denote $\Delta_{ij}=p_{ij}^{-1}\Delta$, and $\Delta_{123}\subset C\times C\times C$ the small diagonal. Then
\begin{align*}
p_{12*}\cal Hom(p_{13}^* \cal K,p_{23}^* \cal E)&\simeq p_{12*}((p_{13}^* \cal K)^*\otimes p_{23}^* \cal E)\simeq p_{12*}(\Ocal(\Delta_{13})\otimes \Ocal_{\Delta_{23}})\simeq p_{12*}(\Ocal_{\Delta_{23}}(\Delta_{123}))\simeq\\
&\simeq\Ocal_{C\times C}(\Delta).
\end{align*}
This proves the first equality.
For the second one, we conclude by a similar computation:
\begin{equation*}
p_{12*}\cal Hom(p_{13}^* \Ocal,p_{23}^* \cal E)\simeq p_{12*}(\Ocal_{\Delta_{23}})\simeq \Ocal_{C\times C}.  \qedhere
\end{equation*}
\end{proof}

\begin{nota}
In order to keep notation concise, for any two sheaves $\cal A,\cal B\in \ona{Coh}(\cal Quot_{0,1}\times C)$ we will write $\Hom(\cal A,\cal B)$ instead of $p_{12*}\cal Hom(p_{13}^* \cal A,p_{23}^* \cal B)$ (here the pushforward $p_{12*}$ is underived).
\end{nota}

Let $j_H:\mathscr C_{d_\bullet}\hookrightarrow T^*\cal Quot^\circ_{0,d_\bullet}$ denote the closed embedding, and let $i_H:\cal Quot_{0,1}^d\hookrightarrow T^*\cal Quot^\circ_{0,d_\bullet}$ be the inclusion of the fixed point set.
Recall that by localization theorem~\ref{loc} the map $i_H^*$ becomes an isomorphism upon tensoring with the fraction field of $A^{\bb T}(pt)$.
Consider the following composition:
$$
\rho_H=i_H^*\circ j_{H*}:A^{H\times T}(\mathscr C_{d_\bullet})\simeq A^{\bb T}(\mathscr C_{d_\bullet})^{W_{d_\bullet}}\ra A^{\bb T}(C^d)^{W_{d_\bullet}},
$$
where the isomorphism on the left is given by Proposition~\ref{GtoT}, Weyl group acts on the right-hand side by Remark~\ref{WeylonT}, and $\cal Quot_{0,1}^d$ is identified with $C^d$.
In the same way, we can define a map
$$
\rho_d=\rho_G:A^{G\times T}(\mathscr C_{d})\ra A^{\bb T}(C^d)^{W_{d}}.
$$
The goal of this section is to construct a map $\Upsilon_{loc}$ (between localized $A$-groups), such that the following diagram commutes:
$$
\begin{tikzcd}
A^{H\times T}(\mathscr C_{d_\bullet})\ar[r,"\Upsilon"]\ar[d,"\rho_H"]& A^{G\times T}(\mathscr C_d)\ar[d,"\rho_G"]\\
A^{T_d\times T}(C^d)_{loc}^{\fk S_{d_\bullet}}\ar[r,dashrightarrow,"\Upsilon_{loc}"] & A^{T_d\times T}(C^d)_{loc}^{\fk S_d}
\end{tikzcd}
$$

As in~\cite[Section 10]{SV1}, one expects $\Upsilon_{loc}$ to be some incarnation of shuffle product.
Let $z$ be a formal variable, and let $g\in \left(A^*(C\times C)(\!(t)\!)\right)(\!(e(z))\!)$ be an $A_T^*(C\times C)_{loc}$-valued formal Laurent series in $e(z)$, where we interpret the latter as a formal symbol.
We will also abuse the notations and write $g$ as a function of $z$.
For any positive $d_1$, $d_2$ with $d_1+d_2=d$, put
\[
g_{d_1,d_2}(z_1,\ldots, z_d)=\prod_{i=1}^{d_1}\prod_{j=d_1+1}^{d}g\left(\frac{z_j}{z_i}\right).
\]
Let us also fix the following set of representatives of classes in $\frak S_d/(\frak S_{d_1}\times \frak S_{d_2})$:
$$
Sh(d_1,d_2)=\left\{\sigma\in\frak S_d\mid \sigma(i)<\sigma(j)\text{ if } 1\leq i< j\leq d_1\text{ or }d_1+1\leq i<j\leq d\right\}.
$$
\begin{defi}\label{shuf}
The \emph{shuffle algebra associated to $g$} is the vector space
$$
A\bf{Sh}_g=\bigoplus_d A\bf{Sh}_g[d]=\bigoplus_{d} \left(A(C^d)(\!(e(t);e(t_1),\ldots,e(t_d))\!)\right)^{\frak S_d}
$$
equipped with the product
\begin{equation}\label{shuffleprod}
\begin{gathered}
\Xi:A\bf{Sh}_g[d_1]\times A\bf{Sh}_g[d_2]\ra A\bf{Sh}_g[d_1+d_2],\\
\Xi(f,h)=\sum_{\sigma\in Sh(d_1,d_2)}\sigma.\left(g_{d_1,d_2}(t_1,\ldots,t_{d_1+d_2})f(t_1,\ldots,t_{d_1})h(t_{d_1+1},\ldots,t_{d_1+d_2})\right).
\end{gathered}
\end{equation}
\end{defi}

The formula~(\ref{shuffleprod}) requires some explanation.
First of all, the product between $f$ and $h$ is given by the map $A(C^{d_1})\otimes A(C^{d_2})\xra{\times} A(C^{d_1+d_2})$.
Next, after replacing $z_i$'s by $t_i$'s and taking Euler classes, the function $g_{d_1,d_2}$ becomes an honest cohomology class, which then operates on the product $f\cdot h$.
Finally, the natural action of $\sigma\in Sh(d_1,d_2)$ (see Remark~\ref{WeylonT}) simultaneously permutes $t_i$'s and factors in the direct product $C^{d_1+d_2}$.

It is easy to check that this product is associative.
We will be mainly concerned with two specific choices of $g$:
\begin{itemize}
	\item the \textit{global shuffle algebra}, denoted by $A\bf{Sh}_C$, is the shuffle algebra associated to
	$$
	g_C=\frac{e(tz^{-1})e(z\Ocal(-\Delta))e(tz\Ocal(\Delta))}{e(z)};
	$$
	\item the \textit{normalized global shuffle algebra}, denoted by $A\bf{Sh}^{norm}_C$, is the shuffle algebra associated to
	$$
	g_C^{norm}=\frac{e(z\Ocal(-\Delta))e(tz\Ocal(\Delta))}{e(z)e(tz)}.
	$$
\end{itemize}
By invoking the formal group law $\star$ associated to $A$, we can deduce that both functions are Laurent series in $e(z)$ (see also discussion before Proposition~\ref{quadrel}).

\begin{rmq}
	Two algebras above are isomorphic under the map $RN=\bigoplus_d RN[d]$, where:
	\begin{gather*}
	RN[d]:A\bf{Sh}^{norm}_C[d]\ra A\bf{Sh}_C[d],\\
	f(t_1,\ldots,t_d)\mapsto \left(\prod_{i\neq j} e\left(t\frac{t_j}{t_i}\right)\right)f(t_1,\ldots,t_d).
	\end{gather*}
\end{rmq}

\begin{thm}\label{prodC}
The collection of maps $\rho_d:A^{G_d\times T}(\mathscr C_d)\ra A^{\bb T}(C^d)^{\fk S_{d}}\subset A^{\bb T}(C^d)^{\fk S_{d}}_{loc}$ defines a morphism of graded associative algebras
$$
\rho:A\bf{Ha}_C^0\ra A\bf{Sh}_C.
$$
\end{thm}
\begin{proof}
Let us first introduce some notations.
Define
\begin{gather*}
I^2_d=[1;d]\times[1;d]\subset\bbZ^2;\\
T_{\mathfrak p}=\bigcup_{i=1}^{k}[d_{i-1}+1;d_i]\times[d_{i-1}+1;d];\qquad T_{\mathfrak n}=\bigcup_{i=1}^{k-1}[d_{i-1}+1;d_{i}]\times[d_{i}+1;d].
\end{gather*}
Also, let us denote the space $W\times_{W_H}C^d$ from Proposition~\ref{ind} by $\hat C^d$.
Recall that we have a projection $s:\hat C^d\ra C^d$, given by shuffle permutations $Sh(d_1,d_2)$.

Recall the notations of Section~\ref{prod}, specifically Lemma~\ref{calc}(2).
Our proof will proceed in two steps.
First, consider the following diagram:
$$
\begin{tikzcd}
A^{H\times T}(\mathscr C_{d_\bullet})\ar[d,"j_*"]\ar[rr,"\ind_H^G"]\ar[dr] &  & A^{G\times T}(T_G^* X)\ar[d,"j_*"] \\
A^{H\times T}(T^* \cal Quot^\circ_{0,d_\bullet})\ar[d,"i^*"]\ar[r] & A^{H\times T}(\mathfrak p_-^*\times_{\mathfrak h^*}T^* \cal Quot^\circ_{0,d_\bullet})\ar[r,"\ind_H^G"]\ar[d,"i^*"] & A^{G\times T}(T^* X)\ar[d,"i^*"]\\
A^{\bbT}(C^d)^{\fk S_d}\ar[r,equal] & A^{\bbT}(C^d)^{\fk S_d}\ar[r,"s^*"] & A^{\bbT}(\hat C^d)^{\fk S_d}
\end{tikzcd}
$$
Let us denote all vertical compositions by $\rho$, leaving out the subscripts.
For any closed embedding of smooth varieties $M\subset N$, we denote by $T_M N$ the normal bundle of $M$.
\begin{lm}
	We have
	$$
	\rho\circ\ind_H^G(c)=s^*\left(\prod_{(i,j)\in T_{\mathfrak n}} e\left(t\frac{t_i}{t_j}\right)\rho(c)\right)
	$$
	for any $c\in A^{H\times T}(\mathscr C_{d_\bullet})$.
\end{lm}
\begin{proof}
Everything in this diagram commutes, except for the lower left square.
Moreover, by Proposition~\ref{pullpush} this square becomes commutative after multiplying by an appropriate Euler class.
Note that since $\mathfrak p_-^*=\mathfrak n_-^*\oplus \mathfrak h^*$, we have an isomorphism
$$
\mathfrak p_-^*\times_{\mathfrak h^*}T^* \cal Quot^\circ_{0,d_\bullet}\simeq \mathfrak n_-^*\times T^* \cal Quot^\circ_{0,d_\bullet}.
$$
This (trivial) vector bundle has the same $T$-fixed points as its zero section $T^* \cal Quot^\circ_{0,d_\bullet}$.
Therefore, the Proposition~\ref{pullpush} tells us that the required Euler class is
\begin{align*}
e\left(\left.T_{T^* \cal Quot^\circ_{0,d_\bullet}}\mathfrak n_-^*\times T^* \cal Quot^\circ_{0,d_\bullet}\right|_{C^d}\right)=e(\mathfrak n_-^*)=\prod_{(i,j)\in T_{\mathfrak n}} e\left(t\frac{t_i}{t_j}\right),
\end{align*}
and we are done.
\end{proof}

Next, consider another diagram:
$$
\begin{tikzcd}
A^{G\times T}(T_G^* X)\ar[d,"j_*"]\ar[r,"\Phi^!"] & A^{G\times T}(Z_G)\ar[d,"j_*"]\ar[r,"\Psi_*"] & A^{G\times T}(\mathscr C_d)\ar[d,"j_*"]\\
A^{G\times T}(T^* X)\ar[r,"\Phi^*"]\ar[d,"i^*"]& A^{G\times T}(Z)\ar[r,"\Psi_*"]\ar[d,"i^*"]& A^{G\times T}(T^* \cal Quot^\circ_{0,d})\ar[d,"i^*"]\\
A^{\bbT}(\hat C^d)^{\frak S_d} \ar[equal,r] & A^{\bbT}(\hat C^d)^{\frak S_d} \ar[r,"s_*"] & A^{\bbT}(C^d)^{\frak S_d}
\end{tikzcd}
$$
Let $p_{ij}:C^d\ra C\times C$ be the projection to $i$-th and $j$-th components, and $\Delta_{ij}=p_{ij}^{-1}(\Delta)$. 
\begin{lm}
We have
$$
\rho\circ\Psi_*\circ\Phi^!(c)=s_*\left(s^*\left(\prod_{(i,j)\in T_{\mathfrak n}} e\left(\frac{t_j}{t_i}\right)^{-1} e\left(\frac{t_j}{t_i}\Ocal(-\Delta_{ij})\right) e\left(t\frac{t_j}{t_i}\Ocal(\Delta_{ij})\right)\right)\rho(c)\right)
$$
for any $c\in A^{G\times T}(T_G^* X)$.
\end{lm}
\begin{proof}
Once again, all squares in this diagram commute, except for the lower right one, which commutes up to multiplication by a certain Euler class (see Lemma~\ref{Gysin}(2) and Proposition~\ref{pullpush}).
Therefore, we have:
$$
\rho\circ\Psi_*\circ\Phi^!(c)=s_*\circ\rho\left(e(T_{\hat C^d}Z)^{-1} s^*(e(T_{C^d}T^* \cal Quot^\circ_{0,d}))\cdot c\right).
$$
It is left to compute the product of Euler classes in parentheses.
We have the following chain of equalities:
\begin{align*}
e(T_{\hat C^d} & Z)^{-1} s^*(e(T_{C^d}T^* \cal Quot^\circ_{0,d}))\\
&=s^*\left(e(T_{C^d}\widetilde{\cal Quot}_{0,d_\bullet})^{-1} e(T_{\widetilde{\cal Quot}_{0,d_\bullet}}Z|_{C^d})^{-1} e(T_{C^d}\cal Quot^\circ_{0,d}) e(T_{\cal Quot^\circ_{0,d}} T^*\cal Quot^\circ_{0,d}|_{C^d})\right)\\
&=s^*\left(e(T_{C^d}\widetilde{\cal Quot}_{0,d_\bullet})^{-1} e(\mathfrak n_-)^{-1}  e\left(\left.\frac{T^*\cal Quot^\circ_{0,d}}{(T\widetilde{\cal Quot}_{0,d_\bullet})^{nilp,*}}\right|_{ C^d}\right)^{-1}\right.\times\\
&\qquad\times\left. e(T_{ C^d}\cal Quot^\circ_{0,d}) e(T^*\cal Quot^\circ_{0,d}|_{C^d})\right)\\
&=s^*\left(e(\mathfrak n_-)^{-1} e((T\widetilde{\cal Quot}_{0,d_\bullet})^{nilp,*}|_{C^d}) e(T_{ C^d}\widetilde{\cal Quot}_{0,d_\bullet})^{-1} e(T_{ C^d}\cal Quot^\circ_{0,d})\right).
\end{align*}

Let $\phi=\oplus_{i=1}^d \phi_i$ be a point in $C^d\simeq\cal Quot_{0,1}^d$, where each $\phi_i$ produces a short exact sequence $0\ra \cal K_i\ra \Ocal_C\xra{\phi_i} \cal E_i\ra 0$.
Formulas (\ref{tand})-(\ref{tandtil}) provide us with explicit expressions for tangent spaces of various $\cal Quot$-schemes:
\begin{gather*}
T_\phi \cal Quot_{0,1}^d =\bigoplus_{i=1}^d \Hom(\cal K_i,\cal E_i);\qquad T_\phi \cal Quot^\circ_{0,d} =\bigoplus_{i,j=1}^d \Hom(\cal K_j,\cal E_i);\\
T_\phi \widetilde{\cal Quot}_{0,d_\bullet} = \bigoplus_{(i,j)\in T_{\mathfrak p}} \Hom(\cal K_j,\cal E_i).
\end{gather*}
Therefore by Lemma~\ref{Homsheaves}
\begin{gather*}
(T\widetilde{\cal Quot}_{0,d_\bullet})^{nilp,*}|_{C^d}=t\bigoplus_{(i,j)\in T_{\mathfrak n}} \Hom(\cal K_j,\cal E_i)^*\simeq \bigoplus_{(i,j)\in T_{\mathfrak n}}t\frac{t_j}{t_i}\Ocal(\Delta_{ij}),\\
T_{C^d}\widetilde{\cal Quot}_{0,d_\bullet}=\bigoplus_{(i,j)\in T_{\mathfrak p},i\neq j} \Hom(\cal K_j,\cal E_i)\simeq \bigoplus_{(i,j)\in T_{\mathfrak p},i\neq j}\frac{t_i}{t_j}\Ocal(-\Delta_{ij}),\\
T_{C^d}\cal Quot^\circ_{0,d}=\bigoplus_{i\neq j} \Hom(\cal K_j,\cal E_i)\simeq \bigoplus_{i\neq j}\frac{t_i}{t_j}\Ocal(-\Delta_{ij}),
\end{gather*}
and a straightforward computation shows that
\begin{gather*}
e(\mathfrak n_-)=\prod_{(i,j)\in T_{\mathfrak n}} e\left(\frac{t_j}{t_i}\right),\\
e((T\widetilde{\cal Quot}_{0,d_\bullet})^{nilp,*}|_{ C^d}) e(T_{ C^d}\widetilde{\cal Quot}_{0,d_\bullet})^{-1} e(T_{ C^d}\cal Quot^\circ_{0,d})=\prod_{(i,j)\in T_{\mathfrak n}} e\left(\frac{t_j}{t_i}\Ocal(-\Delta_{ij})\right) e\left(t\frac{t_j}{t_i}\Ocal(\Delta_{ij})\right).
\end{gather*}
The statement of lemma follows.
\end{proof}	

Combining the results of two lemmas, we get:
\begin{align*}
\rho_G \circ m(c)=\rho\circ\Psi_*\circ\Phi^!\circ\ind_H^G(c)=s_*\circ s^*\circ(g_{C,d_\bullet}(z)\rho(c))=\Xi\circ\rho_H(c),
\end{align*}
which proves the theorem.
\end{proof}

\begin{rmq}
In order to recover the shuffle presentation in~\cite[Theorem 10.1]{SV1}, we can add an action of another torus as in Example~\ref{SValgs}.
If we denote by $q$ the $T'$-character of weight by $-1$, we get $\Ocal(\Delta)=q^{-1}$, and we obtain the desired presentation after further replacing $t$ by $qt$.
Unfortunately, we do not have a succinct explanation for this change of variables.
Morally speaking, it occurs because in the natural compactification $\bb A^2\subset \bb P^2$ the divisor at infinity is ``diagonal'', and for $T^*C\subset \bbP(T^*C)$ it is ``horizontal''.
\end{rmq}

Even though we have got an explicit formula, the morphism $\rho$ depends on the embedding of $\mathscr C_d$ into a smooth ambient variety $T^* \cal Quot_{0,d}^\circ$.
Unfortunately, the scheme $\mathscr C_d$ is highly singular; for instance, the inclusion $C^d\hookrightarrow \mathscr C_d$ is not known to be lci, so that we can not localize to $\bbT$-fixed points directly.
Still, we can do a little better.
Let $\tilde\mu$ be the composition $T^*\cal Quot^\circ_{0,d}\xra{\mu}\fk g_d^*\simeq \fk g_d\ra \fk g_d/\fk t_d$, where $\fk t_d\subset\fk g_d$ is the tangent algebra of $T_d$.
We introduce the following auxiliary variety, analogous to the one in~\cite{Knu}:
$$
\mathscr C_d^\Delta:=\tilde\mu^{-1}(0)=\mu^{-1}(\fk t_d).
$$
\begin{prop}\label{Cdiaglci}
The closed embedding $\mathscr C_d^\Delta\hookrightarrow T^*\cal Quot^\circ_{0,d}$ is a complete intersection.
\end{prop}
\begin{proof}
First of all, the statement is true for $C=\bb A^1$~\cite{Knu}.
For general $C$, it suffices to prove that $\dim \mathscr C_d^\Delta\leq \dim\cal Quot_{0,d}^\circ+\dim \fk t_d=d(d+1)$.
Let us consider the map
\begin{align*}
\sigma:T^*\cal Quot^\circ_{0,d} & \ra S^dC,\\
(\bbk^d\ra\cal E) & \mapsto \supp\cal E.
\end{align*}
For any $\nu=(\nu_1\geq\cdots\geq \nu_k)$ partition of $d$ let 
$$
S^\nu C=\{\underline{x}=\nu_1x_1+\cdots+\nu_kx_k:x_i\neq x_j\text{ for }i\neq j\}\subset S^dC
$$
Then $S^dC=\coprod_{\nu\vdash d} S^\nu C$, and this defines a stratification of $\mathscr C_d^\Delta$:
\begin{align*}
\mathscr C^\Delta_d = \coprod_{\nu\vdash d} \mathscr C^\Delta_\nu, \qquad &\mathscr C^\Delta_\nu=\mathscr C^\Delta_d\cap \sigma^{-1}(S^\nu C).
\end{align*}
Consider the restriction of $\sigma$ to these strata.
For any point $\underline{x}\in S^\nu C$, we have a $G_d$-equivariant map
\begin{align*}
\tau:\sigma^{-1}(\underline{x})\cap\mathscr C^\Delta_d & \ra\prod_i \Gr^{\nu_i}(\bbk^d),\\
(p,\beta) &\mapsto \left(H^0(\cal E|_{x_1}),\ldots,H^0(\cal E|_{x_k})\right),
\end{align*}
The image of this map is an open subset where the vector subspaces defined by points in Grassmanians do not intersect.
At each such point, the $G_d$-action induces an isomorphism between the preimage of $\tau$ and $\prod_i\mathscr C_{\nu_i}^{\mathrm n,\bullet,\Delta}$ by Example~\ref{quotloc}, where
$$
\mathscr C_{k}^{\mathrm n,\bullet,\Delta}=\{(x,y)\in (\mathfrak g_k)^2:[x,y]\in\fk t_k, x\text{ nilpotent}\}.
$$
Since in particular this applies to $C=\bb A^1$, we have
$$
\dim \mathscr C^\Delta_\nu(C) = k+\sum_i\left( \nu_i(d-\nu_i)+\dim \mathscr C_{\nu_i}^{\mathrm n,\bullet,\Delta}\right)=\dim \mathscr C^\Delta_\nu(\bb A^1)\leq d(d+1).
$$
The dimension of $\mathscr C_d^\Delta$ is bounded above by the dimensions of stratas, therefore $\dim \mathscr C_d^\Delta\leq d(d+1)$.
\end{proof}

Armed with the proposition above, let $i_d^\Delta:C^d\hookrightarrow \mathscr C_d^\Delta$, $j_d^\Delta:\mathscr C_d\hookrightarrow \mathscr C_d^\Delta$ be the natural closed embeddings, and consider the composition
$$
\varrho_d=(i_d^\Delta)^*\circ(j_d^\Delta)_*:A^{G\times T}(\mathscr C_d)\ra A^{\bbT}(C^d)^{\fk S_d}.
$$

\begin{corr}\label{prodnorm}
The collection of maps $\varrho_d$, $d\in\bbN^+$ defines a morphism of graded associative algebras
$$
\varrho:A\bf{Ha}_C^0\ra A\bf{Sh}^{norm}_C.
$$
\end{corr}
\begin{proof}
Denote the closed embedding $\mathscr C_d^\Delta\hookrightarrow T^* \cal Quot^\circ_{0,d}$ by $j'_d$.
By Corollary~\ref{pullcomm}, we have the following identity:
\begin{gather*}
\rho(c)=i_d^*\circ j_{d*}(c)=i_d^*\circ j'_{d*}\circ (j^\Delta_d)_*(c)=(i_d^\Delta)^*\circ(j_d^\Delta)_*(e_1 c)=\varrho(e_1c),\qquad e_1=e(T_{\mathscr C_d^\Delta}T^* \cal Quot_{0,d}^\circ).
\end{gather*}
Note that the map $\tilde\mu$ is $T$-equivariant.
Since $T$ contracts $T^*\cal Quot^\circ_{0,d}$ to a subvariety of $\mathscr C_d^\Delta$ and $\dim \mathscr C_d^\Delta=\dim\tilde{\mu}$, the argument similar to the one found in~\cite[Proposition 2.3.2]{GG} shows that $\tilde{\mu}$ is flat.
In particular,
$$
e(T_{\mathscr C_d^\Delta}T^* \cal Quot_{0,d}^\circ)=e((\fk g_d/\fk t_d)^*),
$$
by base change; see~\cite[B.7.4]{Ful}.
Therefore, we have
\begin{align*}
\varrho(c)=\rho(e_1^{-1}c)=\rho\left( \left(\prod_{i\neq j} e\left(t\frac{t_j}{t_i}\right)\right)^{-1}c \right)=\rho\circ RN^{-1}(c).
\end{align*}
Since both $\rho$ and $RN$ are morphisms of algebras, $\varrho$ is as well.
\end{proof}

\begin{rmq}
For any function $g$, one can equip the algebra $A\bf{Sh}_g$ with a topological coproduct, analogous to the coproduct in~\cite[Section 4]{Ne1}.
If the morphism $\rho$ is injective (see Section~\ref{injsh} for discussion and partial results), it can be used to induce a coproduct on $A\bf{Ha}^0_C$.
However, it is less clear how to construct such coproduct without using shuffle presentation.
\end{rmq}

Let us conclude this section by computing some relations in the algebra $A\bf{Sh}_g$ for an arbitrary rational function $g(z)$.
We write $g(z)=h_1(z)/h_2(z)$, where $h_1$, $h_2$ are polynomials.
Given a line bundle $\cal L$ on $C$, define a bi-infinite series
$$
E_{\cal L}(z)=\sum_{i\in\bbZ}e(\cal L)^{-i} e(z)^i,
$$
where $z$ is a formal variable, and we consider $e(z)$ to be a formal expression.
Using the formal group law $\star$ associated to $A$ (see Appendix~\ref{OBM}), we have the following equality for some $f\in A^*(pt)\llbracket u,v \rrbracket$:
\begin{gather*}
0=e(1)=e\left(\frac{z}{z}\right)=e(z)\star e(z^{-1})=e(z)+e(z^{-1})+e(z)e(z^{-1})f(e(z),e(z^{-1}))\\
\Rightarrow e(z^{-1})=-e(z)\left(1+e(z^{-1})f(e(z),e(z^{-1}))\right).
\end{gather*}

Therefore, by implicit function theorem for formal series~\cite[Exercise 5.59]{Sta} $e(z^{-1})$ can be interpreted as a formal series in $e(z)$.
In particular, $g\left(w/z\right)$ is a formal series in $e(w)$, $e(z)$.

\begin{prop}\label{quadrel}
Let $\cal L_1,\cal L_2$ be two line bundles on $C$, and $g(z)$ a rational function.
Suppose that $e(zw)$ is a polynomial in $e(z)$ and $e(w)$, and $e(z^{-1})$ is a Laurent polynomial of $e(z)$.
Then the following equality holds:
\begin{equation}\label{shufflerel}
h\left(\frac{t_1\cal L_1 w}{t_2\cal L_2 z}\right)E_{\cal L_1}(z)E_{\cal L_2}(w)=h\left(\frac{t_2\cal L_2 z}{t_1\cal L_1 w}\right)E_{\cal L_2}(w)E_{\cal L_1}(z),
\end{equation}	
where $h(z)=h_1(z)h_2(z^{-1})$, the product between $E_{\cal L_1}$ and $E_{\cal L_2}$ is taken in $A\bf{Sh}_g$, and we consider both sides as bi-infinite series in $e(z)$, $e(w)$ with coefficients in $A\bf{Sh}_g[2]$.
\end{prop}

\begin{proof}
In order to unburden the notation, denote $Z=e(z)$, $W=e(w)$, $T_i=e(t_i)$, $L_i=e(\cal L_i)$.
Let us also introduce bi-infinite series $\delta(z)=\sum_{i\in\bbZ}z^i$.
Note that for any Laurent polynomial $f(z)$ the following identity, which we call \emph{change of variables}, is satisfied:
\begin{equation}\label{delchange}
\delta(w/z)f(z)=\delta(w/z)f(w).
\end{equation}
We have:
\begin{align*}
E_{\cal L_1}(z) & E_{\cal L_2}(w) =\left(\sum_{i\in\bbZ}L_1^{-i}Z^i\right)\left(\sum_{j\in\bbZ}L_2^{-j}W^j\right)=\sum_{i,j\in\bbZ}\left(L_1^{-i}L_2^{-j}g\left(\frac{t_2}{t_1}\right)+L_1^{-j}L_2^{-i}g\left(\frac{t_1}{t_2}\right)\right)Z^iW^j\\
 & = \delta\left(L_1^{-1}Z\right)\delta\left(L_2^{-1}W\right)g\left(\frac{t_2}{t_1}\right) + \delta\left(L_1^{-1}W\right)\delta\left(L_2^{-1}Z\right)g\left(\frac{t_1}{t_2}\right)\\
 & = \left(h_2\left(\frac{t_1}{t_2}\right)h_2\left(\frac{t_2}{t_1}\right)\right)^{-1}\left(\delta\left(L_1^{-1}Z\right)\delta\left(L_2^{-1}W\right)h\left(\frac{t_2}{t_1}\right) + \delta\left(L_1^{-1}W\right)\delta\left(L_2^{-1}Z\right)h\left(\frac{t_1}{t_2}\right)\right).
\end{align*}
Therefore, the equality~(\ref{shufflerel}) is equivalent to the following:
\begin{align}\label{equivshuf}
\begin{aligned}
\delta\left(L_1^{-1}Z\right) & \delta\left(L_2^{-1}W\right)\left( h\left(\frac{t_1\cal L_1 w}{t_2\cal L_2 z}\right)h\left(\frac{t_2}{t_1}\right)-h\left(\frac{t_2\cal L_2 z}{t_1\cal L_1 w}\right)h\left(\frac{t_1}{t_2}\right) \right)\\
&= \delta\left(L_1^{-1}W\right)\delta\left(L_2^{-1}Z\right)\left( h\left(\frac{t_2\cal L_2 w}{t_1\cal L_1 z}\right)h\left(\frac{t_2}{t_1}\right)-h\left(\frac{t_1\cal L_1 z}{t_2\cal L_2 w}\right)h\left(\frac{t_1}{t_2}\right) \right).
\end{aligned}
\end{align}
However, using change of variables~(\ref{delchange}) for LHS we get:
\begin{align*}
\delta & \left(L_1^{-1}Z\right)\delta\left(L_2^{-1}W\right)\left( h\left(\frac{t_1\cal L_1 w}{t_2\cal L_2 z}\right)h\left(\frac{t_2}{t_1}\right)-h\left(\frac{t_2\cal L_2 z}{t_1\cal L_1 w}\right)h\left(\frac{t_1}{t_2}\right) \right)\\
& =\left(\delta\left(L_1^{-1}Z\right)h\left(\frac{t_1\cal L_1 w}{t_2\cal L_2 z}\right)\right)\left(\delta\left(L_2^{-1}W\right)h\left(\frac{t_2}{t_1}\right)\right)\\
&\quad-\left(\delta\left(L_1^{-1}Z\right)h\left(\frac{t_2\cal L_2 z}{t_1\cal L_1 w}\right)\right)\left(\delta\left(L_2^{-1}W\right)h\left(\frac{t_1}{t_2}\right)\right)\\
& =\delta\left(L_1^{-1}Z\right) \delta\left(L_2^{-1}W\right)\left( h\left(\frac{t_1 w}{t_2\cal L_2}\right)h\left(\frac{t_2\cal L_2}{t_1 w}\right)-h\left(\frac{t_2\cal L_2}{t_1 w}\right)h\left(\frac{t_1 w}{t_2\cal L_2}\right) \right)=0.
\end{align*}
By the same reasoning RHS is also equal to zero.
Therefore~(\ref{equivshuf}) is satisfied, and we may conclude.
\end{proof}

In particular, if we set $\cal L_1=t_1^{-1}$ and $\cal L_2=t_2^{-1}$, the equality~(\ref{shufflerel}) assumes a simpler form:
\begin{align}\label{easyquad}
h\left(\frac{w}{z}\right)E(z)E(w)=h\left(\frac{z}{w}\right)E(w)E(z),
\end{align}
where $E(z)=\sum_{i\in\bbZ}e(t_1^{-1})^{-i} e(z)^i$.

\begin{rmq}
Note that the conditions of Proposition~\ref{quadrel} are extremely restrictive.
While they are satisfied for $A=H$, already for $A=K$ the Euler class $e(z^{-1})$ is not a Laurent polynomial of $e(z)$.
However, if we denote $\tilde{e}(z)=1-e(z)$, then $\tilde{e}(z^{-1})=\tilde{e}(z)^{-1}$, and thus the proof of relations~(\ref{shufflerel}) goes through if we replace $E_{\cal L}(z)$ by
\[
\widetilde{E}_{\cal L}(z)=\sum_{i\in\bbZ}\tilde{e}(\cal L)^{-i}\tilde{e}(z)^i=\sum_{i\in\bbZ}\tilde{e}(z\cal L^{-1})^i.
\] 
In particular, we recover the identity (3.4) in~\cite{Ne2}.
This slight discrepancy is related to the fact that our $K$-theory, considered in the context of OBM homology theories, has a different set of equivariant generators from the usual $K$-theory, as defined for instance in~\cite[Chapter 5]{CG}. 
Nevertheless, the two are isomorphic up to a certain completion, see Remark~\ref{Kthdiff}.
\end{rmq}

%% file: geometry.tex
\section{Injectivity of shuffle presentation}\label{injsh}
Let $\omega = \omega_C$ be the canonical bundle of $C$.
Applying Serre duality to~(\ref{ExtInQuot}), one sees that $\bbk$-points of $\underline{\cal Higgs}_{0,d}$ are given by pairs $(\cal E,\theta)$, where $\cal E\in\ona{Coh}_{0,d}$, and $\theta\in\Hom(\cal E,\cal E\otimes\omega)$.
We call $\theta$ the \emph{Higgs field}.

\begin{defi}
A Higgs sheaf $(\cal E,\theta)$ is called \emph{nilpotent} if $\theta^N=0$ for $N$ big enough. 
\end{defi}
We denote the stack of nilpotent Higgs torsion sheaves by $\underline{\cal Higgs}_{0,d}^{nilp}$. It is a closed substack of $\underline{\cal Higgs}_{0,d}$, which has the following global quotient presentation:
\begin{align*}
\underline{\cal Higgs}_{0,d}^{nilp}&=[\mathscr C^{\bullet,\mathrm n}_d/G_d],\\
\mathscr C^{\bullet,\mathrm n}_d&=\left\{(\bbk^d\otimes\Ocal\xra{p}\cal E,\beta):p\in\cal Quot_{0,d}^\circ,\beta\in \Hom(\cal E,\cal E\otimes\omega),\beta\text{ is nilpotent}\right\},
\end{align*}
and the embedding $\underline{\cal Higgs}_{0,d}^{nilp}\hookrightarrow \underline{\cal Higgs}_{0,d}$ is given by the natural inclusion $\mathscr C^{\bullet,\mathrm n}_d\hookrightarrow\mathscr C_d$.

\begin{prop}\label{semiloc}
Let $G$ be a reductive group, and let $i:\Lambda\hookrightarrow X$ be a closed equivariant embedding of $G\times \bb G_m$-varieties. Suppose that
\begin{equation}\label{Lambdacond}
\left\{x\in X:\overline{G.x}\cap \Lambda\neq\emptyset\right\}=\Lambda,
\end{equation}
and assume that for any $x\in X$ the intersection $\overline{\bb G_m.x}\cap \Lambda$ is not empty. Then the pushforward along $i$ induces an isomorphism of localized $A$-groups:
$$
A^{G\times \bb G_m}(\Lambda)\otimes_{A_{\bb G_m}(pt)} \Frac(A_{\bb G_m}(pt))\xra{\sim} A^{G\times \bb G_m}(X)\otimes_{A_{\bb G_m}(pt)} \Frac(A_{\bb G_m}(pt)).
$$
\end{prop}
\begin{proof}
Note that our assumptions imply $X^{G\times \bb G_m}\subset \Lambda$. Furthermore, by Proposition~\ref{GtoT} we can assume that $G$ is a torus. Take $x\in X$, and let $(g,t)\in G\times \bb G_m$ lie in the stabilizer of $x$. Suppose that $t$ has infinite order. Then $t^{-1}.x=g.x$, and by consequence $\overline{G.x}\cap \Lambda=\overline{(G\times T).x}\cap \Lambda$ is non-empty, so that $x\in \Lambda$. We conclude that for any $x\in X\setminus\Lambda$ there exists a positive number $N(x)$ such that $Stab(x)\subset G\times\mu_{N(x)}$. Since torus actions on finite type schemes always possess finitely many stabilizers, one can assume that $N=N(x)$ does not depend on $x$. Let us consider the following character of $G\times\bb G_m$:
\begin{align*}
\chi:G\times \bb G_m & \ra \bb G_m,\\
(g,t) & \mapsto t^N.
\end{align*}
It is clear that for any $x\in X\setminus\Lambda$ one has $Stab_{G\times \bb G_m}(x)\subset\Ker\chi$. Therefore by Proposition~\ref{partloc} one has an isomorphism
$$
A^{G\times \bb G_m}(\Lambda)[c_1(\chi)^{-1}]\xra{\sim}A^{G\times \bb G_m}(X)[c_1(\chi)^{-1}],
$$
which implies the desired result.
\end{proof}

\begin{corr}\label{isonilp}
The natural map
$$
A^{T}(\underline{\cal Higgs}_{0,d}^{nilp})\ra A^{T}(\underline{\cal Higgs}_{0,d})
$$
becomes an isomorphism upon tensoring by $\Frac(A_T(pt))$.
\end{corr}
\begin{proof}
Take $X=\mathscr C_d$, $\Lambda=\mathscr C^{\bullet,\mathrm n}_d$. Any point in $(p,\beta)\in X\setminus \Lambda$ is separated from $\Lambda$ by the characteristic polynomial of $\beta$. Therefore condition (\ref{Lambdacond}) is verified. Moreover, the action of $T$ contracts any Higgs field to zero, that is for any $x\in X$ the intersection $\overline{T.x}\cap\cal Quot^\circ_{0,d}\subset\overline{T.x}\cap\Lambda$ is not empty. We conclude by invoking Proposition~\ref{semiloc}.
\end{proof}

From now on till the end of the section we suppose that $\bbk=\bbC$, and $A$ is the usual Borel-Moore homology $H$.
\begin{thm}\label{locinj}
The group $H_*^{T}(\underline{\cal Higgs}_{0,d}^{nilp})$ is torsion-free as an $H^*_{G\times T}(pt)$-module.
\end{thm}
\begin{corr}\label{shufflefaith}
The morphism $\rho:H\bf{Ha}_C^{0,T}\ra H\bf{Sh}_C$ of Theorem~\ref{prodC} becomes injective after tensoring by $\Frac(H_T(pt))$.
\end{corr}

We will prove Theorem~\ref{locinj} in three steps:
\begin{enumerate}
	\item shrink localizing set;
	\item reduce the question to $\underline{\cal Coh}_{0,d}\subset \underline{\cal Higgs}_{0,d}^{nilp}$;
	\item explicit computation for $\underline{\cal Coh}_{0,d}$.
\end{enumerate}

First of all, let $I\subset H^*_{G\times T}(pt)$ be the ideal of functions $f(t_1,\ldots,t_d,t)$ such that $f(0,\ldots,0,t)=0$. It is clear that $\bbQ[t]\cap I=\{0\}$. For any $H^*_{G\times T}(pt)$-module $M$ we denote by $M_{loc,I}$ its localization with respect to $H^*_{G\times T}(pt)\setminus I$.

Since $H^*_T(pt)\setminus\{0\}\subset H^*_{G\times T}(pt)\setminus I$, the localization theorem~\ref{loc} yields an isomorphism
$$
H^{G_d\times T}_*((\mathscr C^{\bullet,\mathrm n})^T)_{loc,I}\xra{\sim}H^{G_d\times T}_*(\mathscr C^{\bullet,\mathrm n})_{loc,I}.
$$
Note, however, that
$$
H^{G_d\times T}_*((\mathscr C^{\bullet,\mathrm n})^T)_{loc,I}\simeq H^{G_d\times T}_*(\cal Quot_{0,d}^\circ)_{loc,I}\simeq H_*^T(\underline{\cal Coh}_{0,d})_{loc,I},
$$
where the $T$-action on the latter stack is trivial. By Poincaré duality and a result of Laumon~\cite[Théorème 3.3.1]{La2},
$$
H_*(\underline{\cal Coh}_{0,d})\simeq H^*(\underline{\cal Coh}_{0,d})\simeq S^d(H^*(C\times \mathrm{B}\bb G_m))=S^d(H^*(C)[z]).
$$
The $H^*_{G}(pt)$-action on the latter space is given as follows. The natural free $\bbQ[z]$-module structure on $H^*(C)[z]$ defines embeddings of algebras:

\begin{equation}\label{Cohaction}
\begin{tikzcd}
H^*_{G}(pt)\arrow[r,equal] & S^d(\bbQ[z])\arrow[r,hook]\arrow[d,hook] & S^d(H^*(C)[z])\arrow[d,hook] \\
 & T^d(\bbQ[z])\arrow[r,hook] & T^d(H^*(C)[z])
\end{tikzcd}
\end{equation}

The upper horizontal map defines us the desired action. Note that since lower horizontal map turns $T^d(H^*(C)[z])$ into a free $T^d(\bbQ[z])$-module, the same is true for a $H^*_G(pt)$-module $S^d(H^*(C)[z])$. In particular, this implies that the $H^*_{G\times T}(pt)$-module $H_*^T(\cal Coh_{0,d})_{loc,I}$ is torsion-free.
Putting together the arguments above, we get the following result:
\begin{prop}\label{locI}
The group $H_*^{T}(\underline{\cal Higgs}_{0,d}^{nilp})_{loc,I}$ is torsion-free as a $H^*_{G\times T}(pt)$-module.
\end{prop}

Next, let us break down the stack of nilpotent Higgs sheaves into more manageable pieces. Recall the following stratification of $\underline{\cal Higgs}_{0,d}^{nilp}$ due to Laumon~\cite{La}.

\begin{defi}
	For any partition $\nu\vdash d$, $\nu=(1^{\nu_1}2^{\nu_2}\ldots)$, let
	$$
	\underline{\cal Nil}_\nu=\left\{(\cal E,\theta)\in \underline{\cal Higgs}_{0,d}^{nilp}:\deg\left(\Ker\theta^i/(\Im\theta\cap\Ker\theta^{i}+\Ker\theta^{i-1})\right)=\nu_i\text{ for all }i\right\}.
	$$
\end{defi}

These substacks are locally closed, disjoint, and cover the whole stack of nilpotent Higgs sheaves:
$$
\underline{\cal Higgs}_{0,d}^{nilp}=\bigsqcup_{\nu\vdash d}\underline{\cal Nil}_\nu.
$$

\begin{prop}\label{Nilfib}
For any $\nu\vdash d$, there exists a stack vector bundle $p_\nu:\underline{\cal Nil}_\nu\ra \prod_i \underline{\cal Coh}_{0,\nu_i}$.
\end{prop}
\begin{proof}
See the proof of Proposition 5.2 in~\cite{SM}.
\end{proof}

As a consequence of this proposition, for any $\nu\vdash d$ we have an isomorphism
\begin{align}\label{vecNil}
p^*_\nu:H_*^{T}(\prod_i \underline{\cal Coh}_{0,\nu_i})\xra{\sim}H_*^{T}(\underline{\cal Nil}_\nu).
\end{align}

Before continuing with the rest of the proof, let us recall some basic properties of weight filtration from~\cite{De2,De3} and references therein.
For any algebraic variety $X$, Deligne constructed the \textit{weight filtration} $W_k$ on cohomology groups $H^i(X)$.
This filtration is compatible with K\"{u}nneth isomorphisms.
Moreover, it is \textit{strictly} compatible with natural maps, in the sense that an element in target group belongs to $W_k$ if and only if it is an image of an element in $W_k$.
We say that the weight filtration on $H_i(X)$ is \textit{pure of weight $i$} if $W_{i-1}H^i(X)=0$, $W_iH^i(X)=H^i(X)$.
This is the case for any smooth projective variety $X$, as well as for classifying spaces $\mathrm{B}G$.
Weight filtration also exists for Borel-Moore homology and in equivariant setting; it can thus be extended to homology groups of quotient stacks.

\begin{lm}\label{lesweight}
Let $X$ be a $G$-variety, $U$ an open $G$-subvariety, and $Z=X\setminus U$.
Suppose that homology groups $H_i(U)$, $H_i(Z)$ are pure of weight $i$ for all $i$.
Then the long exact sequence in Borel-Moore homology splits into short exact sequences
\[
0\to H^G_i(Z)\to H^G_i(X)\to H^G_i(U)\to 0
\]
and $H^G_i(X)$ is pure of weight $i$ for all $i$.
\end{lm}
\begin{proof}
The weight filtration is strictly compatible with all maps in the long exact sequence.
In particular, since $H^G_i(U)$ and $H^G_{i-1}(Z)$ are pure and have different weights, the connecting homomorphism vanishes.
Furthermore, by strict compatibility we have the following short exact sequences for each $j$:
\[
0\to W_jH^G_i(Z)\to W_jH^G_i(X)\to W_jH^G_i(U)\to 0
\]
By purity of outer terms we have $W_{i-1}H^G_i(X)=0$, $W_iH^G_i(X)=H^G_i(X)$, so that $H^G_i(X)$ is pure of weight $i$.
\end{proof}

Let us choose a total order $\prec$ on the set of partitions of $d$ such that for any two partitions $\nu,\nu'$ the inclusion $\underline{\cal Nil}_\nu\subset \overline{\underline{\cal Nil}_{\nu'}}$ implies $\nu\prec \nu'$. Denote 
\begin{align*}
\Nil_{\prec\nu}=\coprod_{\nu'\prec \nu} \Nil_{\nu'};\qquad \Nil_{\preceq\nu}=\Nil_{\prec\nu}\sqcup \Nil_\nu.
\end{align*}
For each $\nu$, this order gives rise to a long exact sequence in Borel-Moore homology:
\begin{align}\label{lesNil}
\cdots\ra H_k^{T}(\Nil_{\prec\nu})\ra H_k^{T}(\Nil_{\preceq\nu})\ra H_k^{T}(\Nil_\nu)\ra\cdots
\end{align}

The homology groups $H_*(\underline{\cal Coh}_{0,d})$ comprise the $\fk S_d$-invariant part of $H_*(C\times \mathrm{B}\bb G_m)^{\otimes d}$.
Since the latter group has pure weight filtration, the same is true for the former as well, and by (\ref{vecNil}) for $H_*^{T}(\Nil_\nu)$ for any $\nu$.
A straightforward induction on $\nu$ using Lemma~\ref{lesweight} shows that both $H_*^{T}(\Nil_{\prec\nu})$ and $H_*^{T}(\Nil_{\preceq\nu})$ are also pure.
Additionally, the long exact sequence~(\ref{lesNil}) splits into short exact sequences:
$$
0 \ra H_k^{T}(\Nil_{\prec\nu})\ra H_k^{T}(\Nil_{\preceq\nu})\ra H_k^{T}(\Nil_\nu)\ra 0
$$
These short exact sequences yield a filtration $F_\bullet$ of $H_*^{T}(\underline{\cal Higgs}_{0,d}^{nilp})$, such that $F_i$ is of the form $H_k^{T}(\Nil_{\prec\nu})$, and
\begin{align}\label{grNil}
\gr_F H_*^{T}(\underline{\cal Higgs}_{0,d}^{nilp})=\bigoplus_{\nu\vdash d}H_*^{T}(\Nil_{\nu}).
\end{align}

\begin{proof}[Proof of Theorem~\ref{locinj}]
The desired statement is equivalent to injectivity of the localization morphism, which can be written as a composition of two successive localizations:
$$
H_*^{T}(\underline{\cal Higgs}_{0,d}^{nilp})\xra{l_I} H_*^{T}(\underline{\cal Higgs}_{0,d}^{nilp})_{loc,I}\ra H_*^{T}(\underline{\cal Higgs}_{0,d}^{nilp})_{loc}.
$$ 
The second map being injective by Proposition~\ref{locI}, it suffices to prove injectivity of $l_I$. Taking into account isomorphisms (\ref{grNil}), (\ref{vecNil}), and the fact that $H_*^T(\underline{\cal Coh}_{0,d})$ is a free $H^*_T(pt)$-module, we are reduced to proving the injectivity of
$$
l_{\fk m}:H_*\left(\prod_i\underline{\cal Coh}_{0,\nu_i}\right) \ra H_*\left(\prod_i\underline{\cal Coh}_{0,\nu_i}\right)_{loc,\fk m}
$$
for any $\nu\vdash d$, where we localize at $\fk m\subset H^*_G(pt)$ --- maximal homogeneous ideal with respect to homological grading. Since $H_*(\prod_i\underline{\cal Coh}_{0,\nu_i})$ is evidently a graded $H^*_G(pt)$-module, the annihilator of $c$ is a graded ideal for any $c\in H_*(\prod_i\underline{\cal Coh}_{0,\nu_i})$, therefore fully contained in $\fk m$. This proves injectivity of $l_{\fk m}$ and concludes the proof of the theorem.
\end{proof}

\begin{corr}
The identity~(\ref{shufflerel}) holds in $H\bf{Ha}_C^{0,T}$.
\end{corr}
\begin{proof}
Note that we have the following chain of isomorphisms:
\begin{align*}
H^*_{G_d\times T}(T^*\cal Quot^\circ_{0,d}) & \simeq H^*_{G_d\times T}(\cal Quot^\circ_{0,d})\simeq H^*(\underline{\cal Coh}_{0,d})\otimes H^*_T(pt)\\
& \simeq S^d(H^*(\underline{\cal Coh}_{0,1}))[t]\simeq (H^*(C^d)[t;t_1,\ldots,t_d])^{\fk S_d}.
\end{align*}
Therefore, the closed embedding $\mathscr C_d\hookrightarrow T^*\cal Quot^\circ_{0,d}$ defines a $(H^*(C^d)[t;t_1,\ldots,t_d])^{\fk S_d}$-module structure on $H\bf{Ha}_C^{0,T}[d]$, compatible with $\rho$. In particular, the identity~(\ref{shufflerel}) is well-defined in $H\bf{Ha}_C^{0,T}$.

Combining Corollary~\ref{isonilp} and Theorem~\ref{locinj}, we see that the natural morphism
$$
H_*^T(\underline{\cal Higgs}_{0,d}^{nilp})\ra H_*^T(\underline{\cal Higgs}_{0,d})
$$
is injective. Furthermore, it easily follows from the construction in Section~\ref{prod} that the space $\bigoplus_d H_*^T(\underline{\cal Higgs}_{0,d}^{nilp})$ is a subalgebra in $H\bf{Ha}_C^{0,T}$. Since $H_*^T(\underline{\cal Higgs}_{0,1}^{nilp})= H_*^T(\underline{\cal Coh}_{0,1})\simeq H_*^T(\underline{\cal Higgs}_{0,1})$, all operators intervening in identity~(\ref{shufflerel}) belong to $\bigoplus_d H_*^T(\underline{\cal Higgs}_{0,d}^{nilp})$, and we conclude by Corollary~\ref{shufflefaith} and Proposition~\ref{quadrel}.
\end{proof}

\begin{rmq}
Note that for $A=H$ the function $g_C^{norm}$ takes the following form:
\begin{align*}
g_C^{norm}(z)=\frac{e(z\Ocal(-\Delta))e(tz\Ocal(\Delta))}{e(z)e(tz)}=\frac{(z-\Delta)(z+t+\Delta)}{z(z+t)}=1-\frac{\Delta(t+\Delta)}{z(t+z)},
\end{align*}
where $\Delta\in H^2(C\times C)$ is the class of diagonal. Using this explicit expression, we can rewrite the identity~(\ref{shufflerel}) as a set of relations. In particular, for $\cal L_1=t_1^{-1}$, $\cal L_2=t_2^{-1}$, we get
\begin{align*}
[e_i,e_j]_3-(t^2+\Delta(t+\Delta))[e_i,e_j]_1+t\Delta(t+\Delta)(e_ie_j+e_je_i)=0
\end{align*}
for any $i,j\in\bbZ$, where $e_i=t_1^i$, and
$$
[e_i,e_j]_n:=\sum_{k=0}^{n}(-1)^k {{n}\choose{k}}(e_{i+k}e_{j+n-k}-e_{j+n-k}e_{i+k}).
$$
For general $\cal L_1$ and $\cal L_2$ the relations become more complicated.
\end{rmq}

\begin{conj}\label{conjAll}
For any oriented Borel-Moore theory $A$, the morphism $\rho:A\bf{Ha}_C^{0,T}\ra A\bf{Sh}_C$ of Theorem~\ref{prodC} is injective.
\end{conj}

We hope to prove Conjecture~\ref{conjAll} in subsequent work by analyzing the action of $A\bf{Ha}_C^{0,T}$ on modules $A\mathscr M^T_n$ for varying $n$, defined in next section.

%% file: modules.tex
\section{Moduli of stable Higgs triples}\label{modtri}
In this section we introduce an action of $A\mathbf{Ha}^0_C$ on the $A$-theory of certain varieties, which can be regarded as generalization of the Hilbert schemes of points on $T^*C$ (see Section~\ref{negu}).

We start with the stack $\underline{\cal Coh}^{\leftarrow \cal F}_{0,d}$, where $\cal F\in\ona{Coh}C$ is a fixed coherent sheaf on $C$.
The following proposition seems to be well-known (compare to~\cite[Theorem 4.1.(i)]{HL} and the entirety of~\cite{GK}), but we did not manage to find a precise reference.

\begin{prop}\label{tangentframe}
Let $p=(\cal E,\alpha)\in \underline{\cal Coh}^{\leftarrow \cal F}_{0,d}$ be a pair.
Then the tangent space $T_p\underline{\cal Coh}^{\leftarrow \cal F}_{0,d}$ at $p$ is naturally isomorphic to $\bbHom(\cal F\xra{\alpha}\cal E,\cal E)$.
\end{prop}

\begin{proof}
Let $D=\Spec\bbk[\epsilon]/\epsilon^2$.
By definition, the tangent space $T_p\underline{\cal Coh}^{\leftarrow \cal F}_{0,d}$ is given by the space of maps $D\ra\underline{\cal Coh}^{\leftarrow \cal F}_{0,d}$, which restrict to $p$ at origin.
Again, by definition
$$
\{D\ra\underline{\cal Coh}^{\leftarrow \cal F}_{0,d}\}=\{\cal F[\epsilon]\xra{\tilde{\alpha}} \tilde{\cal E}:\tilde{\cal E}\in \Ocal_C[\epsilon]-\ona{mod};\quad\tilde{\cal E}\text{ flat over }D;\quad \tilde{\alpha}\ona{mod} \epsilon=\alpha\},
$$
and since the infinitesimal deformations of a coherent sheaf over a scheme are given by its self-extensions, we see that maps $D\ra\underline{\cal Coh}^{\leftarrow \cal F}_{0,d}$ are parametrized by diagrams of the form

$$
\begin{tikzcd}
0\arrow{r}&\cal F\arrow{r}\arrow{d}{\alpha}&\cal F\oplus\cal F\arrow{d}{\tilde\alpha}\arrow{r}&\cal F\arrow{r}\arrow{d}{\alpha}&0\\
0\arrow{r}&\cal E\arrow{r}&\tilde{\cal E}\arrow{r}&\cal E\arrow{r}&0
\end{tikzcd}
$$
Splitting off $\cal F$ on the left, this data is equivalent to the following diagram:
\begin{equation}\label{extbr}
\begin{tikzcd}
&&\cal F\arrow{d}{\tilde\alpha}\arrow{dr}{\alpha}&&\\
0\arrow{r}&\cal E\arrow{r}&\tilde{\cal E}\arrow{r}&\cal E\arrow{r}&0
\end{tikzcd}
\end{equation}

On the other hand, let us compute $\bbHom(\cal F\xra{\alpha}\cal E,\cal E)$.
Fix an injective resolution $\cal I^\bullet$ of $\cal E$.
Applying $\Hom$-functor, one produces a double complex
$$
\begin{tikzcd}
\Hom(\cal E,\cal I^0)\ar[r,"d_0\circ -"]\ar[d,"-\circ\alpha"]&\Hom(\cal E,\cal I^1)\ar[r,"d_1\circ -"]\ar[d,"-\circ\alpha"]&\Hom(\cal E,\cal I^2)\ar[r]\ar[d,"-\circ\alpha"]&\cdots\\
\Hom(\cal F,\cal I^0)\ar[r,"d_0\circ -"]&\Hom(\cal F,\cal I^1)\ar[r,"d_1\circ -"]&\Hom(\cal F,\cal I^2)\ar[r]&\cdots
\end{tikzcd}
$$
Taking cohomology of its total complex, we get
\begin{align*}
\bbHom(\cal F\xra{\alpha}\cal E,\cal E)&=\frac{\{(f,g)\in\Hom(\cal E,\cal I^1)\oplus\Hom(\cal F,\cal I^0):d_1\circ f=0, d_0\circ g=f\circ\alpha\}}{\{(d_0\circ h,-h\circ\alpha):h\in\Hom(\cal E,\cal I_0)\}}\\
&=\frac{\{(f,g)\in\Hom(\cal E,\Ker d_1)\oplus\Hom(\cal F,\cal I^0):d_0\circ g=f\circ\alpha\}}{\{(d_0\circ h,-h\circ\alpha):h\in\Hom(\cal E,\cal I_0)\}}.
\end{align*}
But by Yoneda construction, pullback of the extension $0\ra\cal E\ra \cal I^0\ra\Ker d_1\ra 0$ gives a bijection between self-extensions of $\cal E$ and morphisms $\cal E\ra \Ker d_1$ up to the ones factorizing through $\cal I_0$.
Associating to every element $\rho\in\Ext^1(\cal E,\cal E)$ the corresponding extension $0\ra\cal E\ra\cal E_\rho\xra{\pi_\rho}\cal E\ra 0$, we get
\begin{align*}
\bbHom(\cal F\xra{\alpha}\cal E,\cal E)&=\{(\rho\in\Ext^1(\cal E,\cal E),g:\cal F\ra\cal E_\rho):g\circ\pi_\rho=\alpha\},
\end{align*}
which is precisely the space of infinitesimal deformations of $(\cal E,\alpha)$ as seen above in the diagram~(\ref{extbr}).
\end{proof}

\begin{defi}
A \textit{Higgs triple} of rank $r$, degree $d$ and frame $\cal F$ is the data $(\cal E,\alpha,\theta)$ of a coherent sheaf $\cal E\in\ona{Coh}_{r,d}C$, a map $\alpha:\cal F\ra\cal E$, and an element $\theta\in\bbExt^1(\cal E,(\cal F\xra{\alpha}\cal E)\otimes \omega)$. Given two Higgs triples $T_1=(\cal E_1,\alpha_1,\theta_1)$, $T_2=(\cal E_2,\alpha_2,\theta_2)$, a \emph{morphism} from $T_1$ to $T_2$ is a map $f\in\Hom(\cal E_1,\cal E_2)$ such that $\alpha_2=f\circ\alpha_1$, and $\theta_2\circ f=f\circ \theta_1$.
\end{defi}

Thanks to Serre duality and Proposition~\ref{tangentframe}, the $\bbk$-points of the stack $T^*\underline{\cal Coh}^{\leftarrow \cal F}_{0,d}$ are precisely Higgs triples of rank 0, degree $d$ and frame $\cal F$.
More generally, its $T$-points for any scheme $T$ are given by families of triples $(\cal E_T,\alpha_T,\theta_T)$, where $\cal E_T$ is flat over $T$. 

\begin{defi}\label{trstable}
A Higgs triple is called \textit{stable} if there is no subsheaf $\cal E'\subset \cal E$ such that:
\begin{itemize}
\item $\Im\alpha\subset \cal E'$, and
\item $a(\theta)\in \Im(b)$, where $a$, $b$ are the maps below, induced by inclusion $\cal E'\subset \cal E$:
\begin{equation}\label{substab}
\bbExt^1(\cal E,(\cal F\xra{\alpha}\cal E)\otimes \omega)\xra{a} \bbExt^1(\cal E',(\cal F\xra{\alpha}\cal E)\otimes \omega)\xla{b}\bbExt^1(\cal E',(\cal F\xra{\alpha}\cal E')\otimes \omega).
\end{equation}
\end{itemize}
\end{defi}

In other words, a triple is stable if the image of $\alpha$ generates $\cal E$ under $\theta$.
We denote by $\left( T^*\underline{\cal Coh}^{\leftarrow \cal F}_{0,d} \right)^{st}\subset T^*\underline{\cal Coh}^{\leftarrow \cal F}_{0,d}$ the substack of stable Higgs triples of rank 0.

Recall that for an abelian category $\mathscr C$ of homological dimension 1 every complex in the bounded derived category $\mathcal D^b(\mathscr C)$ is quasi-isomorphic to the direct sum of its shifted cohomology objects (see~\cite[Proposition 2.1.2]{Ka}). Because of this observation, we can alternatively write Higgs triples as quadruples
\begin{align*}
(\cal E, \alpha, \theta_e, \theta_h): \quad &\cal E\in\ona{Coh}_{r,d}, \quad \alpha:\cal F\to\cal E,\\
&\theta_e\in \Ext^1(\cal E,\Ker\alpha\otimes \omega),\quad\theta_h\in\Hom(\cal E,\Coker\alpha\otimes\omega).
\end{align*}

\begin{lm}\label{stabinquad}
A Higgs triple is stable if and only if there are no non-trivial subsheaves $\cal E'\subset \cal E$ such that $\Im\alpha\subset \cal E'$ and $\theta_h(\cal E')\subset (\cal E'/\Im\alpha)\otimes \omega$.
\end{lm}
\begin{rmq}
	Note that the stability condition does not depend on $\theta_e$ in this form.
\end{rmq}
\begin{proof}
Replacing the complex $\cal F\xra{\alpha} \cal E$ by the sum of its kernel and cokernel, the diagram~(\ref{substab}) splits into two:
\begin{gather*}
\Ext^1(\cal E,\Ker\alpha\otimes \omega)\xra{a_e} \Ext^1(\cal E',\Ker\alpha\otimes \omega)\xla{b_e}\Ext^1(\cal E',\Ker\alpha\otimes \omega),\\
\Hom(\cal E,\Coker\alpha\otimes \omega)\xra{a_h} \Hom(\cal E',\Coker\alpha\otimes \omega)\xla{b_h}\Hom(\cal E',(\cal E'/\Im\alpha)\otimes \omega).
\end{gather*}
Note that the map $b_e$ is an isomorphism. Therefore, the condition $a(\theta)\in \Im(b)$ is equivalent to $a_h(\theta_h)\in \Im(b_h)$, that is $\theta_h(\cal E')\subset (\cal E'/\Im\alpha)\otimes \omega$.
\end{proof}

We say that a morphism of triples is a \emph{quotient}, if the underlying map of sheaves is surjective.
\begin{lm}\label{quotst}
Let $T=(\cal E,\alpha,\theta)$, $T'=(\cal E',\alpha',\theta')$ be two triples, together with a quotient map $\pi:\cal E\ra \cal E'$. If $T$ is stable, then $T'$ is stable as well.
\end{lm}
\begin{proof}
Suppose that $T'$ is not stable, that is there exists a subsheaf $\cal E'_1\subset \cal E'$, such that $\Im\alpha'\subset \cal E'_1$ and $\theta'(\cal E'_1)\subset (\cal E'_1/\Im\alpha')\otimes \omega$.
Consider its preimage $\cal E_1:=\pi^{-1}(\cal E'_1)$. 
Since $\alpha'=\pi\circ\alpha$ by definition, we get $\Im\alpha\subset \cal E_1$.

Let us denote $U=\Ker \pi$.
Since $\theta'\circ\pi=\pi\circ\theta$, we have
\begin{equation}\label{stquot1}
\theta_h(U)\subset (U/(U\cap \Im\alpha))\otimes \omega\subset (\cal E_1/\Im\alpha)\otimes \omega.
\end{equation}
Moreover, since $\cal E_1/U\simeq \cal E'_1$, we get
\begin{equation}\label{stquot2}
\theta_h(\cal E_1)/(U+\Im\alpha)=\theta'_h(\cal E'_1)\subset (\cal E'_1/\Im\alpha')\otimes \omega=(\cal E_1/(U+\Im\alpha))\otimes \omega.
\end{equation}
Combining (\ref{stquot1}) and (\ref{stquot2}), we conclude that $\theta_h(\cal E_1)\subset (\cal E_1/\Im\alpha)\otimes \omega$.
Therefore $\cal E_1\subset \cal E$ is a destabilizing subsheaf by Lemma~\ref{stabinquad}, and thus instability of $T'$ implies instability of $T$.
\end{proof}

The following lemma can be viewed as an avatar of Schur's lemma.
\begin{lm}\label{noauto}
	Stable Higgs triples have no non-trivial automorphisms.
\end{lm}
\begin{proof}
	Let $T=(\cal E,\alpha,\theta)$ be a stable Higgs triple, and suppose $f\in\End(\cal E)$ induces an automorphism of $T$. We pose $\cal E'=\Ker(f-\id_{\cal E})\subset \cal E$. Since $f\circ \alpha=\alpha$, we have $\Im\alpha\subset\cal E'$. Moreover, by definition $f|_{\cal E'}=\id_{\cal E'}$. Therefore the equality $\theta_h\circ f=f\circ \theta_h$ implies that $(f-\id_{\Coker\alpha})\circ \theta_h|_{\cal E'}=0$, and thus $\theta_h(\cal E')\subset \cal E'/\Im\alpha$. This means that $\cal E'$ is a destabilizing subsheaf, which can only happen for $\cal E'=\cal E$. Thus $f=\id_{\cal E}$.
\end{proof}

From now on, we will only consider Higgs triples of rank 0.

\begin{thm}\label{moduli}
Let $\cal F$ be a locally free sheaf on $C$.
Then the moduli stack of stable Higgs triples of rank $0$, degree $d$ and frame $\cal F$ is represented by a smooth quasi-projective variety $\mathscr B(d,\cal F)$.
In particular, $\mathscr B(d,\Ocal)\simeq \ona{Hilb}_d T^*C$.
\end{thm}

We will prove this theorem in Section~\ref{negu} by realizing $\mathscr B(d,\cal F)$ as a moduli of torsion-free sheaves on a ruled surface.
It is also possible to prove it directly by relating stability of Higgs triples to Mumford's GIT stability~\cite{MFK} on an atlas of $T^*\underline{\cal Coh}_{0,d}^{\leftarrow\cal F}$, which was the approach used in a previous version of this paper.

Let us further assume that $\cal F\simeq \bbk^n\otimes\Ocal$ is a trivial sheaf of rank $n$.
To simplify the notation, we will write\footnote{$\cal Br$ stands for Bradlow, as in ``Bradlow pairs''~\cite{Th}}
\[
\underline{\cal Br}_{0,d}^n:=\underline{\cal Coh}^{\leftarrow \bbk^n\otimes \Ocal}_{0,d},\qquad \mathscr{B}(d,n):=\mathscr{B}(d,\bbk^n\otimes\Ocal).
\]
In the remainder of this section we will produce an action of $A\bf{Ha}_C^0$ on the $A$-theory of moduli spaces $\mathscr B(d,n)$.
In order to do this, we will use the general machinery from the beginning of Section~\ref{prod}.

Let $d_\bullet=\{0=d_0\leq d_1\leq\ldots\leq d_k=d\}$, and $F$ a vector space of dimension $n$. As before, we note $G=G_d$, $P=P_{d_\bullet}$. We put:
$$
\tilde{Y}=\Hom(F,\bbk^{d}/\bbk^{d_{k-1}})\times \cal Quot_{0,d_\bullet}^\circ,\quad \tilde V=\Hom(F,\bbk^{d})\times \widetilde{\cal Quot}_{0,d_\bullet},\quad \tilde X'=\Hom(F,\bbk^{d})\times\cal Quot^\circ_{0,d}.
$$ 
We have a natural closed embedding $\tilde{g}:\tilde V\hookrightarrow \tilde X'$ and an affine fibration $\tilde{f}:\tilde V\twoheadrightarrow \tilde Y$. The formula~(\ref{pushpull}) gives rise to a map in $A$-theory
$$
m_{d_\bullet}:A_*^H(T^*_H\tilde Y)\ra A_*^G(T^*_G\tilde X').
$$
For instance, in the case $k=2$ we get a map:
$$
m_{d_1,d_2}:A\bf{Ha}_C^0[d_1]\otimes A_*(T^*\underline{\cal Br}_{0,d-d_1}^n) \hookrightarrow A_*(\underline{\cal Higgs}_{0,d_1}\times T^*\underline{\cal Br}_{0,d-d_1}^n)\ra A_*(T^*\underline{\cal Br}_{0,d}^n).
$$
Collecting these maps for all $d_1,d_2$, we get a map
$$
m:A\bf{Ha}_C^0\otimes A\bf M_n\ra A\bf M_n,
$$
where $A\bf M_n=\bigoplus_{d} A_*(T^*\underline{\cal Br}_{0,d}^n)$.
\begin{prop}
The map $m$ defines an $A\bf{Ha}_C^0$-module structure on $A\bf M_n$.
\end{prop}
\begin{proof}
	The proof is mostly analogous to the proof of Theorem~\ref{ass}. Namely, using notations of that proof, let us consider the following varieties:
	\begin{align*}
	\tilde X_1&=G\times_P (\Hom(F,\bbk^{d}/\bbk^{d_2})\times\cal Quot^\circ_{0,d_\bullet}); & \tilde W_1&=G\times_{P'} (\Hom(F,\bbk^{d})\times\widetilde{\cal Quot}_{0,d'_\bullet});\\
	\tilde X_2&=G\times_{P'}(\Hom(F,\bbk^{d}/\bbk^{d_1})\times\cal Quot^\circ_{0,d'_\bullet}); & \tilde W_2&=G\times_{P} (\Hom(F,\bbk^{d})\times\widetilde{\cal Quot}_{0,d_\bullet});\\
	\tilde X_3&=\Hom(F,\bbk^{d})\times\cal Quot^\circ_{0,d}; & \tilde W_3&= G\times_P (\Hom(F,\bbk^{d}/\bbk^{d_1})\times\cal Quot^\circ_{0,d_1}\times \widetilde{\cal Quot}_{0,d''_\bullet}).
	\end{align*}
	Again, we have inclusions $\tilde W_i\hookrightarrow \tilde X_{i-1}\times \tilde X_{i+1}$. Taking into account $\Hom$-terms, the proof of Lemma~\ref{transvers} easily implies that $\tilde W_2=\tilde W_1\times_{\tilde X_2}\tilde W_3$, and that the intersection $(\tilde W_1\times \tilde X_1)\cap (\tilde X_3\times \tilde W_3)$ inside $\tilde X_3\times \tilde X_2\times \tilde X_1$ is transversal. Finally, putting $\tilde Z_i=T^*_{\tilde W_i}(\tilde X_{i-1}\times \tilde X_{i+1})$ and contemplating the diagram with cartesian square below:
	\begin{equation}\label{trmodunst}
	\begin{tikzcd}
	T^*\tilde X_1& \tilde Z_3\arrow[r]\arrow[l]& T^*\tilde X_2\\
	& \tilde Z_2\arrow[ul]\ar[dr]\arrow[u]\ar[r] & \tilde Z_1\arrow[u]\arrow[d]\\
	& & T^*\tilde X_3
	\end{tikzcd}
	\end{equation}
	we may conclude as in the proof of Theorem~\ref{ass}.
\end{proof}

Recall that we have open embeddings $\mathscr B(d,n)\subset T^*\underline{\cal Br}_{0,d}^n$. If we denote
$$
A\mathscr M_n=\bigoplus_d A\mathscr M_n[d]:=\bigoplus_d A(\mathscr B(d,n)),
$$
the collection of these embeddings defines us a map of graded vector spaces
\begin{equation}\label{vermahw}
A\bf M_n\rightarrow A\mathscr M_n,
\end{equation}
which is surjective if $A\neq H$ by Proposition~\ref{lespair}.

\begin{corr}\label{reps}
There exists a $A\bf{Ha}_C^0$-module structure on $A\mathscr M_n$, such that the map~(\ref{vermahw}) commutes with the action of $A\bf{Ha}_C^0$.
\end{corr}
\begin{proof}
Let us consider the following diagram:
\begin{equation}\label{alltostable}
\begin{tikzcd}
(T_G^*\tilde{X})^{st}\ar[d,hook,"i"] & (\tilde{Z_G})^{st}\ar[l,"\tilde\Phi'"']\ar[r,"\tilde\Psi'"]\ar[d,hook,"i"] & (T_G^*\tilde{X'})^{st}\ar[d,hook,"i"]\\
T_G^*\tilde{X} & \tilde{Z_G}\ar[l,"\tilde{\Phi}"']\ar[r,"\tilde{\Psi}"] & T_G^*\tilde{X'}
\end{tikzcd}
\end{equation}
where
\begin{align*}
(T_G^*\tilde{X'})^{st}&=T_G^*\tilde{X'}\times_{T^*\underline{\cal Br}_{0,d}^n}\mathscr B(d,n),\\
(T_G^*\tilde{X})^{st}&=T_G^*\tilde{X}\underset{\underline{\cal Higgs}_{0,d_1}\times T^*\underline{\cal Br}_{0,d-d_1}^n}{\bigtimes}\left(\underline{\cal Higgs}_{0,d_1}\times \mathscr B(d-d_1,n)\right),\\
(\tilde{Z_G})^{st}&=\tilde{Z_G}\cap \left( (T_G^*\tilde{X})^{st}\times (T_G^*\tilde{X'})^{st}\right).
\end{align*}
Recall that quotients of stable triples are stable by Lemma~\ref{quotst}.
Therefore we have an equality
$$
\tilde{Z_G}\cap \left( T_G^*\tilde{X}\times (T_G^*\tilde{X'})^{st}\right)=\tilde{Z_G}\cap \left( (T_G^*\tilde{X})^{st}\times (T_G^*\tilde{X'})^{st}\right),
$$
which shows that the map $\tilde\Psi'$ is proper, and right square in the diagram above is cartesian.
Hence, we have
$$
i^*\circ \tilde{\Psi}_*=\tilde{\Psi}'_*\circ i^*
$$
by Lemma~\ref{Gysin}.
This shows us that the diagram~(\ref{alltostable}) defines a commutative square
$$
\begin{tikzcd}
A\bf{Ha}_C^0\otimes A\bf M_n\ar[r,"m"]\ar[d,"\id\otimes\pi"] & A\bf M_n\ar[d,"\pi"]\\
A\bf{Ha}_C^0\otimes A\mathscr M_n\ar[r,"m'"] & A\mathscr M_n
\end{tikzcd}
$$
where $m'=(\tilde\Psi')_*\circ(\tilde\Phi')^!$.
Moreover, if we replace all varieties in diagram (\ref{trmodunst}) by open subvarieties of stable points as above, we can equally see that the upper right square remains cartesian.
Therefore the map $m'$ defines an $A\bf{Ha}_C^0$-module structure on $A\mathscr M_n$.
\end{proof}

Since the whole construction is $T$-equivariant, we also obtain an action of $A\bf{Ha}_C^{0,T}$ on $A\bf M^T_n:=\bigoplus_{d} A^T_*(T^*\underline{\cal Br}_{0,d}^n)$ and $A\mathscr M^T_n:=\bigoplus_d A^T_*(\mathscr B(d,n))$.

\begin{exe}\label{actionA1}
Suppose $C=\bb A^1$, and equip it with the natural action of $\bb G_m$ of weight $1$ as in Example~\ref{SValgs}.
In this setting, for $A=H$ and $A=K$ we recover algebras and representations constructed in~\cite[Proposition 6.2]{SV2} and~\cite[Proposition 7.9]{SV1} respectively.
\end{exe}

We finish this section by comparing our results with the classical construction of Grojnowski and Nakajima. Recall~\cite[Chapter 8]{Nak} that for any smooth surface $S$ there exists an action of Heisenberg algebra on $\bigoplus_d H_*(\ona{Hilb}_d S)$. More precisely, for any positive $k$ and any homology class $\alpha\in H_*(X)$ we possess an operator $P_{\alpha}[k]$, given as follows:
\begin{align*}
P_{\alpha}[i](\beta) & =q_*\circ p^!(\alpha\boxtimes\beta)\text{, where }\beta \in H_*(\ona{Hilb}_d S),\\
Z^\Delta & =\left\{(\cal I_1,\cal I_2)\mid \cal I_1\supset\cal I_2,|\supp(\cal I_1/\cal I_2)|=1\right\}\subset \ona{Hilb}_{d}S\times \ona{Hilb}_{d+k}S,\\
p:Z^\Delta & \ra S\times \ona{Hilb}_d S,\qquad (\cal I_1,\cal I_2)\mapsto (\supp(\cal I_1/\cal I_2), \cal I_1),\\
q: Z^\Delta & \ra S\times \ona{Hilb}_{d+k} S,\qquad (\cal I_1,\cal I_2)\mapsto \cal I_2.
\end{align*}
We now suppose that $S=T^*C$.
Let us compare this action with the $H\bf{Ha}^0_C$-action on $H\mathscr M_1$.
In view of Theorem~\ref{moduli}, $H\mathscr M_1=H_*(\ona{Hilb}_d T^*C)$.
Recall that $\underline{\cal Higgs}_{0,k}\simeq \underline{\cal Coh}_{k}(T^*C)$, where the latter stack parametrizes coherent sheaves of length $k$ on $T^*C$.
Therefore, the correspondence defining the $H\bf{Ha}^0_C$-module structure on $H\mathscr M_1$ can be identified with the lower row in the following diagram with cartesian square:
$$
\begin{tikzcd}
T^*C\times \ona{Hilb}_{d}(T^*C) & & \\
\underline{\cal Coh}^\Delta_{k}(T^*C)\times \ona{Hilb}_{d} (T^*C) \arrow[d,"i\times \id"]\arrow[u,"s\times \id"'] & Z^\Delta \arrow[ul,"p"']\arrow[l,"\Phi^\Delta"]\arrow[dr,"q"]\arrow[d] & \\
\underline{\cal Coh}_{k}(T^*C)\times \ona{Hilb}_{d} (T^*C)  & Z \arrow[r,"\tilde{\Psi}"]\arrow[l,"\tilde{\Phi}"'] & \ona{Hilb}_{d+k}(T^*C)
\end{tikzcd}
$$
where
\begin{align*}
\underline{\cal Coh}^\Delta_{k}(T^*C) & =\{\cal E\in \ona{Coh}_{k}(T^*C)\mid |\supp\cal E|=1\},\\
Z & =\left\{(\cal I_1,\cal I_2)\mid \cal I_1\supset\cal I_2\right\}\subset \ona{Hilb}_{d}T^*C\times \ona{Hilb}_{d+k}T^*C,
\end{align*}
$i$ is the natural closed embedding $\underline{\cal Coh}^\Delta_{k}(T^*C)\hookrightarrow\underline{\cal Coh}_k(T^*C)$, and $s:\underline{\cal Coh}^\Delta_k(T^*C)\ra T^*C$ sends each coherent sheaf to its support.
One would like to prove an equality of the form
\begin{equation}\label{Nakajop}
q_*\circ p^!=(\tilde{\Psi}_*\circ\tilde{\Phi}^!)\circ((\iota\times\id)_*\circ (s\times\id)^!),
\end{equation}
so that the operators $P_{\alpha}[k]$ are realized by action of certain elements in $H\bf{Ha}^0_C$, supported at diagonals $\underline{\cal Coh}^\Delta_k(T^*C)$.
Unfortunately, the map $s$ is too singular for a pullback to be well-defined.
However, one can easily check that it is a locally trivial fibration with a fiber isomorphic to $[\mathscr C_k^{\mathrm n,\mathrm n}/G_k]$, where
$$
\mathscr C_k^{\mathrm n,\mathrm n}:=\{(x,y)\in (\mathfrak g_k)^2:[x,y]=0, x,y\text{ nilpotent}\}.
$$
If the local system $I^k_s$ on $T^*C$, given by homology groups of fibers of $s$, were trivial, $H_*(\underline{\cal Coh}^\Delta_{k}(T^*C))$ would be isomorphic to the direct product $H_*(T^*C)\otimes H_*^{G_k}(\mathscr C_k^{\mathrm n,\mathrm n})$, and one would be able to define the pullback $s^!$ by $c\mapsto c\boxtimes \mathbf 1$.
After that, the identity~(\ref{Nakajop}) would follow once we proved that $p^!=(\Phi^\Delta)^!\circ(s^!\times \id)$.
In light of these considerations, let us state the following conjecture:

\begin{conj}\label{trivfib}
The local system $I^k_s$ is trivial, and the action of $P_\alpha[i]$ on $H\mathscr M_1\simeq\bigoplus_d H_*(\ona{Hilb}_d T^*C)$ is given by $\iota_*(\alpha\boxtimes \mathbf 1)\in H\bf{Ha}^0_C$.
\end{conj} 

Note that Conjecture~\ref{trivfib} is trivially satisfied for $k=1$.
Indeed, $\underline{\cal Coh}^\Delta_{1}(T^*C)\simeq \underline{\cal Coh}_{1}(T^*C)\simeq T^*C\times \mathrm{B}\bb G_m$, thus the diagram above takes the following form:
$$
\begin{tikzcd}
T^*C\times \ona{Hilb}_{d}(T^*C) & & \\
\mathrm{B}\bb G_m\times T^*C\times \ona{Hilb}_{d}(T^*C) \arrow[u,"s\times \id"'] & Z=Z^\Delta \arrow[ul,"p"']\arrow[l,"\tilde{\Phi}"]\arrow[r,"q=\tilde{\Psi}"] & \ona{Hilb}_{d+1}(T^*C)
\end{tikzcd}
$$
Since the scheme $Z$ is smooth by~\cite{Cheah}, pullbacks along all of the maps in triangle are well-defined, and therefore
$$
q_*\circ p^!=q_*\circ p^*=(\tilde{\Psi}_*\circ\tilde{\Phi}^*)\circ(s\times\id)^*,
$$
which gives us a realization of operators $P_\alpha[1]$.

%% file: quiver.tex
\section{Quiver sheaves}\label{quishe}
In this section we recollect some properties of quiver sheaves, as introduced in~\cite{GK}.

Let $X$ be a scheme over $\bbk$.
Let $Q=(I,E)$ be a finite quiver with head and tail maps $h,t:E\to I$, and assume that $Q$ has no cycles.
For each edge $a\in E$, pick a locally free sheaf $\cal M_a \in \ona{Coh}X$, and set $\cal M_i=\Ocal_X$ for all $i\in I$.

Observe that $\cal A_0=\bigoplus_{i\in I}\cal M_i$ is a sheaf of $\Ocal_X$-algebras with coordinate-wise multiplication.
We equip $\cal A_1=\bigoplus_{a\in E}\cal M_a$ with an $\cal A_0$-bimodule structure, where the map
\begin{equation}\label{multmod}
\cal M_i\otimes \cal M_a\otimes \cal M_j=\Ocal_X\otimes \cal M_a\otimes \Ocal_X\to \cal M_a
\end{equation}
is the natural isomorphism if $h(a)=i$, $t(a)=j$, and zero otherwise.

\begin{defi}
The \textit{twisted path algebra} $\cal A=\cal A_{Q,\cal M}$ is the tensor algebra of $\cal A_1$ over $\cal A_0$.
\end{defi}

By definition, $\cal A$ is a sheaf of $\Ocal_X$-algebras.
The category $\cal A\text{-}mod$ of sheaves of coherent $\cal A$-modules (or $\cal A$-modules for short) is an abelian category with enough injectives (see~\cite[Prop. 3.5]{GK}).

Let $e_i$ be the unit section of $\cal M_i=\Ocal_X$.
For each $i\in I$, define left $\cal A$-modules 
$$\cal P_i:=\cal Ae_i=\bigoplus_{j\in I}e_j\cal Ae_i$$
and right $\cal A$-modules 
$$\cal I_i:=e_i\cal A=\bigoplus_{j\in I}e_i\cal Ae_j.$$
By definition of $\cal A$, it decomposes into the direct sum
$$
\cal A=\bigoplus_{i,j\in I}e_i\cal A e_j,
$$
so that we have an equality of left $\cal A$-modules $\cal A=\bigoplus_{i\in I} \cal P_i$ and of right $\cal A$-modules $\cal A=\bigoplus_{i\in I} \cal I_i$.
Note that the multiplication map~(\ref{multmod}) ensures we have maps of $\cal A$-modules
\[
m_a^{(i)}:\cal M_a\otimes \cal I_{t(a)}\to \cal I_{h(a)}, \qquad m_a^{(p)}:\cal P_{h(a)}\otimes \cal M_a\to \cal P_{t(a)}.
\]

An element $\cal V\in \cal A\text{-}mod$ can be equivalently defined as a collection $(\cal V_i,\phi_a)$ of coherent $\Ocal_X$-modules $\cal V_i$, $i\in I$, together with morphisms $\phi_a:\cal M_a\otimes \cal V_{t(a)}\to \cal V_{h(a)}$ for all $a\in E$.
Under this identification, we have a natural isomorphism of $\Ocal_X$-modules $\cal I_i\otimes_{\cal A}\cal V\simeq \cal V_i$.
Since the forgetful functor $\cal A\text{-}mod\to\cal A_0\text{-}mod$ is faithful, we deduce that the functor $\cal I_i\otimes_{\cal A}-$ is exact.

\begin{prop}\label{resolquiv}
We have an exact sequence of left $\cal A^{op}\otimes \cal A$-modules
\[
0\ra\bigoplus_{a\in E} \cal P_{h(a)} \otimes\cal M_a\otimes\cal I_{t(a)}\xra{q} \bigoplus_{i\in I}\cal P_i\otimes \cal I_i\xra{p} \cal A\ra 0,
\]
where all tensor products are considered over $\Ocal_X$, $p$ is the concatenation $\cal Ae_i\otimes e_i\cal A\to \cal A$, and $q$ is given by $q(x,n,y)=m_a^{(p)}(x,n)-m_a^{(i)}(n,y)$.
\end{prop}
\begin{proof}
	The statement is local in $X$.
	When $X=\Spec R$ is affine, this is the standard resolution of the twisted path algebra as a bimodule over itself~\cite[(1.2)]{BK}.
\end{proof}

Let us now consider the derived category $\cal D^b(\cal A\text{-}mod)$.
\begin{corr}\label{resolmod}
	For any $\cal V\in \cal A\text{-}mod$, we have a short exact sequence
	\[
	0\to \bigoplus_{a\in E} \cal P_{h(a)} \otimes\cal M_a\otimes\cal V_{t(a)}\to \bigoplus_{i\in I}\cal P_i\otimes \cal V_i\to \cal V\ra 0.
	\]
	More generally, for any $\cal V^\bullet\in \cal D^b(\cal A\text{-}mod)$ we have an exact triangle
	\[
	\bigoplus_{a\in E} \cal P_{h(a)} \otimes\cal M_a\otimes\cal V^\bullet_{t(a)}\to \bigoplus_{i\in I}\cal P_i\otimes \cal V^\bullet_i\to \cal V^\bullet\xra{+1},
	\]
	where $\cal V_i^\bullet:=\cal I_i\otimes_{\cal A}\cal V^\bullet\in \cal D^b(\ona{Coh}X)$.
\end{corr}
\begin{proof}
Apply the functor $-\otimes_{\cal A} \cal V^\bullet$ to the exact sequence from Proposition~\ref{resolquiv}.
\end{proof}

Let $(\cal V_i,\phi_a),(\cal W,\psi_a)\in\cal A\text{-}mod$, and consider the following complex of sheaves:
\[
C^\bullet(\cal V,\cal W)=\left( \bigoplus_{i\in I}\cal Hom_{\Ocal_X}(\cal V_i,\cal W_i)\xra{\delta} \bigoplus_{a\in E} \cal Hom_{\Ocal_X}(\cal M_a\otimes \cal V_{t(a)},\cal W_{h(a)})\right),
\]
where $\delta$ is given by
\[
(f_i)_{i\in I}\mapsto \left(f_{h(a)}\circ\phi_a-\psi_a\circ (1\otimes f_{t(a)})\right)_{a\in E}.
\]

\begin{thm}[{\cite[Theorem 5.1]{GK}}]\label{Rhoms}
	Let $\cal V,\cal W\in\cal A\text{-}mod$, and suppose $\cal V$ is locally free as $\cal O_X$-module. Then we have an isomorphism of complexes
	\[
	R\Hom(\cal V,\cal W)\simeq R\Gamma(C^\bullet(\cal V,\cal W)).
	\]
\end{thm}

Let us consider a closely related category $\cal A\text{-}mod_{\cal D}$.
Its objects are given by collections $(\cal V_i^\bullet, \phi_a)_{i\in I,a\in E}$, where $\cal V_i^\bullet\in \cal D^b(\cal M_i\text{-}mod)\simeq\cal D^b(\ona{Coh}X)$, $\phi_a\in \bbHom(\cal M_a\otimes \cal V_{t(a)}^\bullet,\cal V_{h(a)}^\bullet)$.
A morphism $(\cal V_i^\bullet, \phi_a)\to (\cal W_i^\bullet, \psi_a)$ is a collection of morphisms $(f_i:\cal V_i^\bullet\to\cal W_i^\bullet)_{i\in I}$, such that $\psi_a \circ f_{t(a)}=f_{h(a)}\circ \phi_a$.

We have a functor $F:\cal D^b(\cal A\text{-}mod)\to\cal A\text{-}mod_{\cal D}$, defined by
\[
\cal V^\bullet\mapsto (\cal V^\bullet_i,\phi_a),
\]
where the maps $\phi_a$ are induced by multiplication maps~(\ref{multmod}).

\begin{lm}\label{Finverse}
The functor $F$ is full and essentially surjective.
\end{lm}
\begin{proof}
Given an object $(\cal V_i^\bullet, \phi_a)\in \cal A\text{-}mod_{\cal D}$, let $\cal V^\bullet$ be a mapping cone of the map
\[
\bigoplus_{a\in E} \cal P_{h(a)} \otimes\cal M_a\otimes\cal V^\bullet_{t(a)}\xra{m_a^{(p)}\otimes 1-1\otimes \phi_a} \bigoplus_{i\in I}\cal P_i\otimes \cal V^\bullet_{i}.
\]
By Corollary~\ref{resolmod}, we have $F(\cal V^\bullet)=(\cal V_i^\bullet, \phi_a)$, so that $F$ is essentially surjective.
Moreover, if we consider any morphism $(\cal V_i^\bullet, \phi_a)\to (\cal W_i^\bullet, \psi_a)$ in $\cal A\text{-}mod_{\cal D}$, the existence of a compatible morphism $\cal V^\bullet\to \cal W^\bullet$ follows from the axioms of a triangulated category.
Thus $F$ is full, and we may conclude.
\end{proof}

Given a category $\mathscr C$, let us denote by $\underline{\mathscr C}$ the groupoid obtained from $\mathscr C$ by forgetting all non-invertible morphisms.

\begin{corr}\label{dermodcommutes}
The functor $F$ induces an equivalence of groupoids $F':\underline{\cal D^b(\cal A\text{-mod})}\to \underline{\cal A\text{-mod}_{\cal D}}$.
\end{corr}
\begin{proof}
Consider the forgetful functor
\[
\cal D^b(\cal A\text{-}mod)\to \cal D^b(\cal A_0\text{-}mod),\qquad \cal V^\bullet\mapsto (\cal V_i^\bullet).
\]
It preserves isomorphisms and factors through $F$. Therefore, $F'$ is faithful.
\end{proof}

%% file: neg.tex
\section{Torsion-free sheaves on $\bbP(T^*C)$}\label{negu}
In this section we prove Theorem~\ref{moduli} by realizing the moduli of Higgs triples as a certain moduli of sheaves on a surface.

Let $X$ be a scheme over $\bbk$, not necessarily smooth.
Pick a line bundle $L$ over $X$, and consider the projectivization $S=\bbP_X(L\oplus \Ocal_X)$ of its total space $\ona{Tot}L$.
Denote the complement of $\ona{Tot}L$ in $S$ by $D$; let also $i:D\hookrightarrow S$ be the natural embedding, and $\pi:S\ra X$ the natural projection.
Note that by definition of $S$ and $D$ we have $R\pi_*\Ocal(D)=\pi_*\Ocal(D)=\Ocal_X\oplus L^\vee$, and $\pi$ induces an isomorphism $D\simeq X$.

Let $\cal T=\Ocal_S(D)\oplus \Ocal_S$, and consider the sheaf of $\Ocal_X$-algebras $\pi_*\cal Hom(\cal T,\cal T)$.
We can write it as a matrix algebra over $X$; the opposite algebra, which we denote by $\cal A$, is then obtained by transposition:
\begin{equation*}
\pi_*\cal Hom(\cal T,\cal T)=\begin{pmatrix}
\Ocal & \Ocal\oplus L^\vee \\
0 & \Ocal
\end{pmatrix},
\qquad
\cal A=\begin{pmatrix}
\Ocal & 0 \\
\Ocal\oplus L^\vee & \Ocal
\end{pmatrix}.
\end{equation*}

Note that $\cal A$ can be seen as a twisted path algebra of the following quiver:
$$
Q=\begin{tikzcd}
\tikz{\node[draw, circle,inner sep=2pt]{1}}\ar[r,shift left=0.5ex,"\Ocal"]\arrow[r,shift right=0.5ex,"L^\vee"']&\tikz{\node[draw, circle,inner sep=2pt]{2}}
\end{tikzcd}
$$
A left $\cal A$-module $\cal V$ is then determined by a quadruple $(\cal V_1,\cal V_2, \phi_0,\phi_1)$, where $\cal V_1,\cal V_2\in\ona{Coh}X$, $\phi_0\in\Hom(\cal V_1,\cal V_2)$, and $\phi_1\in\Hom(\cal V_1\otimes L^\vee,\cal V_2)$.

For any coherent sheaf $E\in\ona{Coh}S$, the Hom-sheaf $\cal Hom(\cal T, E)=\pi_*(\cal T^\vee\otimes E)$ is naturally a left $\cal A$-module, given by the quadruple $(\pi_*E(-D),\pi_*E,\phi_0,\phi_1)$, where $(\phi_0,\phi_1)$ is the natural composition
\[
\pi_*E(-D)\otimes (\Ocal\oplus L^\vee)=\pi_*E(-D)\otimes \pi_*\Ocal(D)\to \pi_*E.
\]

Since $\cal T$ is an $(\cal A^{op},\Ocal_S)$-bimodule, we have a pair of adjoint functors
\begin{equation}\label{adjACK}
\begin{tikzcd}
-\otimes^L_{\cal A}\cal T: \cal D^b(\cal A\text{-}mod) \arrow[r,shift left=0.5ex]& \cal D^b(\ona{Coh}S) : R\pi_*\cal Hom(\cal T,-)\arrow[l,shift left=0.5ex].
\end{tikzcd}
\end{equation}

As a left module over itself, $\cal A$ can be decomposed as a direct sum $\cal P_1\oplus \cal P_2$, where
\[
\cal P_1=\begin{tikzcd}[column sep = small]
\Ocal\arrow[r,shift left=0.5ex]\arrow[r,shift right=0.5ex]&\Ocal\oplus L^\vee
\end{tikzcd}, \qquad \cal P_2=\begin{tikzcd}[column sep = small]
0\arrow[r,shift left=0.5ex]\arrow[r,shift right=0.5ex]&\Ocal
\end{tikzcd}
\]
are the left $\cal A$-modules defined in Section~\ref{quishe}.

The following proposition should be known to experts (for example, see remark at the end of~\cite{Bei}), but we include the proof for completeness.

\begin{prop}\label{Beiliequiv}
The pair of functors (\ref{adjACK}) establishes an equivalence of triangulated categories.
\end{prop}
\begin{proof}
The proof is based on Be\u{\i}linson's lemma~\cite{Bei}.
For any $E\in \ona{Coh}S$, there exists $n>0$ such that $E(nD)$ has no higher cohomology, and the counit map $\pi^*\pi_*E(nD)\to E(nD)$ is surjective.
By the seesaw principle~\cite[Corollary 5.6]{Mum}, the kernel of this map has the form $\pi^*(\cal N)(-D)$, where $\cal N\in \ona{Coh}X$.
Thus $E$ admits a resolution of the form
\[
0\to \pi^*(\cal N_2)(-(n+1)D)\to \pi^*(\cal N_1)(-nD)\to E\to 0,
\]
where $\cal N_1,\cal N_2\in \ona{Coh}X$. Taking into account short exact sequences
\[
0\to \Ocal((n-1)D) \to \Ocal(nD)\oplus \Ocal(nD)\to \Ocal((n+1)D)\to 0,
\]
we see that as a triangulated category, $\cal D^b(\ona{Coh}S)$ is generated by $\ona{Coh}X$ and $\cal T=\Ocal_S(D)\oplus \Ocal_S$.
Similarly, $\cal D^b(\cal A\text{-mod})$ is generated by $\ona{Coh}X$ and $\cal A=\cal P_1\oplus\cal P_2$ as a triangulated category by Corollary~\ref{resolmod}.

We have $R\pi_*\cal Hom(\cal T,\Ocal(D))=\cal P_1$, $R\pi_*\cal Hom(\cal T,\Ocal)=\cal P_2$.
Using Theorem~\ref{Rhoms}, it is easy to check the following isomorphisms:
\begin{align*}
&R\Hom(\pi^*\cal E,\pi^*\cal F)\simeq R\Hom(\cal E,\cal F)\simeq R\Hom(\cal E\otimes \cal P_1,\cal F\otimes \cal P_1),\\
&R\Hom(\pi^*\cal E(D),\pi^*\cal F)\simeq 0\simeq R\Hom(\cal E\otimes \cal P_1,\cal F\otimes \cal P_2),\\
&R\Hom(\pi^*\cal E,\pi^*\cal F(D))\simeq R\Hom(\cal E,\cal F\otimes (\Ocal\oplus L^\vee))\simeq R\Hom(\cal E\otimes \cal P_2,\cal F\otimes \cal P_1),\\
&R\Hom(\pi^*\cal E(D),\pi^*\cal F(D))\simeq R\Hom(\cal E,\cal F)\simeq R\Hom(\cal E\otimes \cal P_2,\cal F\otimes \cal P_2).
\end{align*}

Applying Be\u{\i}linson's lemma, we conclude that the functor $R\cal Hom_X(\cal T,-)$ is an equivalence of triangulated categories.
Moreover, since the functor $-\otimes^L_{\cal A}\cal T$ is its left adjoint, it provides the inverse equivalence.
\end{proof}

Let us apply this proposition to $X=T\times C$, $L=\Ocal_T\boxtimes \omega_C$.
Combining it with Corollary~\ref{dermodcommutes}, we obtain an equivalence of groupoids
\begin{equation}\label{mainderequi}
\Theta:\underline{\cal D^b(\ona{Coh}(T\times\bbP_C(\omega\oplus \Ocal)))}\simeq \underline{(Lp^*\cal A)\text{-}mod_{\cal D}},
\end{equation}
where $p:T\times C\to C$ is the projection.
Moreover, this equivalence commutes with base change in $T$ whenever the latter preserves bounded derived categories, e.g.\ for flat maps $T'\to T$.

\begin{rmq}
It would be desirable to express this as an equivalence of presheaves in groupoids.
The problem is that groupoids on both sides of~(\ref{mainderequi}) are not functorial in $T$.
Namely, boundedness of complexes is not preserved under pullbacks along general maps $T'\to T$.
Nevertheless, in the sequel we are only concerned with certain subgroupoids on both sides, see Proposition~\ref{nattrsubfun}. 
Their objects will satisfy flatness condition over $T$, and therefor will be preserved under arbitrary base change, forming presheaves.
We will thus abuse the notation for convenience, and say that the two sides of~(\ref{mainderequi}) form presheaves $\underline{\cal D^b(\ona{Coh}S)}$ and $\underline{\cal A\text{-}mod_{\cal D}}$ respectively.
\end{rmq}

From now on, let $X=C$, so that $S=\bbP_C(\omega\oplus \Ocal)$ compactifies the cotangent bundle $T^*C$.
Let us recall some properties of sheaves on $S$; we will closely follow the exposition in~\cite[Section 2]{Moz}.
The Neron-Severi group of $S$ is given by $\ona{NS}(S)=H^2(S,\bbZ)=\bbZ D\oplus \bbZ f$, where $f$ is the class of a fiber of $\pi:S\to C$.
Thus, for any coherent sheaf $E$ on $S$ we will write the first Chern class $c_1(E)$ as a linear combination $c_{1,D}(E)D+c_{1,f}(E)f$.
The product in $\ona{NS}(S)$ is determined by the following equalities:
\begin{equation*}
f^2=0,\qquad Df=1,\qquad D^2=2-2g,
\end{equation*}
where the last one follows from the fact that $\Ocal_S(D)|_D\simeq \omega^{-1}$.
Moreover, the canonical divisor of $S$ is $K_S=-2D$. 
We will write elements of $H^{even}(S,\bbZ)$ as triples $(a,b,c)\in\bbZ\oplus\ona{NS}(S)\oplus\bbZ$; the same applies to $H^{even}(C,\bbZ)$.
In this fashion, Todd classes of $S$ and $C$ are respectively given by
$$
\ona{td}S=(1,D,1-g),\qquad \ona{td}C=(1,1-g),
$$
and the pushforward along $\pi$ in cohomology is given by
$$
\pi_*(a,b_D D+b_f f,c)=(b_D,c).
$$
Given a sheaf $E\in\ona{Coh}S$, the Chern character of its derived pushforward $R\pi_*E$ can be computed using Grothendieck-Riemann-Roch theorem.
Namely, let $a=c_{1,D}(E)$, $b=c_{1,f}(E)$, $r=\rk E$, and recall that $$\ona{ch}(E)=\left(r,c_1(E),\frac{c_1(E)^2-2c_2(E)}{2}\right)=(r,aD+bf,a^2(1-g)+ab-c_2(E)).$$
We have:
\begin{align*}
(\rk(R\pi_*E),&\text{ }c_1(R\pi_*E)+(1-g)\rk(R\pi_*E))=\ona{ch}(R\pi_*E)\ona{td}C=\pi_*(\ona{ch}E\ona{td}S)\\
&=\pi_*((r,aD+bf,\ona{ch}_2(E))(1,D,1-g))\\
&=(a+r,(r+2a)(1-g)+b+\ona{ch}_2(E)).
\end{align*}
The result of this computation can be rewritten as follows:
\begin{equation}\label{GRRcomp}
\rk(R\pi_*E)=a+r,\qquad c_1(R\pi_*E)=a(1-g)+b+\ona{ch}_2(E).
\end{equation}

For any nef divisor $H$ on $S$, we can define a notion of \textit{$H$-semistability} for sheaves on $S$.
One example of nef divisor is given by $f$.
Instead of giving general definitions, we will use the following characterization of $f$-semistable sheaves:

\begin{lm}[{\cite[Lemma 4.3]{Moz}}]\label{Mozlem}
A torsion-free sheaf $E$ on $S$ is $f$-semistable if and only if its generic fiber over $C$ is isomorphic to $\Ocal_{\bbP^1}(l)^{\oplus m}$ for some $l\in\bbZ$, $m\in\bbN$.
\end{lm}

\begin{lm}\label{fsstnum}
For a torsion-free $f$-semistable sheaf $E$, the following numerical conditions are equivalent:
\[  \text{1. }\pi_*E=0\text{ and }\rk R^1\pi_*E=0,\qquad\qquad\text{2. }\rk E=-c_{1,D}(E),\qquad\qquad \text{3. }l=-1.\]
If these conditions are fulfilled, we further have $H^0(E)=H^2(E)=0$, and $H^1(E)=H^0(R^1\pi_*E)$.
\end{lm}
\begin{proof}
Let us first prove the equivalence.

$1\Rightarrow 2$: follows from the first formula in~(\ref{GRRcomp});

$2\Rightarrow 3$: rank is a generic invariant, therefore we have
\[
0=\rk R\pi_*E=m\left(h^0(\Ocal_{\bbP^1}(l))-h^1(\Ocal_{\bbP^1}(l))\right)=m(l+1).
\]
Since $m$ is a positive number, this implies that $l=-1$.

$3\Rightarrow 1$: since $R\Gamma(\Ocal_{\bbP^1}(-1))=0$, both $\pi_*E$ and $R^1\pi_*E$ have rank 0.
Furthermore, let $\cal T\subset \pi_*E$ be a torsion subsheaf.
By adjunction, we exhibit a  map $\pi^*\cal T\to E$ from a torsion sheaf to a torsion-free sheaf.
It is a zero map if and only if $\cal T=0$; thus $\pi_*E$ is locally free.
We conclude that $\pi_*E=0$.

In order to prove the second statement, recall that we have Leray spectral sequence
\[
H^i(C,R^j\pi_*E)\Rightarrow H^{i+j}(S,E).
\]
Since $R^j\pi_*E=0$ for $j\neq 1$, it degenerates to the equality $H^{i+1}(E)=H^i(R^1\pi_*E)$.
Finally, $R^1\pi_*E$ is a torsion sheaf, so that $H^i(E)$ is non-zero only for $i=1$.
\end{proof}

We will also need the following computation:
\begin{lm}\label{counit}
Let $\eps:\pi^*\pi_*\Ocal(D)\to \Ocal(D)$ be the natural counit map.
Then $\Ker\eps\simeq\pi^*\omega^\vee(-D)$.
\end{lm}
\begin{proof}
Let us denote $K=\Ker\eps$.
Since $\eps$ is surjective and becomes an isomorphism after applying $\pi_*$, we have $R\pi_*K=0$.
This means that at each point $c\in C$ the fiber $K_c$ is isomorphic to a direct sum of several copies of $\Ocal_{\bbP^1}(-1)$~\cite[Corollary 5.4]{Mum}.
In particular, $K(D)_c$ is trivial at each point $c$, and thus the natural map $\pi^*\pi_*(K(D))\ra K(D)$ is an isomorphism.
Consider the following short exact sequence:
\[
0\ra K(D)\ra \left(\pi^*\pi_*\Ocal(D)\right)\otimes \Ocal(D)\to \Ocal(2D)\to 0.
\]
Note that all these sheaves have globally generated fibers over $C$.
Therefore, after applying $\pi_*$ we obtain 
\[
\pi_*(K(D))\simeq \Ker\left( \pi_*\Ocal(D) \otimes \pi_*\Ocal(D)\to \pi_*\Ocal(2D) \right).
\]
However, since $\pi_*\Ocal(D)= \Ocal\oplus \omega^\vee$, we have
\[
\pi_*(K(D))\simeq\Ker\left((\Ocal\oplus \omega^\vee)\otimes (\Ocal\oplus \omega^\vee)\to \Ocal\oplus \omega^\vee\oplus (\omega^\vee)^2\right)\simeq \omega^\vee.
\]
Therefore $K\simeq \pi^*\pi_*(K(D))\otimes \Ocal(-D)\simeq \pi^*\omega^\vee(-D)$, and we may conclude.
\end{proof}

\begin{rmq}\label{counitconv}
For later purposes, let us fix an isomorphism $\pi_*(K(D))\simeq \omega^\vee$, so that the inclusion $\pi_*(K(D))\subset \pi_*\Ocal(D)\otimes \pi_*\Ocal(D)$ is identified with the composition $$\omega^\vee\xra{(1,-1)}\left(\omega^\vee\otimes\Ocal\right)\oplus \left(\Ocal\otimes\omega^\vee\right)\subset \left(\Ocal\oplus \omega^\vee\right)\otimes \left(\Ocal\oplus \omega^\vee\right).$$
\end{rmq}

Let us return to the equivalence~(\ref{mainderequi}).
Fix $d>0$, and a locally free sheaf $\cal F\in \ona{Coh}C$ of rank $n$.
Consider subfunctors $\underline{\cal Coh}_d^{\cal F}S\subset\underline{\cal D^b(\ona{Coh}S)}$, $\underline{\cal A\text{-}mod}_d^{\cal F}\subset\underline{\cal A\text{-}mod_{\cal D}}$, defined as follows:
\begin{align*}
&\underline{\cal Coh}_d^{\cal F}S(T)  =\left\{
E\in\ona{Coh}(S\times T)\,\middle\vert\,
\begin{array}{c}
E \text{ is torsion-free, flat over $T$, }E|_{D\times T}\simeq \cal F\boxtimes \Ocal_T,\\
c_1(E_t)=\deg\cal F\cdot f,c_2(E_t)=d,E_t\text{ is $f$-semistable }\forall t\in T\\
\end{array}
\right\},\\
&\underline{\cal A\text{-}mod}_d^{\cal F}(T)  =\left\{
(\cal V_1[-1],\cal V_2^\bullet,\phi_0,\phi_1)\,\middle\vert\,
\begin{array}{c}
\cal V_1\in\ona{Coh}(C\times T),\text{ flat over $T$, }Cone(\phi_0)\simeq \cal F\boxtimes \Ocal_T,\\
(\cal V_1)_t\text{ is a torsion sheaf, }\deg (\cal V_1)_t=d\text{ }\forall t\in T\\
\end{array} 
\right\}.
\end{align*}

\begin{prop}\label{nattrsubfun}
For any $d>0$, the equivalence (\ref{mainderequi}) induces a natural transformation $$\Theta_d':\underline{\cal Coh}_d^{\cal F}S\to \underline{\cal A\text{-}mod}_d^{\cal F}.$$
\end{prop}
\begin{proof}
We need to check that for any $T$, we have $\Theta\left( \underline{\cal Coh}_d^{\cal F}S(T) \right)\subset \underline{\cal A\text{-}mod}_d^{\cal F}(T)$.
Let $E\in \underline{\cal Coh}_d^{\cal F}S(T)$.
Then by definition $\Theta(E)=(R\pi_*E(-D),R\pi_*E,\phi_0,\phi_1)$, where $\phi_0:R\pi_*E(-D)\to R\pi_*E$ is obtained by applying $R\pi_*$ to the first map in the short exact sequence
\begin{equation*}
0\to E(-D)\to E\to E|_{D\times T}\to 0.
\end{equation*}
In particular, $Cone(\phi_0)\simeq R\pi_*(E|_{D\times T})\simeq \cal F\boxtimes \Ocal_T$.

Pick a point $t\in T$.
We have $\pi_*E(-D)_t=0$ by Lemma~\ref{fsstnum}, so that $R\pi_*E(-D)_t=R^1\pi_*E(-D)_t[-1]$.
Moreover, the formulas~(\ref{GRRcomp}) applied to $E(-D)_t$ show that $\rk(R\pi_*E(-D)_t)=0$ and $c_1 (R\pi_*E(-D)_t)=-d$.
Therefore, $R^1\pi_*E(-D)_t$ is a torsion sheaf of degree $d$.

Finally, let us prove flatness.
Let $f:T\times S\to T$, $p:T\times C\to T$ be the natural projections.
Using the second part of Lemma~\ref{fsstnum}, a proof analogous to~\cite[Corollary 4.2.12]{LeP} shows that $R^1f_*E(-D)$ is a locally free sheaf.
Let $\cal L$ be an ample line bundle on $C$, and $k\in\bbN$.
Since $R^1\pi_*E(-D)_t$ is a torsion sheaf for any $t\in T$, it is isomorphic to $\cal L^k\otimes R^1\pi_*E(-D)$ in the neighborhood of $t$.
In particular, the fact that $p_*(R^1\pi_*E(-D))\simeq R^1f_*E(-D)$ is locally free implies that $p_*(\cal L^k\otimes R^1\pi_*E(-D))$ is locally free for any $k$.
We conclude that $R^1\pi_*E(-D)$ is flat over $T$ by~\cite[Proposition 2.1.2]{HL2}.
\end{proof}

Consider rigidified functors $\underline{\cal Coh}_d^{\leftarrow\cal F}S$, $\underline{\cal D^b(\ona{Coh}S)}^{\leftarrow\cal F}$, where we fix the additional data of an isomorphism $\Psi:E|_D\xra{\sim} \cal F$.
We will refer to elements of $(\underline{\cal Coh}_d^{\leftarrow\cal F}S)(\bbk)$ as \textit{$\cal F$-framed sheaves}.

\begin{lm}\label{locfreeD}
Any $\cal F$-framed sheaf is locally free in a neighborhood of $D$.
\end{lm}
\begin{proof}
Recall that for any torsion-free sheaf $E$, its double dual $E^{\vee\vee}$ is a vector bundle.
Let $(E,\Psi)$ be an $\cal F$-framed sheaf, and consider the quotient $U=E^{\vee\vee}/E$.
It is a sheaf with zero-dimensional support.
If the intersection $D\cap \supp U$ is non-empty, $E|_D$ is a proper subsheaf of $E^{\vee\vee}|_D$.
However,
\[
\deg E|_D=\deg \cal F=c_1(E)D=c_1(E^{\vee\vee})D=\deg E^{\vee\vee}|_D,
\]
and $\rk E|_D=\rk E^{\vee\vee}|_D$, so that $E|_D=E^{\vee\vee}|_D$.
Therefore the support of $U$ is disjoint from $D$, and we have an isomorphism $E\simeq E^{\vee\vee}$ in a neighborhood of $D$.
\end{proof}

Let us recall a closely related notion of stable pairs.
We specialize the definition in~\cite{BM} to the case when polarization of $S$ given by the divisor $H=D+Nf$, and $N>2g-2$.
Recall that for any locally free sheaf $\cal E$ on $C$, its \textit{slope} is defined as $\mu(\cal E)=\deg\cal E/\rk \cal E$.
\begin{defi}
Let $E$ be a torsion-free sheaf on $S$ satisfying $\ona{ch}E=(n,\deg \cal F\cdot f,-d)$, and $\Psi:E|_D\xra{\sim}\cal F$ an isomorphism.
Fix $N>2g-2$, and $\delta>0$.
A pair $(E,\Psi)$ is said to be \textit{$(N,\delta)$-stable}, if for any subsheaf $E'\subset E$ with $0<\rk E'<n$ the following inequality holds:
\[
\frac{c_1(E')H}{\rk E'}<
\begin{cases*}
\mu(\cal F)-\delta/n & if $E'\subset E(-D)$, \\
\mu(\cal F)+\delta/\rk E'-\delta/n & otherwise.
\end{cases*}
\]
\end{defi}
It is known that the moduli of $(N,\delta)$-stable pairs is represented by a quasi-projective variety, see~\cite[Theorem 2.3]{BM}.

\begin{prop}\label{fstmodrep}
There exist $N,\delta$ big enough, such that every $\cal F$-framed sheaf $(E,\Psi)$ is $(N,\delta)$-stable.
In particular, the functor $\underline{\cal Coh}_d^{\leftarrow\cal F}S$ is represented by a quasi-projective variety $\mathscr B(d,\cal F)$.
\end{prop}
\begin{proof}
The $(N,\delta)$-stability condition is vacuous for sheaves of rank 1.
Therefore, we will assume that $n\geq 2$.
The existence of Harder-Narasimhan filtration~\cite[Chapter 5]{LeP} implies that for a locally free sheaf $\cal F$ on $C$, there exists a constant $\mu_{max}(\cal F)$, such that $\mu(\cal F')<\mu_{max}(\cal F)$ for all $\cal F'\subset\cal F$.
From now on, we will assume that $\delta>(\mu_{max}(\cal F)-\mu(\cal F))n^2$, and $N>2g-2+\delta$.

Let $E'\subset E$ be a subsheaf of rank $n'$ with $0<n'<n$.
Since $E_c\simeq\Ocal_{\bbP^1}^n$ for a generic $c\in C$, there exist integers $0\leq k_{1}\leq\ldots\leq k_{n'}$ such that generically $E'_c\simeq \bigoplus_{i=1}^{n'}\Ocal_{\bbP^1}(-k_i)$.

Assume first that $E'\not\subset E(-D)$.
Consider the saturation $\overline{E'}$ of $E'$ inside $E$.
It has the same rank as $E'$, and $c_{1,D}(\overline{E'})\leq 0$.
Moreover, since $E$ is a vector bundle in the neighborhood of $D$ by Lemma~\ref{locfreeD}, $\overline{E'}$ is its subbundle in the same neighborhood.
As a consequence, we have $\overline{E'}|_D\subset E|_D$, and $c_{1}(\overline{E'})D=\deg(\overline{E'}|_D)$.
Putting this together, we obtain
\begin{equation}\label{stisst}
\begin{aligned}
\frac{c_1(E')H}{r'} & \leq \frac{c_1(\overline{E'})H}{r'}=\mu\left(\overline{E'}|_D\right)+\frac{Nc_{1,D}(\overline{E'})}{r'}\leq \mu_{max}(\cal F)<\mu(\cal F)+\delta/n^2\\
& <\mu(\cal F)+\delta/r'-\delta/n,
\end{aligned}
\end{equation}
which is the desired estimate.

Now, suppose $E'\subset E(-D)$.
In this case $k_1>0$, and $E'$ is not contained in $E(-(k_1+1)D)$.
Let $k$ be the maximal positive integer such that $E'\subset E(-kD)$; we have $k\leq k_1$.
In particular, $E'(kD)$ is naturally a subsheaf of $E$, which is not contained in $E(-D)$.
Moreover, for a generic point $c\in C$, we have $E'(kD)_c\simeq\bigoplus_{i=1}^{n'}\Ocal_{\bbP^1}(-\widetilde{k}_i)$, where $\widetilde{k}_i=k_i-k\geq 0$ for all $i$.
Therefore, the inequality~(\ref{stisst}) holds for $E'(kD)$ by previous considerations.
We have
\begin{align*}
\frac{c_1(E')H}{r'} & =\frac{c_1(E'(kD))H}{r'}-kD\cdot H =\frac{c_1(E'(kD))H}{r'}+k(2g-2-N)\\
 & <\mu(\cal F)+\delta/n^2+k(2g-2-N)<\mu(\cal F)+\delta/n^2-k\delta<\mu(\cal F)-\delta/n.
\end{align*}
We can thus conclude that for our choice of $N,\delta$ every $\cal F$-framed sheaf is $(N,\delta)$-stable.
\end{proof}
\begin{rmq}
The divisor $D\subset S$ is not nef when $g(C)>1$, so that~\cite[Theorem 3.1]{BM} could not be invoked directly.
Moreover, one can check that if we omit framing, $f$-semistable sheaves can possess automorphisms.
\end{rmq}

Analogously to $\underline{\cal Coh}_d^{\leftarrow\cal F}S$, consider the rigidified functor $\underline{\cal A\text{-}mod}_d^{\leftarrow\cal F}$, given by fixing a distinguished triangle $\Delta = \left(\cal V_1[1]\xra{\phi_0} \cal V_2^\bullet\xra{\psi} \cal F\xra{\phi'_0}\right)$.
Then $\Theta_d'$ extends to a natural transformation $\Theta^\leftarrow=\Theta^\leftarrow_d: \underline{\cal Coh}_d^{\leftarrow\cal F}S\to \underline{\cal A\text{-}mod}_d^{\leftarrow\cal F}$.

Let us establish relation between $\underline{\cal A\text{-}mod}_d^{\leftarrow\cal F}$ and the stack of Higgs triples.
\begin{lm}\label{coneunique}
Let $\cal E,\cal F\in \ona{Coh}C$, $\phi\in\Hom(\cal F, \cal E)$, and $\cal C_1, \cal C_2$ two cones of $\phi$.
Then there exists the unique map $f\in\bbHom(\cal C_1,\cal C_2)$ making the following diagram commute:
\[
\begin{tikzcd}
 \cal E\ar[d,equal]\ar[r,"i"] & \cal C_1\ar[d,"f"]\ar[r,"j"] & \cal F[1]\ar[r,"\phi"]\ar[d,equal]& \text{ }\\
 \cal E\ar[r,"i"] & \cal C_2\ar[r,"j"] & \cal F[1]\ar[r,"\phi"] & \text{ }
\end{tikzcd}
\]
where $i,j$ are the natural maps.
\end{lm}
\begin{proof}
The existence of map $f$ is assured by axioms of triangulated category.
Let $f_1$, $f_2$ be two such maps, and consider their difference $g=f_1-f_2:\cal C_1\to \cal C_2$.
By definition, we have $g\circ i=0$ and $j\circ g=0$.
Therefore, $g$ lies in the image of composition
\[
\bbHom(\cal F[1],\cal E)\to \bbHom(\cal C_1,\cal E)\to \bbHom(\cal C_1,\cal C_2).
\]
Since both $\cal E$ and $\cal F$ lie in the heart of $\cal D^b(\ona{Coh}C)$, we have $\bbHom(\cal F[1],\cal E)=0$.
Thus $g=0$, and the unicity of $f$ follows.
\end{proof}

Thanks to the lemma above, we can define a natural transformation $\tau:\underline{\cal A\text{-}mod}_d^{\leftarrow\cal F}\to T^*\underline{\cal Coh}^{\leftarrow \cal F}_{0,d}$, which to each element $(\cal V_1[-1],\cal V_2^\bullet,\phi_0,\phi_1,\Delta)\in \underline{\cal A\text{-}mod}_d^{\leftarrow\cal F}(T)$ associates the triple
\[
(\cal V_1, \phi'_0, f\circ\phi_1),
\]
where $f:\cal V_2^\bullet\xra{\sim} (\cal F\xra{\phi_0} \cal V_1)$ is the unique isomorphism given by Lemma~\ref{coneunique}.

\begin{prop}\label{tauequiv}
The functor $\tau$ is a natural equivalence.
\end{prop}
\begin{proof}
Let us consider a natural transformation
\[
\upsilon: T^*\underline{\cal Coh}^{\leftarrow \cal F}_{0,d}\to \underline{\cal A\text{-}mod}_{d}^{\leftarrow\cal F},
\]	
defined on $T$-points by the following formula:
\[
\upsilon(\cal E, \alpha,\theta)=(\cal E[-1],\cal F\xra{\alpha}\cal E,\iota,\theta[-1],\Delta).
\]
Here $\iota$ is the natural map $\cal E[-1]\to(\cal F\to\cal E)$, and $\Delta$ is obtained from the mapping cone of $\alpha$:
\[
\Delta = \left(\cal E[-1]\xra{\iota} (\cal F\to\cal E) \xra{\psi} \cal F\xra{\alpha}\right).
\]
It is clear that $\tau\circ \upsilon$ is the identity functor.
On the other hand, the composition $\upsilon\circ \tau$ sends $(\cal V_1,\cal V_2^\bullet,\phi_0,\phi_1,\Delta)$ to $(\cal V_1,\cal F\xra{\phi'_0}\cal V_1,f\circ \phi_0,f\circ \phi_1,\Delta')$, where $\Delta'$ is the triangle
\[
\Delta' = \left(\cal V_1[-1]\xra{\iota} (\cal F\to\cal V_1) \xra{\psi} \cal F\xra{\phi_0'}\right).
\] 
The map $f$ induces a natural equivalence $\upsilon\circ \tau\simeq \Id_{\underline{\cal A\text{-}mod}_{d}^{\leftarrow\cal F}}$, so that $\tau$ and $\upsilon$ are mutually inverse equivalences.
\end{proof}

\begin{thm}\label{torfree}
The composition $\overleftarrow{\Theta} =\tau\circ \Theta^\leftarrow$ factors through the stack of stable Higgs triples, and induces an equivalence $$\overleftarrow{\Theta}:\underline{\cal Coh}_d^{\leftarrow\cal F}S\xra{\sim}\left( T^*\underline{\cal Coh}^{\leftarrow \cal F}_{0,d} \right)^{st}.$$
\end{thm}
\begin{proof}
Since the functor $\overleftarrow{\Theta}$ is fully faithful, it is enough to compute its image on $T$-points of $\underline{\cal Coh}_d^{\leftarrow\cal F}S$ for every $T$.
Let $(\cal E,\alpha,\theta)$ be a Higgs triple, and consider the morphism
\begin{align*}
\xi=\eps\otimes 1+\pi^*(\iota,\theta\otimes\omega^\vee) & \in\bbHom\left( \pi^*\pi_*\Ocal(D)\otimes \pi^*\cal E, \Ocal(D)\otimes\pi^*\cal E\oplus \pi^*\alpha[1]\right),
\end{align*}
where $\eps:\pi^*\pi_*\Ocal(D)\to\Ocal(D)$ is the counit map from Lemma~\ref{counit}, and we use the identification $\pi_*\Ocal(D)\simeq \Ocal\oplus\omega^\vee$.
After restricting to $D$, we get
\[
\xi|_D=\begin{pmatrix}
\iota & \theta\otimes\omega^\vee \\
0 & 1
\end{pmatrix}:\cal E\oplus\cal E\otimes\omega^\vee\to \alpha[1]\oplus\cal E\otimes\omega^\vee.
\]
The map $\xi|_D$ can be naturally completed to a distinguished triangle
\[
\cal F\xra{(\alpha,0)}\cal E\oplus\cal E\otimes\omega^\vee\xra{\xi|_D} \alpha[1]\oplus\cal E\otimes\omega^\vee\xra{+1}.
\]
Thus, we obtain an identification $\Psi:Cone(\xi|_D)\xra{\sim} \cal F$.

Note that $$\overleftarrow{\Theta}(E)=\tau\circ F'(R\pi_*\cal Hom(\cal T,E))$$ up to remembering the framing.
Combining the inverses of each functor in the composition provided by Proposition~\ref{tauequiv}, Lemma~\ref{Finverse} and Proposition~\ref{Beiliequiv}, we obtain the following left inverse of $\Theta^\leftarrow$ on the set of isomorphism classes of $T$-points:
\[
G_T:T^*\underline{\cal Coh}^{\leftarrow \cal F}_{0,d}(T) \to \underline{\cal D^b(\ona{Coh}S)}^{\leftarrow\cal F}(T),\qquad(\cal E,\alpha,\theta) \mapsto (Cone(\kappa\circ\xi)[-1],\Psi),
\]
where $\kappa:\pi^*\cal E(D)\oplus \pi^*\alpha[1]\to \pi^*\cal E(D)\oplus \pi^*\alpha[1]$ multiplies first summand by $1$, and the second one by $-1$.
In what follows, we are only interested in the set-theoretic image of $G_T$.
As such, we will not concern ourselves with functoriality, and will liberally make use of (non-unique) cones of various maps.
\begin{lm}\label{adjuct}
As a complex, $\xi$ is quasi-isomorphic to $\Ocal(-D)\otimes (\pi^*(\cal E\otimes \omega^\vee)\xra{\xi'} \Ocal(D)\otimes \pi^*\alpha[1])$, where $\xi'$ is obtained by adjunction from $(-\theta\otimes\omega^\vee,\iota\otimes\omega^\vee)\in\bbHom(\cal E\otimes \omega^\vee,\pi_*\Ocal(D)\otimes \alpha[1])$.
\end{lm}
\begin{proof}
Let $p_1,p_2$ be the projection maps from $\pi^*\cal E(D)\oplus \pi^*\alpha[1]$ to the first and the second summand respectively.
Note that $p_1\circ \xi=\eps\otimes 1$.
Applying octahedral axiom to these three maps, we obtain a distinguished triangle, denoted by dashed arrows below:
\[
\begin{tikzcd}
\pi^*(\cal E\otimes\omega^\vee)(-D)\ar[r,dashed,"\xi''"]\ar[d,hook,"j"]&\pi^*\alpha[1]\ar[r,dashed]\ar[d,"i_2"] &Cone(\xi)\ar[d,equal]\\
\pi^*\cal E\otimes \pi^*\pi_*\Ocal(D)\ar[r,"\xi"]\ar[d,two heads,"\eps\otimes 1"]&\pi^*\cal E(D)\oplus \pi^*\alpha[1]\ar[r]\ar[d,"p_1"] &Cone(\xi)\\
\pi^*\cal E(D)\ar[r,equal]&\pi^*\cal E(D) &
\end{tikzcd}
\]
Here, $j$ is defined by Lemma~\ref{counit}, and $i_2$ is the natural inclusion of a summand.
We see that $\xi$ is quasi-isomorphic to $\xi''$, so it remains to compute the map $\xi''$.
Since $j$ is injective and the diagram commutes, we have $\xi''=p_2\circ\xi\circ j=\pi^*(\iota,\theta\otimes\omega^\vee)\circ j$.
After tensoring with $\Ocal(D)$ and applying $\pi_*$, we obtain the composition
\[
\cal E\otimes\omega^\vee\hookrightarrow (\Ocal\oplus\omega^\vee)\otimes(\Ocal\oplus\omega^\vee)\otimes\cal E\xra{1\otimes (\iota,\theta\otimes\omega^\vee)}(\Ocal\oplus\omega^\vee)\otimes\alpha[1].
\]
The map on the left is induced by the diagonal embedding $\omega^\vee\xra{(1,-1)} \omega^\vee\otimes \omega^\vee$ as in Remark~\ref{counitconv}.
Thus the composition is precisely $(-\theta\otimes\omega^\vee,\iota\otimes\omega^\vee)$, and we may conclude.
\end{proof}

Let us consider the map $G_\bbk$ between $\bbk$-points.
Recall (see Section~\ref{modtri}) that as a complex, $\alpha$ is quasi-isomorphic to $K\oplus J[-1]$, where $K=\Ker\alpha$, $J=\Coker \alpha$.
Thus, we can express $\xi'$ as a sum:
\[
\xi'=\xi'_e+\xi'_h,\qquad \xi'_e:\pi^*(\cal E\otimes\omega^\vee)\ra\pi^*K(D)[1],\qquad\xi'_h:\pi^*(\cal E\otimes\omega^\vee)\ra\pi^*J(D).
\]
Let $M\in\Ext^1(\pi^*(\cal E\otimes\omega^\vee),\pi^*K(D))$ be the extension given by $\xi'_e$.
Then the two-step complex given by $\xi'$ is quasi-isomorphic to $M\ra \pi^*J(D)$, with arrow defined as the composition of $\xi'_h$ with the projection $M\twoheadrightarrow \pi^*(\cal E\otimes\omega^\vee)$.
Consequently, the cone of $\xi'$ has length $1$ if and only if $\xi'_h$ is surjective.

\begin{lm}\label{stabsurj}
The map $\xi'_h$ is surjective if and only if the triple $(\cal E,\alpha,\theta)$ is stable.
\end{lm}
\begin{proof}
By abuse of notation, we will write $\xi=\xi'_h\otimes\omega$ throughout the proof, and study surjectivity of $\xi$.
Thanks to Lemma~\ref{adjuct}, $\xi$ is adjoint to
$$
(\theta,-\iota):\cal E\ra J\otimes(\omega\oplus\Ocal),
$$
where $\theta=\theta_h$ as in Section~\ref{modtri}, and $\iota=\iota_h:\cal E\twoheadrightarrow J$ is the natural projection.
Let $c\in C$, and $s=\pi^{-1}(c)=\bbP(\omega_c\oplus \Ocal_c)$.
If we choose an identification $\omega_c\simeq \Ocal_c$, the stalk $\xi_s$ at the point $s$ is given by a linear combination $a\theta_c+b\iota_c$ for some $a,b\in\bbk$.
In particular, $\xi$ is surjective if and only if it is surjective at each point $s\in S$, that is for every $c\in C$ and $[a:b]\in\bbP^1$ the map $a\theta_c+b\iota_c$ is surjective.
	
Suppose that $\xi$ is not surjective.
Then there exists a point $s\in \pi^{-1}(c)$, $c\in C$ where surjectivity fails.
This means that $J'_c:=\Im(a\theta_c+b\iota_c)$ is a proper subsheaf of $J_c$ for some $a\neq 0$, $b$.
Denote $\cal E'_c=\iota^{-1}(J'_c)$; then $\theta_c(\cal E'_c)\subset J'_c$.
Further, let
$$
\cal E'=\cal E'_c\oplus\bigoplus_{p\in C\setminus\{c\}}\cal E_p,\quad J'=J'_c\oplus\bigoplus_{p\in C\setminus\{c\}}J_p.
$$
Then $\Im\alpha\subset \cal E'$ and $\theta(\cal E')\subset J'$, which precludes the triple $(\cal E,\alpha,\theta)$ from being stable.
	
Conversely, suppose that $(\cal E,\alpha,\theta)$ is destabilized by a subsheaf $\cal E'\subset \cal E$.
Denote $J'=\iota(\cal E')$.
Let us choose a point $c\in C$, such that $\cal E'_c\subset \cal E_c$ is a proper subsheaf.
By assumption $\xi(\cal E'_c)\subset J'_c$ and $\cal E/\cal E'\simeq J/J'$.
Suppose $\xi$ is surjective.
Then for any $s\in\pi^{-1}(c)$ the stalk $\xi_s$ induces an automorphism of $\cal E_c/\cal E'_c$.
In particular, the map $a\Id+b\theta_c$ is an automorphism of $\cal E_c/\cal E'_c$ for each $[a:b]\in \bbP^1$.
However, $\cal E$ is a torsion sheaf, therefore $\cal E_c/\cal E'_c$ is finite-dimensional as a $\bbk$-module.
Because of this, $\theta$ must possess an eigenvalue $\lambda$, so that $\theta-\lambda \Id$ cannot be invertible.
Thus $\xi$ is not surjective.
\end{proof}

Lemma~\ref{stabsurj} shows that any family of Higgs triples which contains a non-stable one is mapped outside of $\underline{\cal Coh}_d^{\leftarrow\cal F}S$ by $G_T$.
This proves that the essential image of $\overleftarrow{\Theta}$ is contained in $\left( T^*\underline{\cal Coh}^{\leftarrow \cal F}_{0,d} \right)^{st}$.
We now need to show that every flat $T$-family $(\cal E_T,\alpha_T,\theta_T)$ of stable Higgs triples lies in the image of $\overleftarrow{\Theta}$, or equivalently its image $G_T(\cal E_T,\alpha_T,\theta_T)$ lies in $\underline{\cal Coh}_d^{\leftarrow\cal F}S(T)$.
By Lemma~\ref{stabsurj}, it is a coherent sheaf $E_T$ on $S\times T$, equipped with an isomorphism $\Psi:E_T|_{D\times T}\xra{\sim}\cal F\boxtimes\Ocal_T$.
Moreover, by construction $E_T$ is a subsheaf of $M_T$, with latter being obtained as an extension of a torsion-free sheaf $\pi^*K_T(D)$ by $\pi^*\cal E_T$.
As a consequence, the torsion $\ona{Tor}E_T$ is contained in the support of $\pi^*\cal E_T$.
However, since $\cal F$ is locally free, the existence of $\Psi$ implies that the support of $\ona{Tor}E_T$ must be disjoint from $D\times T$.
Since the support of every subsheaf of $\pi^*\cal E$ intersects $D\times T$, we conclude that $E$ is torsion-free.

Pick a point $t\in T$.
Outside of the support of $\pi^*\cal E$, the complex $\xi[-1]$ is quasi-isomorphic to $\pi^*\cal F$.
By Lemma~\ref{Mozlem}, it implies that $E_t$ is $f$-stable.

Let us compute the Chern character of $E_t$:
\begin{align*}
\ona{ch}(E_t) & =\ona{ch}(\Ocal(-D)\otimes \pi^*(\cal E_t\otimes\omega^\vee))-\ona{ch}(\pi^*\alpha_t[1])\\
& =(1,-D,1-g)(0,df,0)-(0,df,0)-(n,\deg F\cdot f,0) =(n,\deg F\cdot f,-d).
\end{align*}
Thus $c_1(E_t)=\deg F\cdot f$, and $c_2(E_t)=c_1(E_t)^2/2-\ona{ch}_2(E_t)=d$.

It remains to show that $E_T$ is $T$-flat.
For this, we will express $E_T$ in a different fashion.
In what follows, we will drop the subscript $T$, implicitly assuming that all objects live in families over $T$.

Denote by $ev:H^0(\cal E)\otimes\Ocal\twoheadrightarrow \cal E$ the natural evaluation map.
Let $\widetilde{\cal K}$ be the kernel of
\[
(\alpha,ev):\cal F\oplus \left(H^0(\cal E)\otimes\Ocal\right)\twoheadrightarrow \cal E.
\]
The octahedral axiom applied to the composition $\cal F\xra{(\id,0)}\cal F\oplus H^0(\cal E)\otimes\Ocal\xra{(\alpha,ev)} \cal E$ produces a distinguished triangle
\[
\pi^*\alpha\to \pi^*\widetilde{\cal K}\to H^0(\cal E)\otimes\Ocal_S\xra{+1}.
\]
Next, consider the composition $H^0(\cal E)\otimes\Ocal_S[-1]\to \pi^*\alpha\to E$, where the first map is defined by the triangle above, and the second map comes from the quasi-isomorphism $E\simeq \xi'\otimes(-D)$.
One more application of the octahedral axiom gives rise to the following diagram:
\begin{equation}\label{octaflat}
\begin{tikzcd}
\pi^*\alpha\ar[r]\ar[d]&[-12pt]E\ar[r]\ar[d,hook] &[-12pt]\pi^*(\cal E\otimes\omega^\vee)(-D)\ar[d,equal]\\
\pi^*\widetilde{\cal K}\ar[r,hook,dashed]\ar[d]&\widetilde{E}\ar[r,two heads,dashed]\ar[d,two heads] &\pi^*(\cal E\otimes\omega^\vee)(-D)\\
H^0(\cal E)\otimes\Ocal_S\ar[r,equal]&H^0(\cal E)\otimes\Ocal_S &
\end{tikzcd}
\end{equation}
Here, $\widetilde{E}$ is defined as a cone of the composition above.
Note that since both $E$ and $H^0(\cal E)\otimes\Ocal_S$ are sheaves (as opposed to complexes of sheaves), $\widetilde{E}$ is also a sheaf.

Recall that if $N$ is a $T$-flat sheaf, and $$0\to M_1\to M_2\to N\to 0$$ is a short exact sequence, then $M_1$ is $T$-flat if and only if $M_2$ is.
As a consequence of this, $\pi^*\widetilde{\cal K}$ is $T$-flat as the kernel of $\pi^*(\alpha,ev)$; the middle row of diagram~(\ref{octaflat}) shows $\widetilde{E}$ is $T$-flat; and finally, the middle column of~(\ref{octaflat}) shows that $E$ is $T$-flat as well.
\end{proof}

\begin{proof}[Proof of Theorem~\ref{moduli}]
Representability follows from Theorem~\ref{torfree} together with Proposition~\ref{fstmodrep}.
For smoothness, recall~\cite[Theorem 4.3]{BM} that $\mathscr B(d,\cal F)$ is smooth at a point $(E,\Psi)$ if the kernel of the trace map
\[
\Ext^2(E,E(-D))\to H^2(S,\Ocal(-D))
\]
vanishes.
Since $\Ext^2(E,E(-D))\simeq \Hom(E,E(-D))^*$ by Serre duality, and $H^2(S,\Ocal(-D))\simeq\Hom(\Ocal,\Ocal(-D))^*=0$, it suffices to show that for any $f$-stable sheaf $E$ there exist no non-zero maps from $E$ to $E(-D)$.
Let $\phi:E\to E(-D)$ be such a map.
By definition of $f$-stability, $E|_{\pi^{-1}(c)}\simeq \Ocal_{\bbP^1}^n$ for a generic point $c\in C$.
Since $\Ocal_{\bbP^1}(-1)$ has no global sections, this implies that the image of $\phi$ must be a torsion sheaf.
However, $E(-D)$ is a torsion-free sheaf, so we may conclude.

For the second claim, let $E_1$ be a locally free sheaf of rank $1$ on $S$, such that $c_1(E_1)=0$.
By the seesaw principle, $E_1\simeq \pi^*\pi_*E_1$.
In particular, if $E_1|_D\simeq \Ocal_C$, then $E_1\simeq \Ocal_S$.
By consequence, we have $E^{\vee\vee}\simeq \Ocal_S$ for any $(E,\Psi)\in\mathscr B(d,1)$, and fixing $\Psi$ makes this isomorphism canonical.
Therefore, the map
\[
E\mapsto (E|_{T^*C})\subset E^{\vee\vee}|_{T^*C}=\Ocal_{T^*C}
\]
establishes the desired isomorphism $\mathscr B(d,1)\simeq \ona{Hilb}_dT^*C$.
\end{proof}

In view of Theorem~\ref{moduli}, it is instructive to compare our results with recent works of Negu{\c t}~\cite{Ne2,Ne3}.
For any smooth projective surface $S$ and an ample divisor $H$, he considers the moduli space $\cal M$ of $H$-stable sheaves on $S$ with varying second Chern class, and for every $n\in \bbZ$ defines an operator $e_n:K(\cal M)\ra K(\cal M\times S)$ by Hecke correspondences.
These operators generate a subalgebra $\cal A$ inside $\bigoplus_{k>0}\Hom(K(\cal M)\ra K(\cal M\times S^k))$, which can be then projected to a shuffle algebra $\cal V_{sm}$.
The content of Conjecture 3.20 in~\cite{Ne2} is that this projection is supposed to be an isomorphism.
This conjecture is proved under rather restrictive assumptions; for instance, it is required that $K(S\times S)\simeq K(S)\otimes K(S)$.  

Let us now take $S=T^*C$ together with a scaling action of $T\simeq \bb G_m$, and replace usual $K$-groups with their $T$-equivariant counterpart.
In this case, the algebra $\cal V_{sm}$ can be identified with the subalgebra of $K\bf{Sh}^{norm}_C$, generated by $K(\mathrm{B}\bb G_m)\subset K\bf{Sh}^{norm}_C[1]\simeq K^T(C\times \mathrm{B}\bb G_m)$.
If we further replace $K$-groups by Borel-Moore homology, then by Corollary~\ref{shufflefaith} homological version of $\cal V_{sm}$ is realized as a subalgebra of $H\bf{Ha}_C^{0,T}$.
Therefore, one can regard results of Section~\ref{modtri} as a ``homological non-compact'' version of Negu\c{t}'s conjecture for $S=T^*C$, $c_{1,D}=0$, and stability condition given by $f$.
Another modest gain of our approach is that while $\cal A$ is given by operators on $K$-groups $K(\cal M)$, the definition of $A\bf{Ha}_C^{0,T}$ is independent from its natural representations, which allows to study this algebra without invoking torsion-free sheaves on $T^* C$.

In general, one expects that the moduli of framed sheaves on $\bbP_C(\omega\oplus \Ocal)$ with non-trivial first Chern class can be recovered from the moduli of stable Higgs triples of \emph{positive} rank.
Nevertheless, as stability condition for triples varies, Lemma~\ref{stabsurj} seems to suggest that the objects on $S$ which correspond to stable Higgs sheaves do not have to lie in the usual heart of $\cal D^b(\ona{Coh}S)$.
These questions will be investigated in future work~\cite{Mi2}.

%% file: cobord.tex
\section{Oriented Borel-Moore homology theories}\label{OBM}
In this appendix we recall the notion of equivariant oriented Borel-Moore functor and recollect some of its properties. For a more detailed exposition, we refer the reader to the monograph~\cite{LM} for a treatment of non-equivariant version, and to~\cite{HeLo} for the equivariant case.

\begin{defi}
An \textit{oriented Borel-Moore homology theory} $A$ on $\Sch/\bbk$ (or OBM for short) is the data of:
\begin{enumerate}
	\item for every object in $X\in \Sch/\bbk$, a graded abelian group $A_*(X)$;
	\item for every projective morphism $f:X\ra Y$, a homomorphism $f_*:A_*(X)\ra A_*(Y)$;
	\item for every locally complete intersection (lci for short) morphism $g:X\ra Y$ of relative dimension $d$, a homomorphism $f^*:A_*(Y)\ra A_{*+d}(X)$;
	\item an element $\mathbbm 1\in A_0(pt)$, and for any $X,Y\in \Sch/\bbk$ a bilinear pairing
	\begin{align*}
	\times: A_*(X)\otimes A_*(Y) & \ra A_*(X\times Y),\\
	u\otimes v & \mapsto u\times v,
	\end{align*}
	which is associative, commutative and has $\mathbbm 1$ as unit;
\end{enumerate}
satisfying the following conditions:
\begin{enumerate}
	\item[(BM0)] $A_*(X_1\sqcup X_2)=A_*(X_1)\oplus A_*(X_2)$;
	\item[(BM1)] $\Id^*_X=\Id_{A_*(X)}$, $(f\circ g)^*=g^*\circ f^*$;
	\item[(BM2)] $g^*\circ f_*=f_*'\circ g'^*$ for any cartesian diagram with projective $f$, transversal to lci $g$:
	$$
	\begin{tikzcd}
	W\arrow[r,"g'"]\arrow[d,"f'"] & X\arrow[d,"f"]\\
	Y\arrow[r,"g"] & Z
	\end{tikzcd}
	$$
	\item[(BM3)] $(f\times g)_*(u'\times v')=f_*(u')\times g_*(v')$, $(f\times g)^*(u\times v)=f^*(u)\times g^*(v)$;
	\item[(PB)] let $E\ra X$ be a vector bundle of rank $n+1$, $q:\bbP E\ra X$ its projectivization, $\Ocal(1)\ra \bbP E$ tautological line bundle, and $s:\bbP E\ra \Ocal(1)$ its zero section. Then the map
	$$
	\sum_{i=0}^n (s^*\circ s_*)^i\circ q^*:\bigoplus_{i=0}^n A_{*+i-n}(X)\ra A_*(\bbP E)
	$$
	is an isomorphism;
	\item[(EH)] let $p:V\ra X$ be an affine fibration of rank $n$. Then $p^*:A_*(X)\ra A_{*+n}(V)$ is an isomorphism; 
	\item[(CD)] for $r,N>0$, let $W=(\bbP^N)^r$, with $p_i:W\ra \bbP^N$ being $i$-th projection. Let also $[X_0:\cdots:X_N]$ be homogeneous coordinates on $\bbP^N$, $n_1,\ldots,n_r\in\bbZ$, and $i:E\ra W$ the subscheme defined by $\prod_{i=1}^{r}p_i^*(X_N)^{n_i}=0$. Then $i_*:A_*(E)\ra A_*(W)$ is injective. 
\end{enumerate}
\end{defi}

Given an OBM $A$, we can further define:
\begin{itemize}
	\item for any smooth variety $Y$, set $A^*(Y):=A_{\dim Y-*}(Y)$. The map $\Delta_Y^*\circ\times: A^*(Y)\otimes A^*(Y)\ra A^*(Y)$ defines an associative product on $A^*(Y)$;
	\item for any $f:Z\ra Y$, we have the graph morphism $\Gamma_f=(f,\Id_Z):Z\ra Y\times Z$, which is always a regular embedding. The map $\Gamma_f^*\circ\times: A^*(Y)\otimes A_*(Z)\ra A_*(Z)$ defines a $A^{-*}(Y)$-module structure on $A_*(Z)$;
	\item for any lci morphism $f:Y\ra X$ and arbitrary morphism $g:Z\ra X$, we have a \textit{Gysin pullback} map
	$$
	f^!=f^!_g:A_*(Z)\ra A_*(Z\times_X Y).
	$$
	It coincides with the usual pullback for $Z=X$, $g=\Id_X$. We will also liberally replace $f$ by its pullback $f':Z\times_X Y\ra Z$ in the notations;
	\item for any line bundle $L\ra X$, denote by $s:X\ra L$ its zero section. We have a graded homomorphism
	$$
	c_1(L):=s^*\circ s_*:A_*(X)\ra A_{*-1}(X).
	$$
	Moreover, let us consider any vector bundle $E\ra X$ of rank $n$ together with its projectivization $q:\bbP E\ra X$ and the tautological line bundle $\Ocal(1)\ra \bbP E$. Then there exist unique homomorphisms $c_i(E):A_*(X)\ra A_{*-i}(X)$ for $i=\{0,\ldots,n\}$, called \emph{$i$-th Chern classes}, such that $c_0(E)=1$ and
	$$
	\sum_{i=0}^n(-1)^i c_1(\Ocal(1))^{n-i}\circ q^*\circ c_i(E)=0.
	$$
	They satisfy all the usual properties of Chern classes (see~\cite[Proposition 4.1.15]{LM}). For a smooth variety $X$, the classes $c_i(E)$ can be realized by elements of $A^i(X)$;
	\item there exists a formal group law $F_A\in (u+v)+uvA^*(pt)\llbracket u,v\rrbracket$ on $A^*(pt)$ such that for any $X\in \Sch_\bbk$ and two line bundles $L, M$ on $X$, we have
	$$
	c_1(L\otimes M)=F_A(c_1(L),c_1(M)).
	$$
	Since we never consider two OBMs at the same time, we will use multiplicative notation for group laws, and write $F_A(c_1(L),c_1(M))=c_1(L)\star c_1(M)$.
\end{itemize}

In~\cite{LM}, Levine and Morel define and study algebraic cobordism theory $\Sigma_*$ associated to the universal formal group law $(\bb L,F_{\bb L})$ on the Lazard ring $\bb L$. Let us call an OBM $A$ \emph{free} if the natural map $\Sigma_*\otimes_{\bb L_*}A^*(pt)\ra A_*$ is an isomorphism. For this class of OBMs many properties will follow immediately after establishing them for $\Sigma_*$.

\begin{exe}
Chow group functor $CH_*$ and the Grothendieck group of coherent sheaves $K_0$ are free OBMs. 
\end{exe}

Note that usual Borel-Moore homology is \emph{not} an OBM, because of the presence of odd-dimensional part. Moreover, even the even-dimensional part fails to be a free OBM, which prevents us from translating results found in~\cite{LM} in a straightforward way. Still, all of the results we need can be proved in a similar way for the usual Borel-Moore homology. We will thus abuse the language somewhat and allude to it as to a free OBM in the propositions below, giving separate proofs where needed; in the case of omitted proof, we will give a separate reference.

For any reductive group $G$, free OBM $A$, and a $G$-variety $X$ Heller and Malag\'on-L\'opez~\cite{HeLo} define equivariant homology groups $A_*^G(X)$. Roughly speaking, the group $G$ has a classifying space represented by a projective system $\{EG_N\}_{N \in\bbN}$ of $G$-varieties, and we set
$$
A_*^G(X)=\varprojlim_N A_*(X\times_G EG_N).
$$ 
For example, if $G=GL_d$, the varieties $EG_N$ are just the Grassmanians $\Gr^d(d,N)$. Most of the constructions mentioned above for ordinary OBMs can be extended to the equivariant ones.

\begin{exe}\label{Aofpoint}
If $G=T$ is an algebraic torus of rank $d$, then by Lemma 1.3 in~\cite{YZ} we have
$$
A_*^T(pt)=A_*(pt)\llbracket c_1(t_1),\ldots,c_1(t_d) \rrbracket,
$$
for some choice of basis $t_1,\ldots,t_d$ of the character lattice of $T$.
\end{exe}

\begin{rmq}\label{Kthdiff}
One can observe that in the case of algebraic $K$-theory we get $K^T(pt)=\bbZ[1-t_1^{-1},\ldots,1-t_d^{-1}]$, which is different from the usual ring of Laurent polynomials $\bbZ\llbracket t_1^{\pm 1},\ldots,t_d^{\pm 1}\rrbracket$. However, the two become isomorphic after passing to completion. One can prove that this happens for any $T$-scheme $X$, using the argument in~\cite[Lemma 3.1]{AS} (the author would like to thank Gufang Zhao for pointing out this article). 
\end{rmq}

From now on until the end of appendix, let us fix a free OBM $A$. Moreover, since we are not concerned with questions of integrality, we also \textit{assume that $A^*(pt)$ contains $\bbQ$}, so that all $A$-groups are $\bbQ$-vector spaces. We will often omit homological grading from notations, and write $A=A_*$, $A^G=A_*^G$, $A_G=A^*_G$.

\begin{rmq}
To the best of author's knowledge, the notion of oriented Borel-Moore homology theory is not yet fully developed for arbitrary algebraic stacks. However, since all the stacks of interest in our paper are quotient stacks, we usually slightly abuse the notation and write
$$
A([X/G]):=A^G(X)
$$
for any quotient stack $[X/G]$ (see also~\cite[Proposition 27]{HeLo}).
\end{rmq}

\begin{prop}[{\cite[Theorem 26]{HeLo}}]\label{indA}
Let $H\subset G$ be a closed subgroup, and $X$ a $G$-variety. Then there exists a natural isomorphism
$$
\ind_H^G:A^H(X)\xra{\sim} A^G((X\times G)/H),
$$
where $H$ acts on $X\times G$ diagonally. 
\end{prop}

\begin{prop}[{\cite[Theorem 33]{HeLo}}]\label{GtoT}
Let $G$ be a reductive simply connected algebraic group, $T\subset G$ a maximal torus with normalizer $N$, $W=N/T$ the Weyl group, and $X$ a $G$-variety. Then $W$ acts on $A^T(X)$, and we have a natural isomorphism
$$
A^G(X)\simeq A^T(X)^W.
$$
\end{prop}

\begin{prop}
Let $Z$ be a non-reduced $G$-scheme, and denote by $Z^{red}$ its reduction. Then the pushforward map along the natural embedding
$$
A^G(Z^{red})\ra A^G(Z)
$$
is an isomorphism.
\end{prop}
\begin{proof}
Follows from the definition of algebraic cobordism theory. For an explicit mention of this fact, see the proof of Proposition 3.4.1 in~\cite{LM}.
\end{proof}

\begin{prop}[Projection formula]\label{projf}
Let $f:X\ra Y$ be a $G$-equivariant projective morphism of smooth varieties, $\beta\in A^G_*(X)$, and $\alpha \in A^*_G(Y)$. We have the following identity:
$$f_*(f^*\alpha\cdot \beta)=\alpha\cdot f_*\beta.$$
\end{prop}
\begin{proof}
See~\cite[Proposition 5.2.1]{LM} for non-equivariant version; equivariant proof is completely analogous. 
\end{proof}

The following proposition holds only for universal OBMs, usual Borel-Moore homology not included.

\begin{prop}[{\cite[Theorem 20]{HeLo}}]\label{lespair}
Let $i:Z\ra X$ be a closed equivariant embedding of $G$-varieties, and $j:U\ra X$ the complementary open embedding. Then the sequence
$$
A^G(Z)\xra{i_*}A^G(X)\xra{j^*}A^G(U)\ra 0
$$
is exact.
\end{prop}

For a $G$-equivariant vector bundle $E\ra X$, we define its \textit{Euler class} $e(E)$ as the top Chern class:
$$
e(E):=c_{\rk E}(E)\in A^G(X).
$$
Also, for any regular embedding of smooth varieties $M\subset N$, let $T_M N$ be the normal bundle of $M$ in $N$.

\begin{prop}[Self-intersection formula]\label{euler}
Let $i:N\hookrightarrow M$ be a $G$-equivariant regular embedding of smooth $G$-varieties. Then
$$
i^*i_*(c)=e(T_N M)\cdot c
$$
for any $c\in A^G(N)$.
\end{prop}
\begin{proof}
Follows from~\cite[Theorem 6.6.9]{LM}. For $A=H$, see~\cite[Corollary 2.6.44]{CG}.
\end{proof}
Given $G$-equivariant regular embeddings $j:P\hookrightarrow N$, $i:N\hookrightarrow M$, Whitney product formula applied to the short exact sequence
$$
0\ra T_P N \ra T_P M \ra j^*T_N M \ra 0
$$
tells us that
\begin{equation}\label{pullEu}
e(T_P M)=e(T_P N)\cdot j^*e(T_N M).
\end{equation}


For our purposes, one of the most important pieces of data coming from an OBM is the Gysin pullback. Let us state several compatibility results about it.
\begin{lm}\label{Gysin}
The following properties of Gysin pullback are verified:
\begin{enumerate}
\item Gysin pullback commutes with composition, that is for any diagram with cartesian squares
$$
\begin{tikzcd}
Z'\arrow[r]\arrow[d] & Y'\arrow[r]\arrow[d,"g'"] & X'\arrow[d,"g"]\\
Z\arrow[r,"f_2"] & Y\arrow[r,"f_1"] & X
\end{tikzcd}
$$
one has $(f_1\circ f_2)_g^!=(f_2)_{g'}^!\circ (f_1)_g^!$, provided that $f_1$ and $f_2$ are locally complete intersections;
\item let $F:Y\ra X$, $G:X'\ra X$, $\iota:Z\ra X$ be morphisms of schemes such that $F$ is lci, $G$ and $\iota$ are proper, and $F$ and $G$ are transversal. Consider the following diagram, where all squares are cartesian:
$$
\begin{tikzcd}
Y'\arrow[rrr]\arrow[ddd] &  &  & X'\arrow[ddd,"G"]\\
 & W'\arrow[r,"f'"]\arrow[d,"g'"]\arrow[ul] & Z'\arrow[d,"g"]\arrow[ur,"\iota'"] & \\
 & W\arrow[r,"f"]\arrow[dl] & Z\arrow[dr,"\iota"] & \\
Y\arrow[rrr,"F"] &  &  & X
\end{tikzcd}
$$
Then we have an equality $$g'_*\circ f'^!_{\iota'}=f_\iota^!\circ g_*.$$
\end{enumerate}
\end{lm}
\begin{proof}
See~\cite[Theorem 6.6.6(3)]{LM} for (1) and~\cite[Lemma 1.14]{YZ} for (2).
\end{proof}

\begin{prop}\label{partloc}
Let $i:Y\hookrightarrow X$ be a closed embedding of $T$-varieties, and $\{\chi_1,\ldots,\chi_k\}\subset T^\vee$ a finite set of characters. Suppose that $X^T$ is not empty, $X^T\subset Y$, and for any point $x\in X\setminus Y$ its stabilizer under the action of $T$ is contained in $\bigcup_{i=1}^k\Ker(\chi_i)$. Then the pushforward along $i$ induces an isomorphism
$$
i_*:A_*^T(Y)[c_1(\chi_1)^{-1},\ldots,c_1(\chi_k)^{-1}]\xra{\sim}A_*^T(X)[c_1(\chi_1)^{-1},\ldots,c_1(\chi_k)^{-1}].
$$
\end{prop}
\begin{proof}
In the interest of brevity, we will abuse the notation and write $\chi_i$ instead of $c_1(\chi_i)$. Let us start with surjectivity. By Proposition~\ref{lespair} we have an exact sequence
$$
A^T(Y)\xra{i_*}A^T(X)\ra A^T(X\setminus Y)\ra 0.
$$ 
Thus it suffices to prove that $A_*^T(X\setminus Y)[\chi_1^{-1},\ldots,\chi_k^{-1}]=0$. By Lemma 2 in~\cite{EG}, there exists an open subvariety $U\subset X\setminus Y$ and a subgroup $T_1\subset T$ such that $U\simeq \tilde{U}\times T/T_1$ as $T$-variety, where $\tilde{U}$ is equipped with a trivial action of $T$. In particular, $A_*^T(U)\simeq A_*(\tilde{U})\otimes_{A_*(pt)}A_*^T(T/T_1)$. Because of our hypotheses, one has $T_1\subset \Ker(\chi_i)$ for some $i$, and thus $\chi_iA_*^T(U)=0$. We conclude by Noetherian induction. Namely, let $Z$ be the complement of $U$ in $X\setminus Y$. We have the following exact sequence:
$$
A^T(Z)\xra{i_*}A^T(X\setminus Y)\ra A^T(U).
$$
By induction $pA^T(Z)=0$, where $p$ is a monomial in $\chi_1,\ldots \chi_k$. Therefore $A^T(X)$ is annihilated by $\chi_ip$, and thus $A^T(X\setminus Y)[\chi_1^{-1},\ldots,\chi_k^{-1}]=0$.

It is left to prove injectivity. If $A=H$, we may already conclude by invoking long exact sequence in homology. Otherwise, we follow an approach found in~\cite[2.3, Corollary 2]{Br}. First, let us denote
$$
Y_{\natural}=\overline{\left\{y\in Y:Stab(y)\not\in\bigcup_{i=1}^k\Ker(\chi_i)\right\}}.
$$
Note that $Y_\natural$ is non-empty, since $X^T\subset Y_\natural$. We have the following commutative triangle
$$
\begin{tikzcd}
A_*^T(Y_\natural)[\chi_1^{-1},\ldots,\chi_k^{-1}]\arrow[rr]\arrow[dr,two heads] & & A_*^T(X)[\chi_1^{-1},\ldots,\chi_k^{-1}] \\
& A_*^T(Y)[\chi_1^{-1},\ldots,\chi_k^{-1}]\arrow[ur,two heads] & 
\end{tikzcd}
$$
where the diagonal arrows are surjective by the first part of the proof. If the horizontal arrow is injective, the same can be said of the map $A_*^T(Y)[\chi_1^{-1},\ldots,\chi_k^{-1}]\ra A_*^T(X)[\chi_1^{-1},\ldots,\chi_k^{-1}]$. Thus, from now on we will assume that $Y=Y_\natural$.

Let $U\simeq \tilde{U}\times T/T_1\subset X\setminus Y$ be as in the proof of surjectivity. Since $T_1\subset \Ker(\chi_i)$ for some $i$, we get a regular function 
\begin{align*}
f:\tilde{U}\times T/T_1 \ra \bbk^*,\qquad (u,t) \mapsto \chi_i(t).
\end{align*}
It extends to a rational function $f:X\ra \bbP^1_\bbk$ with the property that $f(t.x)=\chi_i(t)f(x)$ for any $x\in X$ and $t\in T$. In particular, if $y\in Y$ and $t\in\Ker(\chi_i)\setminus Stab_T(y)$ this equality becomes $\chi_i(t)f(y)=f(y)$, so that $y$ belongs to the support of the divisor associated to $f$. Let us denote this support by $D$. Thus $Y\subset D$, and the map $A_*^T(D)\ra A_*^T(X)$ becomes injective after inverting $\chi_i$; see~\cite{Br} for details. We conclude by Noetherian induction on $D$.
\end{proof}

For any commutative ring $R$ and an $R$-module $M$, let $M_{loc}$ be the localized $\Frac(R)$-module $\Frac(R)\otimes_R M$. 

\begin{thm}[Localization theorem]\label{loc}
Let $T$ be an algebraic torus, $R=A_*^T(pt)$, $X$ a $T$-variety, and $i_T:X^T\ra X$ inclusion of the fixed point set. Suppose that $X^T$ is not empty. Then the $\Frac(R)$-linear map
$$i_{T*}:A_*^T(X^T)_{loc}\ra A_*^T(X)_{loc}$$
is an isomorphism. Moreover, if $X$ is smooth, then the map
$$i^*_T:A_*^T(X)_{loc}\ra A_*^T(X^T)_{loc}$$
is an isomorphism as well.
\end{thm}
\begin{proof}
Any action of an algebraic torus has finitely many distinct stabilizers. One can therefore assume that
$$
\bigcup_{x\in X\setminus X^T}Stab_T(x)\subset \bigcup_{i=1}^k\Ker(\chi_i)
$$
for some choice of characters $\chi_1,\ldots,\chi_k\in T^\vee$. Applying Proposition~\ref{partloc} to the embedding $X^T\hookrightarrow X$ and localizing Chern classes of all characters instead of chosen ones, we see that $i_{T*}$ becomes an isomorphism after localization. In view of Proposition~\ref{euler}, it remains to prove that if $X$ is smooth, then the Euler class $e(T_{X^T}X)$ is not a zero-divisor in $A_*^T(X^T)$. Since the $T$-action on $X^T$ is trivial, we can decompose $T_{X^T}X$ into isotypical components:
$$
T_{X^T}X=\bigoplus_i p_i\otimes E_i,
$$
where $p_i$ are non-trivial characters of $T$, and $E_i$ are vector bundles on $X^T$. It suffices to assume $T_{X^T}X=p_1\otimes E_1$, because Euler class is multiplicative with respect to direct sums. In the case when $E_1$ is a line bundle, we have
$$
e(p_1\otimes E_1)=c_1(p_1)\star c_1(E_1)=c_1(p_1)+c_1(E_1)(1+c_1(p_1+\cdots)).
$$
The class $c_1(p_1)$ is not a zero-divisor by Example~\ref{Aofpoint}, $c_1(E_1)$ is nilpotent by~\cite[Remark 5.2.9]{LM}, therefore $e(p_1\otimes E_1)$ is not a zero-divisor as well. Finally, the case when rank of $E_1$ is bigger than $1$ can be reduced to the former by using axiom (PB) and the usual technique of Chern roots.
\end{proof}

The following proposition describes the behavior of localization map with respect to pullbacks and pushforwards.
\begin{prop}\label{pullpush}
Let $f:X\ra Y$ be a morphism of smooth $T$-varieties. Assume that the fixed point sets $X^T$, $Y^T$ are non-empty, and consider the natural commutative diagram
$$
\begin{tikzcd}
X\arrow[r,"f"]& Y\\
X^T\arrow[r,"f_T"]\arrow[u,hook,"i_X"] & Y^T\arrow[u,hook,"i_Y"']
\end{tikzcd}
$$
\begin{enumerate}
\item if $f$ is lci, then $i_X^*\circ f^*=f_T^*\circ i_Y^*$;
\item if $f$ is projective, then the following diagram commutes:
$$
\begin{tikzcd}
A^T(X)_{loc}\arrow[r,"f_*"]\arrow[d,"e(T^*_{X_T}X)^{-1}\otimes i_X^*(-)"']& A^T(Y)_{loc}\arrow[d,"e(T^*_{Y_T}Y)^{-1}\otimes i_Y^*(-)"]\\
A^T(X^T)_{loc}\arrow[r,"f_{T*}"] & A^T(Y^T)_{loc}
\end{tikzcd}
$$
\end{enumerate}
\end{prop}
\begin{proof}
First claim is obvious, and Proposition~\ref{euler} coupled with the localization theorem proves the second claim as well. 
\end{proof}

\begin{corr}\label{pullcomm}
Under conditions of Proposition~\ref{pullpush}(2), we have
$$
i^*_Y\circ f_*(c)=f_{T*}\left(f_T^*(e(T_{Y_T}Y))\cdot e(T_{X_T}X)^{-1}\cdot i^*_X(c)\right)
$$
for any $c\in A^T(X)_{loc}$. If, moreover, $f$ is a regular embedding such that $X^T=Y^T$, then
$$
i^*_Y\circ f_*(c)=i^*_X(e(T_X Y)c).
$$
\end{corr}

For the following proposition we fix a reductive group $G$, and let $T\subset H\subset P$ be a maximal torus, Levi and parabolic subgroup of $G$ respectively. Denote by $W$ and $W_H$ the Weyl groups of $T$ in $G$ and $H$ respectively; we also fix a representative $\sigma$ for each class in $W/W_H$.

\begin{prop}\label{ind}
Let $X$ be an $H$-variety, and denote $Y=G\times_P X$, where the action of $P$ on $X$ is given by the natural projection $P\ra H$. Then $Y^T=W\times_{W_H}X^{T}$, and we have a commutative diagram
$$
\begin{tikzcd}
A^H(X)\arrow[r,"\ind_H^G"]\arrow[d,"i_X^*"']& A^G(Y)\arrow[d,"i_Y^*"]\\
A^T(X^T)\arrow[r,"s^*"] & A^T(W\times_{W_H}X^T)
\end{tikzcd}
$$
where $s:W\times_{W_H}X^T\ra X^T$ is the projection associated to the choice of representatives $\sigma$.
\end{prop}
\begin{rmq}\label{WeylonT}
For an arbitrary $H$-variety $X$, the action of normalizer $N_H(T)$ can be restricted to $X^T$, and thus induces an action of Weyl group $W_H$ on $A^T(X^T)$. Moreover, the restriction map $A^T(X)\ra A^T(X^T)$ can be seen to be $W_H$-equivariant. With that in mind, note that even though $s^*$ depends on the choice of representatives $\sigma$, its restriction to the $W_H$-equivariant part $A^{T}(X^{T})^{W_H}$, which contains the image of $i_X^*$, does not.
\end{rmq}
\begin{proof}
First, let us compute $T$-fixed points of $Y$. Let $(g,x)$ be a point in $G\times X$. Then we have:
\begin{align*}
t.(g,x)=(g,x)\ona{mod} P\,\,\forall t\in T &\Leftrightarrow \forall t\in T\,\exists p\in P :tg=gp^{-1}, p.x=x\\
&\Leftrightarrow g^{-1}Tg\subset P, x\in X^{g^{-1}Tg}\\
&\Leftrightarrow g\in N_G(T)\cdot P, x\in X^{g^{-1}Tg}\\
&\Leftrightarrow \exists p'\in P: gp'^{-1}\in N_G(T), x\in p'.X^{T}
\end{align*}
Therefore $(g,x)\ona{mod} P$ is $T$-stable iff $(g,x)\in (N_G(T)P)\times (p'.X^{T})/P=W\times_{W_H}X^{T}$, which proves the first claim.
Next, let $i_\xi:X\hookrightarrow Y$ be the inclusion of the fiber over $\xi\in G/P$. Note that by definition of $\ind_H^G$, it is a right inverse to $i_e^*$. Therefore, $G$-equivariance implies that
$$
i_{gP}^*(\ind_H^G c)\simeq g.c\in A^{H^g}(X)\text{ for all $g\in G$.}
$$
If we restrict all our structure groups to $T$ and suppose that $g\in N(T)$, we get
$$
i_\xi^*(\ind_H^G c)\simeq g.c\in A^T(X)\text{ for all $\xi\in G/P$}
$$
with the action of $g$ on $A^T(X)$ is as in Proposition~\ref{GtoT}. Moreover, we have the following commutative square
$$
\begin{tikzcd}
A^T(X)\arrow{r}{g}\arrow{d}{i_X^*} & A^T(X)\arrow{d}{i_X^*}\\
A^T(X^T) \arrow{r}{g}& A^T(X^T)
\end{tikzcd}
$$
since the action of $g$ on $A^T(X^T)$ is just the restriction of the action above. Finally,
\begin{align*}
i_Y^*(\ind_H^G c)& = i_X^*\left(\sum_{\xi\in N_G(T)/N_H(T)}i_{\xi P}^*\right)(\ind_H^G c)=i_X^*\left(\sum_{\xi\in N_G(T)/N_H(T)}\xi.c\right)\\
& = \sum_{\sigma\in W/W_H}\sigma.i_X^*(c)=s^*i_X^*(c)
\end{align*}
for all $c\in A^H(X)$, and the second claim follows.
\end{proof}
\begin{rmq}
The same proof as above shows that $(G\times_P X)^T=W\times_{W_H}X^{T}$ for a $P$-variety $X$.
\end{rmq}